\documentclass[11pt]{amsart}

\usepackage{amsmath}
\usepackage{amsthm}
\usepackage{amsfonts}
\usepackage{amssymb}
\usepackage{mathtools}
\usepackage{mathrsfs}
\usepackage{tikz}
\usepackage{tikz-cd}
\usepackage{enumerate}
\usepackage[T1]{fontenc}
\usepackage[utf8]{inputenc}
\usepackage{hyperref}

\theoremstyle{plain}
\newtheorem{theorem}{Theorem}
\newtheorem{proposition}[theorem]{Proposition}
\newtheorem{lemma}[theorem]{Lemma}
\newtheorem{corollary}[theorem]{Corollary}
\newtheorem{question}[theorem]{Question}
\newtheorem{maintheorem}{Theorem}

\theoremstyle{definition}
\newtheorem{definition}[theorem]{Definition}

\theoremstyle{remark}
\newtheorem{remark}[theorem]{Remark}
\newtheorem*{remark*}{Remark}
\newtheorem{introremark}{Remark}[maintheorem] 

\newcommand{\A}{\mathbb{A}}

\newcommand{\F}{\mathbb{F}}

\renewcommand{\P}{\mathbb{P}}
\newcommand{\Q}{\mathbb{Q}}
\newcommand{\R}{\mathbb{R}}

\newcommand{\Z}{\mathbb{Z}}

\newcommand{\scg}{\mathscr{G}}

\newcommand{\sco}{\mathscr{O}}

\newcommand{\sct}{\mathscr{T}}

\newcommand{\cala}{\mathcal{A}}
\newcommand{\calc}{\mathcal{C}}

\newcommand{\calK}{\mathcal{K}}
\newcommand{\calj}{\mathcal{J}}
\newcommand{\calm}{\mathcal{M}}
\newcommand{\calo}{\mathcal{O}}
\newcommand{\calp}{\mathcal{P}}

\newcommand{\mfa}{\mathfrak{a}}

\newcommand{\mfp}{\mathfrak{p}}

\newcommand{\mfm}{\mathfrak{m}}

\newcommand{\Aut}{{\mathrm{Aut}}}

\newcommand{\Ker}{{\mathrm{Ker}}}

\newcommand{\Div}{{\mathrm{Div}}}

\newcommand{\Per}{\mathrm{Per}}
\newcommand{\PrePer}{\mathrm{PrePer}}
\newcommand{\rank}{\operatorname{rank}}
\newcommand{\tors}{\mathrm{tors}}
\newcommand{\divisor}{\operatorname{div}}
\newcommand{\Cl}{\operatorname{Cl}}
\newcommand{\rk}{\mathrm{rk}}

\newcommand{\UML}{\mathrm{UML}}

\newcommand{\ol}{\overline}

\newcommand{\longhookrightarrow}{%
	\, \xhookrightarrow{\quad} \,}
\newcommand{\lra}{{\, \longrightarrow \,}}

\newcommand{\paren}[1]{%
	\mathopen{}\left(#1\right)\mathclose{}}
\newcommand{\set}[1]{%
	\mathopen{}\left\{#1\right\}\mathclose{}}
\newcommand{\sbrac}[1]{%
	\mathopen{}\left[#1\right]\mathclose{}}
\newcommand{\abrac}[1]{%
	\mathopen{}\left\langle#1\right\rangle\mathclose{}}
\newcommand{\verts}[1]{%
	\mathopen{}\left\lvert#1\right\rvert\mathclose{}}

\newcommand\restr[2]{{%
	\left.\kern-\nulldelimiterspace
	#1
	\right|_{#2}
}}

\newcommand{\Mid}{\, \middle|\,}

\newcommand{\sm}{-}

\makeatletter
\AtBeginDocument
	{
	\def\@thm#1#2#3{%
		\ifhmode
		\unskip\unskip\par
		\fi
		\normalfont
		\trivlist
		\let\thmheadnl\relax
		\let\thm@swap\@gobble
		\let\thm@indent\indent 
		\thm@headfont{\scshape}
		\thm@notefont{\fontseries\mddefault\upshape}%
		\thm@headpunct{.}
		\thm@headsep 5\p@ plus\p@ minus\p@\relax
		\thm@space@setup
		#1
		\@topsep \thm@preskip
		\@topsepadd \thm@postskip
		\def\dth@counter{#2}%
		\ifx\@empty\dth@counter
		\def\@tempa{%
			\@oparg{\@begintheorem{#3}{}}[]%
		}%
		\else
		\H@refstepcounter{#2}%
		\hyper@makecurrent{#2}%
		\let\Hy@dth@currentHref\@currentHref
		\AddToHookNext{para/begin}{%
			\MakeLinkTarget*{\Hy@dth@currentHref}}%
		\def\@tempa{%
			\@oparg{%
			\@begintheorem{#3}{%
				\csname the#2\endcsname}}[]%
		}%
		\fi
		\@tempa
	}%
	\dth@everypar={%
	\@minipagefalse \global\@newlistfalse
	\@noparitemfalse
	\if@inlabel
		\global\@inlabelfalse
		\begingroup \setbox\z@\lastbox
		\ifvoid\z@ \kern-\itemindent \fi
		\endgroup
		\unhbox\@labels
	\fi
	\if@nobreak \@nobreakfalse \clubpenalty\@M
	\else \clubpenalty\@clubpenalty \everypar{}%
	\fi
	}%
}

\title{Dynamical uniform boundedness for unicritical polynomials}

\author{Robin Zhang}
\address{Institut des Hautes \'Etudes Scientifiques}
\email{rzhang@ihes.fr}

\date{October 4, 2026}

\begin{document}

\begin{abstract}
	We prove the dynamical uniform boundedness conjecture for unicritical polynomials $z^d + c$ over every number field $K$ when $d \geq 4$, and over every number field $K$ not containing $\mathbb{Q}(\sqrt{-3})$ when $d = 3$. Furthermore, our unconditional bounds on periodic points are effectively computable. In the remaining cases, we show that the dynamical uniform boundedness conjecture is implied by the boundedness of Mordell--Weil ranks over $K$ for a specific family of elliptic curves when $d = 3$ and for Jacobians of dynatomic modular curves when $d = 2$.
\end{abstract}

\maketitle

\setcounter{tocdepth}{1}
\tableofcontents


\vspace*{-1em}

\section{Introduction}
\label{sec:intro}

The uniform boundedness conjectures
in arithmetic dynamics
are inspired by
the torsion conjecture for elliptic curves,
which dates back to a classification problem posed
by Levi in 1908 for the finite abelian group
$E(\Q)_{\tors}$ of $\Q$-rational torsion points
on elliptic curves.
The torsion conjecture then evolved into
the uniform boundedness statement that,
for every number field $K$,
there is a constant $B(K)$ such that
$\#E(K)_{\tors} \leq B(K)$
for every elliptic curve $E/K$;
it also has a strong form in which
the bound $B$ depends only
on the degree $[K : \Q]$.
In a landmark paper, Mazur \cite{mazur-torsion}
proved Levi's conjecture over $\Q$.
After Kamienny \cite{kamienny-torsion},
Kamienny--Mazur \cite{kamienny-mazur},
and Abramovich \cite{abramovich}
established strong uniform boundedness
of torsion points on elliptic curves
for $[K : \Q] \leq 14$,
Merel \cite{merel-torsion}
proved the strong uniform boundedness
conjecture in full generality,
resolving a major question
that spanned nearly the entire 20th century.
We note that Merel's theorem
can be interpreted as the existence of
an integer $n_D > 0$
such that there are no $K$-rational points on
classical affine modular curves $Y_1(n)$
for all integers $n \geq n_D$
and all number fields $K$ of degree $[K : \Q] \leq D$.

\subsection*{Dynamical uniform boundedness}

Motivated by this history,
Morton--Silverman \cite{morton-silverman-ubc}
formulated the dynamical
uniform boundedness conjectures
for preperiodic and periodic points
of endomorphisms on $\P^N$.
They observed that
uniform boundedness for
preperiodic points of
endomorphisms on $\P^1$
would imply the uniform boundedness of
rational torsion points
on elliptic curves over number fields
(i.e. Merel's theorem),
and Fakhruddin \cite[Corollary~2.4]{fakhruddin-selfmaps}
later showed that
the uniform boundedness conjecture
for preperiodic points
of endomorphisms of $\P^N$
would imply uniform boundedness
for torsion points on abelian varieties.

In this paper, we unconditionally prove the
dynamical uniform boundedness conjecture
for periodic points of
unicritical polynomials
over number fields when $d \geq 4$
and over number fields
not containing $\Q(\sqrt{-3})$
when $d = 3$.
Furthermore, our bound is effective.

\begin{maintheorem}
\label{thm:main}
	Let $K$ be a number field, and let $d \geq 3$.
	If $d = 3$, assume that
	$\Q(\sqrt{-3}) \not\subseteq K$.
	There is an effectively
	computable constant $B_{\Per}(K, d)$ such that
	$z^d + c$ has at most $B_{\Per}(K, d)$
	many $K$-rational periodic points
	for every $c \in K$.
\end{maintheorem}

\begin{introremark}
	Similarly to
	the torsion theorems of Mazur \cite{mazur-torsion}
	and Merel \cite{merel-torsion},
	the uniform bound in
	Theorem~\ref{thm:main}
	has a Diophantine interpretation:
	there exists an effectively computable integer $n_K > 0$
	such that there are no $K$-rational points on
	dynatomic modular curves $Y_{1, d}(n)$
	for all integers $n \geq n_K$.
	For more on dynatomic modular curves,
	see Section~\ref{sec:dynatomic-curves}.
\end{introremark}

\begin{introremark}
	Since $\Q(\sqrt{-3})$ is imaginary quadratic,
	the $d = 3$ case of Theorem~\ref{thm:main}
	includes every totally real field
	and every field of odd degree.
	Section~\ref{sec:quantitative}
	gives explicit formulas for the bound $B_{\Per}(K, d)$
	and explains why they are not uniform
	in the degree $[K : \Q]$.
	\hyperref[app:example-effective-constants]{Appendix 1}
	numerically evaluates the bound $B_{\Per}(K, d)$
	in several examples over $\Q$ and quadratic fields.
	Section~\ref{sec:related-families}
	explains how to mildly extend
	Theorem~\ref{thm:main} to certain families of
	non-unicritical polynomials.
\end{introremark}

By the non-effective passage
from periodic uniform boundedness
to preperiodic uniform boundedness
due to Doyle--Poonen \cite[Theorem~1.8]{doyle-poonen},
Theorem~\ref{thm:main} implies the
dynamical uniform boundedness conjecture
for periodic points of
unicritical polynomials for the same polynomials.

\begin{corollary}
\label{cor:ubc-preperiodic-unconditional}
	Let $K$ be a number field and let $d \geq 3$.
	If $d = 3$, assume that
	$\Q(\sqrt{-3}) \not\subseteq K$.
	There is a constant $B_{\PrePer}(K, d)$ such that
	$z^d + c$ has at most $B_{\PrePer}(K, d)$
	many $K$-rational preperiodic points
	for every $c \in K$.
\end{corollary}

The unicritical polynomials
$f_{d, c}(z) := z^d + c$
are so named because they have
exactly one finite critical point---they form
the most elementary nontrivial family
of endomorphisms on $\P^1$
and a primary testing ground
for arithmetic dynamics.
We recall some of the existing results
towards the dynamical uniform boundedness conjectures.
\begin{itemize}
	\item There are unconditional bounds
		for preperiodic points
		that depend on a finite set $S$ containing the
		places of bad reduction in
		\cite{narkiewicz,morton-silverman-periodic,
		benedetto-global,canci-paladino,troncoso,
		canci-vishkautsan,doyle-hindes-integral,
		yap-small-points}.
		Among these, Benedetto
		\cite[Theorem~7.1]{benedetto-global}
		and Yap
		\cite[Theorem~1.1]{yap-small-points}
		obtain bounds of order
		$\#S \log\paren{\#S}$ for polynomials and
		rational maps on $\P^1$, respectively,
		and Doyle--Hindes
		\cite[Corollary~1.2]{doyle-hindes-integral}
		gives a bound for unicritical polynomials
		that depends only on $[K : \Q]$ and the
		largest residue characteristic of a finite
		place in $S$.
	\item There is an unconditional explicit bound
		$d^{[K : \Q]}$ on periodic points
		of unicritical polynomials
		with good reduction 
		at some place above a
		prime dividing $d$ in
		previous work of the
		author with Rajagopal
		\cite[Theorem~3]{rajagopal-zhang}.
	\item Assuming the $abc$ conjecture
		and its higher-dimensional generalizations,
		there are uniform bounds in
		\cite{looper-2021,panraksa,looper-2025,doyle-hindes}.
		Among these,
		Looper \cite{looper-2021,looper-2025}
		establishes conditional uniform boundedness,
		first for unicritical polynomials
		and then for arbitrary polynomials,
		and Doyle--Hindes \cite{doyle-hindes}
		improves the conditional bound
		for unicritical polynomials
		to be independent of $d$.
	\item Over function fields of characteristic zero,
		Looper--Yap
		\cite[Theorem~1.1]{looper-yap-function-fields}
		give a bound the order of every torsion point on an
		abelian variety with trivial trace in terms of the
		dimension and the gonality of the base curve;
		Gao--Gu \cite[Theorem~1.3]{gao-gu}
		give another proof with an explicit such bound.
		For non-isotrivial
		one-parameter families of rational maps on $\P^1$
		over complex function fields
		and endomorphisms of $\P^N$ with additional hypotheses,
		Ji--Xie
		\cite[Theorems~1.14 and~1.15]{ji-xie-gonality}
		prove uniform boundedness for preperiodic points
		whose fields of definition have bounded gonality.
\end{itemize}

\subsection*{The remaining cases}

Through an unexpected connection
between dynamical torsion and Mordell--Weil ranks,
we can also conditionally treat the remaining
$d = 2, 3$ cases of
the dynamical uniform boundedness conjecture
for unicritical polynomials.
In particular, we use
periodic points of a fixed polynomial
to construct independent rational points
on particular elliptic curves and abelian varieties.
One consequence of our dynamical constructions is that the
uniform boundedness of periodic points of $z^3 + c$
(resp. $z^2 + c$)
is implied by the boundedness of Mordell--Weil ranks
for a particular family of elliptic curves
(resp. the Jacobians of dynatomic modular curves).

Let $\mu_d$ denote the $d$-th roots of unity
in an algebraic closure $\overline{K}$ of $K$.
For distinct
$\zeta, \xi \in \mu_d \sm \set{1}$,
an element $u \in K(\mu_d)^\times$, and
distinct nonzero periodic points $Q, R$
of $f_{d, c}$,
let $E_{Q, R, \zeta, \xi}^{(u)}$ be the smooth projective
curve with affine model
\[
	E_{Q, R, \zeta, \xi}^{(u)}:\quad
	uY^2
		= (X - \zeta Q)(X - \xi Q)
		(X - \zeta R)(X - \xi R).
\]
When $d = 3$, the roots $\zeta$ and $\xi$ are
the two nontrivial third roots of unity, and
the family $E_{Q, R, \zeta, \xi}^{(u)}$ has affine model
\[
	E_{Q, R}^{(u)}:\quad
	uY^2
		= (X^2 + QX + Q^2)(X^2 + RX + R^2).
\]
If $\mu_3 \subset K$
(equiv. if $\Q(\sqrt{-3}) \subseteq K$),
then this family is defined over $K$.
Note that the boundedness of $\rank E_{Q, R}^{(u)}$
is a special case of the folklore conjecture
that asserts the boundedness of ranks of all elliptic curves
and dates back to Honda \cite[p.~98]{honda}
(see also Pasten \cite{pasten-bounded-ranks}
and the heuristics of Park--Poonen--Voight--Wood
\cite[Theorem~1.1.1 and Section~12]{park-poonen-voight-wood}).

\begin{maintheorem}
\label{thm:elliptic-curves-large-rank}
	Let $K$ be a number field, let $d \geq 3$ be odd, and
	let $c \in K$.
	There are constants $a(K, d) > 0$ and
	$b(K, d) \geq 0$ depending only on $K$ and $d$
	such that:
	if $f_{d, c}$ has a $K$-rational cycle
	of length $n \geq 5$,
	then there are at least
	$\lceil(n - 2)/12\rceil$
	geometric isomorphism classes
	of elliptic curves over $K(\mu_d)$
	with full $K(\mu_d)$-rational $2$-torsion
	such that each class contains an
	elliptic curve $E/K(\mu_d)$ in the family
	$E_{Q, R, \zeta, \xi}^{(u)}$ with
	\[
		\rank E\paren{K(\mu_d)}
		\geq a(K, d)\log n - b(K, d).
	\]

	Moreover, suppose that $d = 3$ and
	$\Q(\sqrt{-3}) \subseteq K$.
	If the ranks of $E_{Q, R}^{(u)}$
	over $K$ are uniformly bounded
	as $u, Q, R \in K^\times$ vary
	such that $Q/R \notin \mu_3$,
	then there is a constant $B_{\Per}(K, 3)$
	such that
	$z^3 + c$ has at most $B_{\Per}(K, 3)$
	many $K$-rational periodic points
	for every $c \in K$.
\end{maintheorem}

\begin{introremark}
	\label{rem:elliptic-curves-large-rank-precise}
	The constants $a(K, d)$ and $b(K, d)$
	in Theorem~\ref{thm:elliptic-curves-large-rank}
	are specified in the more precise estimate
	\eqref{eq:odd-rank-lower};
	the precise bound $B_{\Per}(K, 3)$ is stated
	in Corollary~\ref{cor:conditional-d3}.
	We also prove an analogous statement
	for $(d - 2)$-dimensional
	principally polarized Jacobians
	of large Mordell--Weil rank
	in Theorem~\ref{thm:jacobians-large-rank}
\end{introremark}

\begin{introremark}
	\label{rem:elliptic-curves-large-rank-complementary}
	Theorem~\ref{thm:elliptic-curves-large-rank} keeps the field
	$K(\mu_d)$ fixed and varies the elliptic curve.
	In Section~\ref{sec:fixed-elliptic-fields},
	we prove a complementary statement
	for a fixed elliptic curve $E$
	over field extensions $L_n$ generated
	by periodic points of exact period $n$:
	Theorem~\ref{thm:fixed-elliptic-curve-rank}
	gives a lower bound on the rank of the subgroup
	generated by lifts of periodic points of exact period $n$,
	and therefore on $\rank E(L_n)$.
	Proposition~\ref{prop:iterated-splitting-rank}
	gives an explicit example of rank growth
	for the fixed non-CM elliptic curve
	$E : Y^2 = (X^2 + 1)(X + 1 - i)$ over $\Q(i)$;
	the extensions $K_m := \Q(i)(f_{2, i}^{-m}(0))$
	are unramified outside $(1 + i)$,
	and $E$ has the following rank growth
	over $K_m$ for every $m \geq 3$:
	\[
		\rank E(K_m)
		\geq \frac{m\log 2 - \log(256\cdot 10^{13})}
		{2\log\paren{1 + \frac{5}{4\sqrt{2}}}}.
	\]
\end{introremark}

The elliptic curve construction
of Theorem~\ref{thm:elliptic-curves-large-rank}
assumes that $d$ is odd,
and therefore does not apply when $d = 2$.
So we give another construction
that covers the case when $d$ is even,
using the Jacobians of the dynatomic curves
for the unicritical family.
For $d \geq 2$ and $n \geq 1$,
the affine dynatomic curve $Y_{1, d}(n)$
for the family $f_{d, c}(z) = z^d + c$
is the closure of the locus
of pairs $(c, P) \in \A^2$
for which $P$ has exact
period $n$ under $f_{d, c}$;
at special parameters, the closure can also
contain points of smaller exact period
(see \cite[Section~4.1]{silverman-dynamics}
for a standard reference).
Let $X_{1, d}(n)$ be the smooth projective model
of $Y_{1, d}(n)$ and
let $J_{1, d}(n)$ be the Jacobian of
$X_{1, d}(n)$.

\begin{maintheorem}
\label{thm:unicritical-dynatomic-rank}
	Let $K$ be a number field, and let $d \geq 2$.
	There is an integer $N_{K, d} \geq 1$
	such that:
	if $c \in K$ and $f_{d, c}$ has a $K$-rational
	point of exact period $n > N_{K, d}$, then
	\[
		\rank J_{1, d}(n)\paren{K(\mu_d)}
			\geq \max_{\ell^a \mid n}
				\ell^{a - 1}(\ell - 1),
	\]
	where the maximum is over the prime powers
	that divide $n$.

	Furthermore, if the ranks
	$\rank J_{1, d}(n)\paren{K(\mu_d)}$ are uniformly
	bounded as $n$ varies,
	then there is a constant
	$B_{\Per}(K, d)$ such that
	$z^d + c$ has at most $B_{\Per}(K, d)$
	many $K$-rational periodic points
	for every $c \in K$.
\end{maintheorem}

\begin{introremark}
	Proposition~\ref{prop:unicritical-cyclotomic-rank}
	gives a sharper, but more notation-heavy, form of
	Theorem~\ref{thm:unicritical-dynatomic-rank}:
	it is enough to bound the ranks of certain
	cyclotomic factors of $J_{1, d}(n)$ rather than
	the rank of the whole Jacobian.
	Furthermore, if the uniform bound on
	the Mordell--Weil rank is effectively computable,
	then the constant $B_{\Per}(K, d)$
	is also effectively computable.
\end{introremark}

By the non-effective passage
from periodic uniform boundedness to
preperiodic uniform boundedness due to
Doyle--Poonen~\cite[Theorem~1.8]{doyle-poonen},
Theorems~\ref{thm:elliptic-curves-large-rank} and
\ref{thm:unicritical-dynatomic-rank}
gives the following two cases
of the uniform boundedness conjecture
for preperiodic points of unicritical polynomials
not covered by
Corollary~\ref{cor:ubc-preperiodic-unconditional}.

\begin{corollary}
\label{cor:preperiodic}
	Let $K$ be a number field and suppose one of the
	following holds:
	\begin{itemize}
		\item $d = 2$ and the ranks
			of dynatomic Jacobians $J_{1, 2}(n)\paren{K}$
			are uniformly bounded as $n$ varies;
		\item $d = 3$
			and the ranks of elliptic curves
			in the family $E_{Q, R}^{(u)}$
			over $K$ are uniformly bounded
			as $u, Q, R \in K^\times$ vary subject to
			$Q/R \notin \mu_3$.
	\end{itemize}
	There is a constant $B_{\PrePer}(K, d)$ such that
	$z^d + c$ has at most $B_{\PrePer}(K, d)$
	preperiodic points for every $c \in K$.
\end{corollary}

\subsection*{Proof strategy and organization}
First, we note the logical relationships between
our main results:
\begin{itemize}
	\item Theorems~\ref{thm:main},
		\ref{thm:elliptic-curves-large-rank},
		and \ref{thm:unicritical-dynatomic-rank}
		are each proved using different methods
		and are logically independent,
		although the proofs of
		Theorems~\ref{thm:main}
		and \ref{thm:elliptic-curves-large-rank}
		have a common starting point
		in Theorem~\ref{thm:power-classes}.
	\item Corollary~\ref{cor:ubc-preperiodic-unconditional}
		follows from Theorem~\ref{thm:main}
		by \cite[Theorem~1.8]{doyle-poonen}.
	\item Corollary~\ref{cor:preperiodic} follows
		from Theorems~\ref{thm:elliptic-curves-large-rank}
		and \ref{thm:unicritical-dynatomic-rank}
		by \cite[Theorem~1.8]{doyle-poonen}.
\end{itemize}
We describe the proofs for
Theorems~\ref{thm:main},
\ref{thm:elliptic-curves-large-rank},
and \ref{thm:unicritical-dynatomic-rank}
in section order.

The proofs of Theorems~\ref{thm:main} and
\ref{thm:elliptic-curves-large-rank}
begin with the same local descent
in Section~\ref{sec:local}.
The essential first step is
Benedetto's product identity
\cite[Theorems~2 and~3]{benedetto-product-identity}
(Lemma~\ref{lem:telescoping}),
which shows that
for all distinct $P, Q \in \Per_K(f_{d, c})$ and
$\zeta \in \mu_d \sm \set{1}$,
the difference $P - \zeta Q$ is a unit at every finite place
of $K(\mu_d)$ where $c$ is integral.
In Theorem~\ref{thm:power-classes},
we combine this with
congruences of valuations at the finite places
where $c$ is nonintegral
to obtain a fixed finite subgroup $\calc_{K, d}$ of
$K(\mu_d)^\times / (K(\mu_d)^\times)^{d - 1}$
that is independent of $c$
and contains all the classes of
$(P - \zeta Q)/(P - \xi Q)$
in $K(\mu_d)^\times / (K(\mu_d)^\times)^{d - 1}$.

Section~\ref{sec:proof-main} uses the subgroup
$\calc_{K, d}$ to prove Theorem~\ref{thm:main}.
For $d \geq 4$, fix distinct
$\zeta_1, \zeta_2, \zeta_3 \in \mu_d \sm \set{1}$
and a finite set $A \subset K(\mu_d)^\times$ of
representatives of $\calc_{K, d}$.
We show that for distinct periodic points
$P, Q$ with $Q \neq 0$,
the quotient $P/Q$
corresponds to a rational point on
one of the twisted Fermat curves
\[
	C_{a_1, a_2, a_3}: \quad (\zeta_2 - \zeta_3)a_1 X_1^{d - 1}
	+ (\zeta_3 - \zeta_1)a_2 X_2^{d - 1}
	+ (\zeta_1 - \zeta_2)a_3 X_3^{d - 1} = 0.
\]
of genus $(d - 2)(d - 3)/2$, indexed by
$(a_1, a_2, a_3) \in A^3$.
For each fixed nonzero periodic point $Q$,
Proposition~\ref{prop:quotient-curves} gives an injection
\[
	\Per_K(f_{d, c}) \sm \{Q\} \longhookrightarrow
		\coprod_{(a_1, a_2, a_3) \in A^3} C_{a_1, a_2, a_3}(K(\mu_d)).
\]
For $d \geq 5$, an application of
Faltings's theorem on the Mordell conjecture~\cite{faltings-mordell}
and the quantitative version by Yu--Yuan--Zhou~\cite{yu-yuan-zhou}
then bounds each $\#C_{a_1, a_2, a_3}(K(\mu_d))$.
For $d = 4$, the twisted Fermat curves $C_{a_1, a_2, a_3}$
have genus $1$ and may have infinitely many rational points.
So after fixing periodic points $P_0$ and $Q_0$
and observing that the twisted Fermat curves
become isomorphic to $E := C_{1, 1, 1}$
over the fixed field $L := K(\mu_4, \sqrt[3]{a} : a \in A)$,
we instead identify periodic points of $f_{d, c}$
with rational points in the subvariety
$Z := \{(R, S) \in E \times E \mid \phi(R) = (P_0/Q_0) \phi(S)\}$.
For each pair of distinct nonzero periodic points $P_0, Q_0$,
Lemma~\ref{lem:quartic-periodic-pairs} gives
an injection
\[
	\Per_K(f_{4, c}) \sm \set{P_0, Q_0} \longhookrightarrow Z(L).
\]
An application of R\'{e}mond's quantitative Mordell--Lang theorem
\cite[Theorem~2.1]{remond-count} then bounds $\#Z(L)$
independently of $P_0/Q_0$.
For $d = 3$ with $\Q(\sqrt{-3}) \not\subset K$,
we instead choose a finite place $v \nmid 3$ that does
not split in $K(\mu_3)/K$.
In Proposition~\ref{prop:d3},
we then apply Proposition~\ref{prop:tame-distance} which permits
at most one periodic point when $v(c) < 0$,
and Pezda's local period bound~\cite{pezda-local}
when $v(c) \geq 0$.
The section concludes by transferring the result to
several related polynomial families
in Corollary~\ref{cor:semiconjugacy-transfer}.

Section~\ref{sec:rank} proves
Theorem~\ref{thm:elliptic-curves-large-rank}.
Let $d \geq 3$ be odd, fix $c \in K$ and distinct
nonzero $Q, R \in \Per_K(f_{d, c})$, and choose
distinct $\zeta, \xi \in \mu_d \sm \set{1}$.
Theorem~\ref{thm:power-classes}
places every periodic point on a quadratic twist
$E_{Q, R}^{(u)}$ of the double cover of $\P^1$
branched at $\zeta Q$, $\xi Q$, $\zeta R$, and $\xi R$,
with parameters $u \in K(\mu_d)^\times$ ranging over a
fixed finite set depending only on $K, d$.
For the degree-$2$ coordinate maps $x_1, x_2$ on two
such twists, $K(\mu_d)$-rational lifts $A, B$ of
$P$ and $f_{d, c}(P)$ satisfy
$x_2(B) = f_{d, c}(x_1(A))$.
We study this equation on a curve in the product
of the two elliptic curves.
A comparison of pole orders excludes translates of
elliptic curves, so
Gao--Ge--K\"{u}hne~\cite[Theorem~1.1]{gao-ge-kuhne}
bound the rational point counts in terms of the two
Mordell--Weil ranks, uniformly as $Q, R, c$ vary.
Summing over the pairs of twists gives the rank lower
bound for a long cycle; a uniform bound for these ranks
gives Corollary~\ref{cor:conditional-d3}.
With $Q$ fixed, the degree-$12$ map from $R/Q$ to the
$j$-line bounds how many choices of $R$ give the same
geometric isomorphism class.

Section~\ref{sec:fixed-elliptic-fields}
proves complementary results to
Theorem~\ref{thm:elliptic-curves-large-rank}
by keeping the elliptic curve and
its degree-$2$ map to $\P^1$ fixed,
and bounding the rank of the subgroup
generated by lifts of the points
of a given exact period.

Section~\ref{sec:higher-jacobians} first proves
Theorem~\ref{thm:unicritical-dynatomic-rank}.
A rational cycle of length $n$ produces divisor
classes on $J_{1, d}(n)$ that retain the full
cyclic symmetry of the orbit.
This symmetry forces cyclotomic contributions
to the Mordell--Weil rank of $J_{1, d}(n)$
and gives the lower bound.
Conversely, a uniform rank bound restricts the
prime powers that can occur in a rational period $n$.
Hence the possible periods are uniformly bounded,
and the degrees of $f_{d, c}^j(z) - z$ bound the
total number of periodic points.
When $d = 2$, one has $K(\mu_2) = K$, so this
deduction uses only the ranks of $J_{1, 2}(n)(K)$.
The remainder of the section proves
Theorem~\ref{thm:jacobians-large-rank}
by constructing higher-genus Galois covers of
$\P^1$ with cyclic deck group.

Finally, Section~\ref{sec:quantitative}
makes the bounds $B_{\Per}(K, d)$ explicit
and isolates the obstacles to strong
uniformity.
\hyperref[app:example-effective-constants]{Appendix 1}
works out numerical examples
of the bounds.

\subsection*{AI usage}
Some numerical SageMath calculations in
\hyperref[app:example-effective-constants]{Appendix 1}
were checked with assistance from GPT 5.6 Sol.
This paper was also copyedited
with assistance from GPT 5.6 Sol.

\subsection*{Acknowledgments}
I am grateful to Joseph Silverman and Jit Wu Yap
for helpful comments.

I would also like to express my gratitude
to Niccol\`{o} Ronchetti,
who introduced me to the dynamical uniform boundedness
conjectures 12 years ago during the 2014 SURIM summer research
program at Stanford University
and inspired my continuing interest in these questions.


\numberwithin{equation}{section}
\numberwithin{theorem}{section}


\section{Kummer theory of periodic points}
\label{sec:local}
Throughout this section, let $K$ be a number field,
let $d \geq 2$, and set $f_{d, c}(z) := z^d + c$ for
$c \in K$.
Write $\Per_K(f_{d, c})$ for its $K$-rational
periodic points.
The purpose of this section is to
construct a fixed finite subgroup $\calc_{K, d}$ of
$K(\mu_d)^\times / (K(\mu_d)^\times)^{d - 1}$.
Theorem~\ref{thm:power-classes} shows that,
for each fixed pair of
distinct $P, Q \in \Per_K(f_{d, c})$,
the classes $[P - \zeta Q]$ belong to a single coset
$\cala_{P, Q}$ of $\calc_{K, d}$ as
$\zeta \in \mu_d \sm \set{1}$ varies.
In particular, the classes $[(P - \zeta Q)/(P - \xi Q)]$
belong to $\calc_{K, d}$ for distinct
$\zeta, \xi \in \mu_d \sm \set{1}$.

At a finite place $w \nmid d$ of $K(\mu_d)$,
the proof distinguishes whether $c$ is integral.
If $w(c) < 0$, Proposition~\ref{prop:tame-distance} gives
$w(P - \zeta Q) \equiv w(c)/d \pmod{d - 1}$.
If $w(c) \geq 0$, Benedetto's product identity
Benedetto~\cite[Theorems~2 and~3]{benedetto-product-identity}
(Lemma~\ref{lem:telescoping} shows that $P - \zeta Q$
is a unit in the valuation ring of $K(\mu_d)_w$,
so $w(P - \zeta Q) = 0$.
For the fixed pair $P, Q$, these calculations give
the same valuation modulo $d - 1$ for every
$\zeta \in \mu_d \sm \set{1}$.
Section~\ref{sec:kummer} uses these congruences to prove
Theorem~\ref{thm:power-classes}.

For a nonarchimedean local field $k$,
write $v$ for its normalized additive valuation and
\[
	\calK_k(f)
	:= \set{x \in k \Mid
	\set{f^j(x) \Mid j \geq 0} \text{ is bounded}}
\]
for the $k$-rational filled Julia set of
a polynomial $f$.
Thus every periodic point belongs to $\calK_k(f)$;
for example, $\calK_k(z^d) = \calo_k$.

\subsection{Nonintegral parameters}

For $f_{d, c}(z) = z^d + c$ over a number field,
Hutz~\cite[Lemma~7]{hutz} proves that a periodic
point $x$ satisfies $v(c) = d v(x)$ at every
nonarchimedean place with $v(c) < 0$;
Panraksa~\cite[Lemma~1]{panraksa} observes that the
same argument applies to preperiodic points.
The following lemma generalizes preperiodicity
to boundedness of the forward orbit and allows an
arbitrary nonarchimedean local field.

\begin{lemma}
\label{lem:shell}
	Let $d \geq 2$, let $k$ be a nonarchimedean local
	field, let $c \in k$ satisfy $v(c) < 0$, and set
	$f_{d, c}(z) := z^d + c$.
	If the filled Julia set $\calK_k(f_{d, c})$
	is nonempty, then there is an
	integer $s \geq 1$ such that $v(c) = -ds$.
	Furthermore, $v(f_{d, c}^j(x)) = -s$ for every
	$x \in \calK_k(f_{d, c})$ and every $j \geq 0$.
\end{lemma}

\begin{proof}
	Set $r := v(c) < 0$. If $d v(x) < r$, then
	$v(f_{d, c}^j(x)) = d^j v(x)$ for every
	$j \geq 1$, so these valuations tend to $-\infty$.
	If $d v(x) > r$, then $v(f_{d, c}(x)) = r$ and
	$v(f_{d, c}^{j + 1}(x)) = d^j r$
	for every $j \geq 1$. These valuations also tend
	to $-\infty$. Thus boundedness forces
	$d v(x) = r$.

	Every iterate of a bounded point again has bounded
	orbit, so the same equality gives
	$v(f_{d, c}^j(x)) = r/d$ for every
	$j \geq 0$. Since the value group is $\Z$, one has
	$r = -ds$ for some integer $s \geq 1$.
\end{proof}

The next lemma shows that, when the residue
characteristic does not divide $d$, the valuation
shell in Lemma~\ref{lem:shell} can be scaled to the
unit group.

\begin{lemma}
\label{lem:tame-normalization}
	Let $d \geq 2$, let $k$ be a nonarchimedean local
	field whose residue characteristic does not divide
	$d$, let $c \in k$ satisfy $v(c) < 0$, and let
	$f_{d, c}(z) := z^d + c$.
	If the filled Julia set $\calK_k(f_{d, c})$
	is nonempty, then $-c \in (k^\times)^d$.
	Furthermore, for every
	$\beta \in k$ satisfying $\beta^d = -c$, the change
	of variables $z = \beta y$ conjugates $f_{d, c}$ to
	\[
		T_\beta(y) := \beta^{d - 1}(y^d - 1),
	\]
	and $f_{d, c}^j(x)/\beta \in \calo_k^\times$ for every
	$x \in \calK_k(f_{d, c})$ and every $j \geq 0$.
\end{lemma}

\begin{proof}
	Choose $x \in \calK_k(f_{d, c})$ and write
	$v(c) = -ds$. By Lemma~\ref{lem:shell},
	\[
		-\frac{c}{x^d}
		= 1 - \frac{f_{d, c}(x)}{x^d}
		\in 1 + \mfm_k.
	\]
	Since the residue characteristic does not divide
	$d$, the $d$-th-power map is an automorphism of
	$1 + \mfm_k$; see Serre
	\cite[Section~II.5]{serre-local-fields}.
	Thus $-c \in (k^\times)^d$.

	Since $\beta^d = -c$, substitution gives
	\[
		\beta^{-1}f_{d, c}(\beta y)
		= \beta^{d - 1}(y^d - 1)
		= T_\beta(y).
	\]
	Moreover, $v(\beta) = -s$, while
	Lemma~\ref{lem:shell} gives
	$v\paren{f_{d, c}^j(x)} = -s$ for every
	$x \in \calK_k(f_{d, c})$ and every $j \geq 0$.
	Therefore
	$f_{d, c}^j(x)/\beta \in \calo_k^\times$.
\end{proof}

By Lemma~\ref{lem:tame-normalization}, if $\beta \in k$
satisfies $\beta^d = -c$, then $f_{d, c}^j(x)/\beta$
lies in $\calo_k^\times$ for every
$x \in \calK_k(f_{d, c})$ and every $j \geq 0$; we call
the unit $f_{d, c}^j(x)/\beta$
the $j$-th normalized iterate of $x$.
We now compare two points $X, Y \in \calK_k(f_{d, c})$ through
the smallest index $j$ at which their $j$-th normalized
iterates have distinct reductions.

\begin{definition}
\label{def:separation-index}
	Assume the hypotheses of
	Lemma~\ref{lem:tame-normalization}, choose
	$\beta \in k$ with $\beta^d = -c$, and let
	$X, Y \in \calK_k(f_{d, c})$.
	Define
	\[
		\ell_v(X, Y)
			:= \min\set{j \geq 0 \Mid
				\ol{f_{d, c}^j(X)/\beta}
				\neq \ol{f_{d, c}^j(Y)/\beta}},
	\]
	with $\ell_v(X, Y) := \infty$ when no such $j$
	exists.
\end{definition}

Note that the index $\ell_v(X, Y)$
does not depend on the choice of $\beta$.
Indeed, if $\beta' \in k$ satisfies $(\beta')^d = -c$,
then $\beta' = \zeta\beta$ for some
$d$-th root of unity $\zeta \in k$, and $\zeta$ is a
unit because the residue characteristic does not
divide $d$.
Dividing by $\beta'$ instead of by $\beta$ therefore
multiplies the reduction $\ol{f_{d, c}^j(x)/\beta}$ by
the nonzero element $\ol{\zeta}^{-1}$ of the residue
field, for every $x \in \calK_k(f_{d, c})$ and every
$j \geq 0$. For each $j$, the reductions of
$f_{d, c}^j(X)/\beta'$ and $f_{d, c}^j(Y)/\beta'$ are
therefore distinct if and only if the reductions of
$f_{d, c}^j(X)/\beta$ and $f_{d, c}^j(Y)/\beta$ are
distinct, and the two minima agree.

For $j < \ell_v(X, Y)$, the $j$-th normalized iterates
of $X$ and of $Y$ have the same reduction, and the
normalized map $T_\beta$ of
Lemma~\ref{lem:tame-normalization} sends those two
iterates to the $(j + 1)$-st normalized iterates. The
following lemma shows that passing from index $j$
to index $j + 1$ decreases the valuation of the
difference of the two normalized iterates by exactly
$(d - 1)s$,
because $T_\beta$ multiplies a difference
of two units with the same reduction
by $\beta^{d - 1}$ and a unit, while
$v\paren{\beta^{d - 1}} = -(d - 1)s$.

\begin{lemma}
\label{lem:one-step-distance}
	Let $d \geq 2$, let $k$ be a nonarchimedean local
	field whose residue characteristic does not divide
	$d$, let $s \geq 1$, and let $\beta \in k$ satisfy
	$v(\beta) = -s$. Set
	$T_\beta(y) := \beta^{d - 1}(y^d - 1)$. If
	$u, w \in \calo_k^\times$ have the same reduction,
	then
	\[
		v(u - w) = (d - 1)s
			+ v\paren{T_\beta(u) - T_\beta(w)}.
	\]
\end{lemma}

\begin{proof}
	Since $T_\beta(y) = \beta^{d - 1}(y^d - 1)$,
	the two constant terms cancel,
	and factoring $u^d - w^d$
	gives
	\[
		T_\beta(u) - T_\beta(w)
		= \beta^{d - 1}\paren{u^d - w^d}
		= \beta^{d - 1}(u - w)
		\sum_{i = 0}^{d - 1}
		u^{d - 1 - i}w^i.
	\]
	Since $\ol{u} = \ol{w}$,
	the sum has reduction $d\,\ol{u}^{\,d - 1}$,
	which is nonzero because $u$ is a unit
	and the residue characteristic does not
	divide $d$.
	The sum is therefore a unit, and
	$v\paren{\beta^{d - 1}} = -(d - 1)s$, so taking
	valuations gives
	$v\paren{T_\beta(u) - T_\beta(w)}
	= -(d - 1)s + v(u - w)$.
\end{proof}

A difference of two units has nonnegative valuation,
so the valuation of the difference of the $j$-th
normalized iterates of $X$ and of $Y$ cannot decrease
by $(d - 1)s$ at every index $j \geq 0$ unless
$X = Y$. For distinct $X$ and $Y$, the index
$\ell_v(X, Y)$ is therefore finite, and adding the
decreases at the indices $j < \ell_v(X, Y)$ gives
$v(X - Y)$.

\begin{proposition}
\label{prop:tame-distance}
	Let $d \geq 2$, let $k$ be a nonarchimedean local
	field whose residue characteristic does not divide
	$d$, and let $c \in k$ satisfy $v(c) < 0$. Set
	$f_{d, c}(z) := z^d + c$, suppose that
	$\calK_k(f_{d, c})$ is nonempty, and let $s \geq 1$
	be the integer with $v(c) = -ds$ given by
	Lemma~\ref{lem:shell}.
	If $X, Y \in \calK_k(f_{d, c})$ are distinct,
	then $\ell_v(X, Y)$ is finite and
	\[
		v(X - Y)
		= -s + (d - 1)s\,\ell_v(X, Y).
	\]
\end{proposition}

\begin{proof}
	Fix $\beta \in k$ with $\beta^d = -c$, and write
	$u_j := f_{d, c}^j(X)/\beta$ and
	$w_j := f_{d, c}^j(Y)/\beta$ for the normalized
	iterates of $X$ and of $Y$.
	By Lemma~\ref{lem:tame-normalization}, every $u_j$
	and every $w_j$ is a unit, and $v(\beta) = -s$.
	Moreover $T_\beta$ is the conjugate of $f_{d, c}$
	by $z = \beta y$, so that
	$u_{j + 1} = T_\beta(u_j)$ and
	$w_{j + 1} = T_\beta(w_j)$ for every $j \geq 0$.
	Lemma~\ref{lem:one-step-distance} therefore gives
	\begin{equation}
		v\paren{u_j - w_j}
		= (d - 1)s
		+ v\paren{u_{j + 1} - w_{j + 1}}
		\label{eq:one-step-distance}
	\end{equation}
	for every $j \geq 0$ with $\ol{u_j} = \ol{w_j}$.

	Suppose that $\ol{u_j} = \ol{w_j}$ for every
	$j \geq 0$. Applying \eqref{eq:one-step-distance}
	for $j < N$ and using
	$v\paren{u_N - w_N} \geq 0$ gives
	$v\paren{u_0 - w_0} \geq N(d - 1)s$ for every
	$N \geq 1$. As $(d - 1)s \geq 1$, this forces
	$u_0 = w_0$, hence $X = Y$, contrary to the
	hypothesis. Therefore $\ell_v(X, Y)$ is finite.

	By the definition of $\ell_v(X, Y)$, the reductions
	$\ol{u_j}$ and $\ol{w_j}$ agree for every
	$j < \ell_v(X, Y)$ and differ for
	$j = \ell_v(X, Y)$.
	The difference $u_j - w_j$
	therefore has nonzero reduction, and hence
	valuation $0$, when $j = \ell_v(X, Y)$.
	Applying \eqref{eq:one-step-distance} for every
	$j < \ell_v(X, Y)$ gives
	$v\paren{u_0 - w_0} = (d - 1)s\,\ell_v(X, Y)$.
	Since $X - Y = \beta\paren{u_0 - w_0}$ and
	$v(\beta) = -s$,
	we have
	\[
		v(X - Y)
		= -s + (d - 1)s\,\ell_v(X, Y).
	\]
\end{proof}

\subsection{Benedetto's product identity}

To prove Theorem~\ref{thm:power-classes},
the main Kummer-theoretic input
used in both Theorems~\ref{thm:main}
and \ref{thm:elliptic-curves-large-rank},
we need to control $P - \zeta Q$
at every finite place $w \nmid d$
for distinct periodic points $P, Q$ and
$\zeta \in \mu_d \sm \set{1}$.
Proposition~\ref{prop:tame-distance} treats $w(c) < 0$.
When $w(c) \geq 0$, we use the product identity of
Benedetto~\cite[Theorems~2 and~3]{benedetto-product-identity}
to show that $P - \zeta Q$ is a unit in the valuation
ring of $K(\mu_d)_w$.

We first record a general observation about
periodic points and automorphisms preserving
the fibers of a morphism. It will explain why
the factors $P - \zeta Q$ are nonzero.

\begin{lemma}
\label{lem:periodic-orbits}
	Let $k$ be a field and let $f \colon X \lra X$
	be a morphism of $k$-varieties.
	If $G \subseteq \Aut_k(X)$ is a subgroup such that
	$f \circ g = f$ for every $g \in G$,
	then $f$ is injective on its periodic points
	in $X(\ol{k})$.
	Moreover, $G$ preserves the set of preperiodic
	points in $X(\ol{k})$, and each $G$-orbit
	contains at most one periodic point.
\end{lemma}

\begin{proof}
	We first prove injectivity, which also gives
	the assertion about $G$-orbits.
	If $P, Q$ are periodic and $f(P) = f(Q)$,
	choose a common multiple $n$ of their periods.
	Then $P = f^n(P) = f^n(Q) = Q$.
	For any $P \in X(\ol{k})$ and $g \in G$,
	the identity $f^j(gP) = f^j(P)$ holds
	for every $j \geq 1$.
	Hence $P$ is preperiodic if and only if $gP$
	is preperiodic.
	If $P$ and $gP$ are periodic, then
	$f(gP) = f(P)$, so injectivity gives $gP = P$.
\end{proof}

Now suppose that $k$ is a field of characteristic zero,
$c \in k$, and $f_{d, c}(z) := z^d + c$.
Multiplication by $\mu_d$ gives an action on
$\A^1$ over $k(\mu_d)$ satisfying
$f_{d, c}(\zeta z) = f_{d, c}(z)$.
Thus Lemma~\ref{lem:periodic-orbits} gives
$P - \zeta Q \neq 0$ for distinct periodic points
$P, Q \in k$ and every $\zeta \in \mu_d$.

Define $G_d(X, Y) := (X^d - Y^d)/(X - Y)$.
The following lemma is the specialization to $f_{d, c}$
of a product identity due to
Benedetto~\cite[Theorems~2 and~3]{benedetto-product-identity}
for monic polynomials;
we include a short version of his proof
for the case that we need.

\begin{lemma}
\label{lem:telescoping}
	Let $P, Q \in k$ be distinct periodic points of
	$f_{d, c}$. If $n$ is a common multiple of their
	periods, then
	\[
		\prod_{j = 0}^{n - 1}
		G_d\paren{f_{d, c}^j(P), f_{d, c}^j(Q)}
		= 1.
	\]
	Furthermore, if $k$ is nonarchimedean,
	$\mu_d \subset k$, and $c \in \calo_k$,
	then $P - \zeta Q \in \calo_k^\times$ for every
	$\zeta \in \mu_d \sm \set{1}$.
\end{lemma}

\begin{proof}
	We first show that the identity follows by telescoping,
	and then show that integrality
	forces each cyclotomic factor to be a unit.

	By Lemma~\ref{lem:periodic-orbits}, all differences
	$f_{d, c}^j(P) - f_{d, c}^j(Q)$ are nonzero.
	Iterating $f_{d, c}(X) - f_{d, c}(Y) = (X - Y)G_d(X, Y)$
	and using $f_{d, c}^n(P) = P$ and $f_{d, c}^n(Q) = Q$
	proves the desired product identity.

	Finally, suppose that $k$ is nonarchimedean,
	$\mu_d \subset k$, and $c \in \calo_k$.
	Every periodic point is integral, because a point of
	negative valuation has unbounded orbit.
	Hence all the factors
	$G_d(f_{d, c}^j(P), f_{d, c}^j(Q))$
	in the product identity are integral,
	and their product being one makes each one
	a unit. In particular,
	\[
		G_d(P, Q)
		= \prod_{\zeta \in \mu_d \sm \set{1}}
		(P - \zeta Q)
	\]
	is a unit.
	Each factor on the right is integral, so
	each factor $P - \zeta Q$ is a unit.
\end{proof}

\subsection{Kummer classes}
\label{sec:kummer}

Let $S$ consist of the archimedean places of
$K(\mu_d)$ and the finite places above $d$.
The following group records the Kummer classes whose
valuations outside $S$ are divisible by $d - 1$.

\begin{definition}
\label{def:classes-modulo-powers}
	Define
	\[
		\calc_{K, d}
			:= \Ker\paren{
			\frac{K(\mu_d)^\times}
			{(K(\mu_d)^\times)^{d - 1}}
			\xrightarrow{\divisor}
			\frac{\Div\paren{\calo_{K(\mu_d), S}}}
			{(d - 1)
			\Div\paren{\calo_{K(\mu_d), S}}}}.
		\]
	Fix a subset $A \subset K(\mu_d)^\times$ containing $1$
	and exactly one representative of each class in
	$\calc_{K, d}$.
\end{definition}

By the definition of the divisor map $\divisor$,
a class $[a] \in K(\mu_d)^\times / (K(\mu_d)^\times)^{d - 1}$
is in $\calc_{K, d}$ exactly when $w(a)$ is
divisible by $d - 1$ at every finite place
$w \nmid d$.
Since an $S$-unit has valuation zero at
every such place, every $S$-unit represents a class
in $\calc_{K, d}$.
The group $\calc_{K, d}$ depends only on $K$ and $d$.

The following exact sequence identifies the two
sources of its elements: classes of $S$-units and
$(d - 1)$-torsion in the $S$-ideal class group.

\begin{lemma}
\label{lem:finite-power-classes}
	There is a short exact sequence
	\[
		\begin{tikzcd}
			1
				\arrow[r] &
			\frac{\calo_{K(\mu_d),S}^\times}
				{(\calo_{K(\mu_d),S}^\times)^{d - 1}}
				\arrow[r] &
			\calc_{K, d}
				\arrow[r] &
			\Cl\paren{\calo_{K(\mu_d),S}}[d - 1]
				\arrow[r] &
			1.
		\end{tikzcd}
	\]
	In particular, $\calc_{K, d}$ is finite
	and can be computed from the
	$S$-unit group and the $S$-ideal class group.
\end{lemma}

\begin{proof}
	If $a \in K(\mu_d)^\times$ represents a class in
	$\calc_{K, d}$, then
	$a\calo_{K(\mu_d),S} = \mfa^{d - 1}$ for a
	fractional ideal $\mfa$ of
	$\calo_{K(\mu_d),S}$. Sending $[a]$ to $[\mfa]$
	defines a map to
	$\Cl\paren{\calo_{K(\mu_d),S}}[d - 1]$. Its kernel consists
	exactly of the classes represented by units of
	$\calo_{K(\mu_d),S}$.

	Conversely, if
	$[\mfa] \in \Cl\paren{\calo_{K(\mu_d),S}}[d - 1]$, choose
	$a \in K(\mu_d)^\times$ with
	$\mfa^{d - 1} = a\calo_{K(\mu_d),S}$. Then $[a]$ belongs to
	$\calc_{K, d}$ and maps to $[\mfa]$. The unit theorem
	and the finiteness of the ideal class group give the
	last assertion; see
	Neukirch \cite[Chapters~I and V]{neukirch-ant}.
\end{proof}

In our proofs of Theorems~\ref{thm:main} and
\ref{thm:elliptic-curves-large-rank},
a key object is the quotient
$\frac{P - \zeta Q}{P - \xi Q} \in K(\mu_d)^\times$
for distinct $P, Q \in \Per_K(f_{d, c})$
and distinct roots of unity $\zeta, \xi \in \mu_d \sm \set{1}$.
This quotient may be viewed as a cyclotomic twist
of the dynamical units studied by
Narkiewicz~\cite{narkiewicz} and
Morton--Silverman~\cite{morton-silverman-periodic},
which concern $S$-units
$(P_1 - P_2)/(P_3 - P_4)$
when $P_i \in \Per_K(f_{d, c})$.
In our setting,
$\zeta Q$ and $\xi Q$ have the same image
under $f_{d, c}$ since
$f_{d, c}(\zeta Q) = f_{d, c}(\xi Q) = f_{d, c}(Q)$.
The congruences resulting from this cyclotomic symmetry
yield the stronger conclusion below,
with $\calc_{K, d}$ independent of $c$.

\begin{theorem}
\label{thm:power-classes}
	Let $K$ be a number field, let $d \geq 2$,
	and let $c \in K$.
	Fix distinct $P, Q \in \Per_K(f_{d, c})$.
	Then the following statements hold.
	\begin{enumerate}[(a)]
		\item There is a unique coset $\cala_{P, Q}$ of
			$\calc_{K, d}$ in
			$K(\mu_d)^\times / (K(\mu_d)^\times)^{d - 1}$
			such that $[P - \zeta Q] \in \cala_{P, Q}$ for
			every $\zeta \in \mu_d \sm \set{1}$.
		\item For all distinct
			$\zeta, \xi \in \mu_d \sm \set{1}$,
			\[
				\sbrac{\frac{P - \zeta Q}{P - \xi Q}}
				\in \calc_{K, d}.
			\]
	\end{enumerate}
\end{theorem}

\begin{proof}
	We fix $P, Q$ and compare the valuations of
	$P - \zeta Q$ modulo $d - 1$ as $\zeta$ varies.
	For $\zeta \in \mu_d \sm \set{1}$,
	Lemma~\ref{lem:periodic-orbits} gives
	$P - \zeta Q \neq 0$.
	Fix a finite place $w \nmid d$ of $K(\mu_d)$.
	If $w(c) \geq 0$, the unit assertion of
	Lemma~\ref{lem:telescoping} gives $w(P - \zeta Q) = 0$.

	Suppose that $w(c) < 0$. Lemma~\ref{lem:shell} gives
	$w(c) = -ds_w$ for an integer $s_w \geq 1$.
	The point $\zeta Q$ has bounded orbit
	over $K(\mu_d)_w$,
	since $f_{d, c}(\zeta Q) = f_{d, c}(Q)$.
	Since $P \neq \zeta Q$,
	Proposition~\ref{prop:tame-distance}
	gives $w(P - \zeta Q) \equiv -s_w \pmod{d - 1}$.
	Thus, for the fixed pair $P, Q$, all the classes
	$[P - \zeta Q]$ have the same image under the divisor
	map in Definition~\ref{def:classes-modulo-powers}.

	Choose $\eta \in \mu_d \sm \set{1}$,
	which is possible since $d \geq 2$.
	The kernel of the divisor map is $\calc_{K, d}$,
	so all the classes $[P - \zeta Q]$ belong to the coset
	$\cala_{P, Q} := [P - \eta Q] \calc_{K, d}$.
	Any coset satisfying part (a) contains $[P - \eta Q]$,
	and two cosets with a common element are equal.
	This proves existence and uniqueness in (a).

	Part (b) follows by taking the quotient
	of the classes $[P - \zeta Q]$ and $[P - \xi Q]$
	in $\cala_{P, Q}$.
\end{proof}

Theorem~\ref{thm:power-classes} is
the Kummer-theoretic input used in
both Theorems~\ref{thm:main} and
\ref{thm:elliptic-curves-large-rank}.
For $d \geq 4$, Section~\ref{sec:proof-main} uses the
classes of $P - \zeta_i Q$, for three distinct
$\zeta_i \in \mu_d \sm \set{1}$, to construct twisted
Fermat curves.
When $d$ is odd, Section~\ref{sec:rank}
uses the images of $\calc_{K, d}$ in
$K(\mu_d)^\times / (K(\mu_d)^\times)^2$ to construct
quadratic twists of elliptic curves.


\section{Uniform boundedness of periodic points}
\label{sec:proof-main}

We prove Theorem~\ref{thm:main} by associating rational
points on twisted Fermat curves to pairs of periodic
points. The construction applies when $d \geq 4$.
For $d \geq 5$, the genus is at least two; for $d = 4$,
the genus is one and we use an additional Mordell--Lang
argument. The case $d = 3$ and
$\Q(\sqrt{-3}) \not\subset K$ has a separate local proof.

\subsection{Twisted Fermat curves}
\label{sec:twisted-fermat}

Let $K$ be a number field and let $d \geq 4$.
Fix distinct $\zeta_1, \zeta_2, \zeta_3 \in
\mu_d \sm \set{1}$, and retain the finite set
$A \subset K(\mu_d)^\times$ of representatives of
$\calc_{K, d}$ from Definition~\ref{def:classes-modulo-powers}.

Let $c \in K$.
For each pair of distinct
$P, Q \in \Per_K(f_{d, c})$,
use the coset $\cala_{P, Q}
\subset K(\mu_d)^\times / (K(\mu_d)^\times)^{d - 1}$
from Theorem~\ref{thm:power-classes}(a)
and choose a scalar $c_{P, Q} \in K(\mu_d)^\times$
with $\cala_{P, Q} = [c_{P, Q}] \calc_{K, d}$.
Then $c_{P, Q} A$ is a set of representatives for $\cala_{P, Q}$,
so we can write $P - \zeta_i Q = c_{P, Q} a_i x_i^{d - 1}$
with $a_i \in A$ and $x_i \in K(\mu_d)^\times$ for
$i = 1, 2, 3$.
From these three equations,
we will eliminate $P, Q$ and
cancel $c_{P, Q}$ to construct rational points on the
following twisted Fermat curves.

\begin{definition}
\label{def:twisted-fermat-family}
	For $a_1, a_2, a_3 \in A$, define
	$C_{a_1, a_2, a_3} \subset \P^2_{K(\mu_d)}$ by
	\[
		(\zeta_2 - \zeta_3)a_1 X_1^{d - 1}
		+ (\zeta_3 - \zeta_1)a_2 X_2^{d - 1}
		+ (\zeta_1 - \zeta_2)a_3 X_3^{d - 1} = 0.
	\]
\end{definition}

The three coefficients are nonzero, so the partial
derivatives have no common projective zero.
Hence $C_{a_1, a_2, a_3}$ is a smooth projective curve
of genus $g_d := (d - 2)(d - 3)/2$.
For example, $C_{1, 1, 1}$ contains $[1 : 1 : 1]$,
since its coefficients sum to zero.
Note that the finite family of curves
$C_{a_1, a_2, a_3}$ indexed by $A^3$
depends only on $K, d$.

We now relate periodic points of $f_{d, c}$
to rational points on the twisted Fermat curves $C_{a_1, a_2, a_3}$.

\begin{proposition}
\label{prop:quotient-curves}
	Let $K$ be a number field, let $d \geq 4$,
	and let $c \in K$. For every nonzero $Q \in \Per_K(f_{d, c})$,
	there is an injective map
	\[
		\Per_K(f_{d, c}) \sm \set{Q}
		\longhookrightarrow
			\coprod_{(a_1, a_2, a_3) \in A^3}
			C_{a_1, a_2, a_3}(K(\mu_d)).
	\]
\end{proposition}

\begin{proof}
	We fix a nonzero periodic point $Q$
	so that the quotient $P/Q$ determines $P$.
	If $Q$ is the only periodic point, the assertion
	is immediate.
	For each $P \in \Per_K(f_{d, c}) \sm \set{Q}$,
	choose $c_{P, Q} \in K(\mu_d)^\times$ with
	$\cala_{P, Q} = [c_{P, Q}]\calc_{K, d}$.
	Choose $a_1, a_2, a_3 \in A$ and
	$x_1, x_2, x_3 \in K(\mu_d)^\times$ supplied by
	Theorem~\ref{thm:power-classes} such that
	\[
		P - \zeta_i Q = c_{P, Q} a_i x_i^{d - 1}
	\]
	for every $i \in \set{1, 2, 3}$.
	Substituting these equations into the identity
	\[
		(\zeta_2 - \zeta_3)(P - \zeta_1 Q)
		+ (\zeta_3 - \zeta_1)(P - \zeta_2 Q)
		+ (\zeta_1 - \zeta_2)(P - \zeta_3 Q) = 0
	\]
	and canceling $c_{P, Q}$ gives
	$[x_1 : x_2 : x_3] \in C_{a_1, a_2, a_3}(K(\mu_d))$.
	Hence our choice of $a_i$ and $x_i$
	defines a map
	$\Per_K(f_{d, c}) \sm \set{Q}
	\lra \coprod_{(a_1, a_2, a_3) \in A^3}
	C_{a_1, a_2, a_3}(K(\mu_d))$.
	Also notice that $a_1 x_1^{d - 1} - a_3 x_3^{d - 1}
	= (\zeta_3 - \zeta_1)Q/c_{P, Q} \neq 0$,
	so solving the equations
	$P - \zeta_1 Q = c_{P, Q} a_1 x_1^{d - 1}$
	and $P - \zeta_3 Q = c_{P, Q} a_3 x_3^{d - 1}$ gives
	\[
		\frac{P}{Q}
		= \frac{\zeta_3 a_1 x_1^{d - 1}
			- \zeta_1 a_3 x_3^{d - 1}}
		{a_1 x_1^{d - 1} - a_3 x_3^{d - 1}}.
	\]

	To show that the choice of $a_i$ and $x_i$
	defines an injective map,
	suppose that $P, P' \in \Per_K(f_{d, c}) \sm \set{Q}$
	have the same image in the disjoint union
	$\coprod_{A^3} C_{a_1, a_2, a_3}(K(\mu_d))$.
	Then they have the same indices
	$(a_1, a_2, a_3) \in A^3$ and
	$[x_1 : x_2 : x_3] = [x_1' : x_2' : x_3']$
	in $\P^2(K(\mu_d))$.
	Hence there is a $t \in K(\mu_d)^\times$
	such that $x_i' = t x_i$ for each $i = 1, 2, 3$.
	But then
	\[
		\frac{P'}{Q}
		= \frac{\zeta_3 a_1 x_1'^{d - 1}
			- \zeta_1 a_3 x_3'^{d - 1}}
		{a_1 x_1'^{d - 1} - a_3 x_3'^{d - 1}}
		= \frac{\zeta_3 a_1 t^{d - 1} x_1^{d - 1}
			- \zeta_1 a_3 t^{d - 1} x_3^{d - 1}}
		{a_1 t^{d - 1} x_1^{d - 1} - a_3 t^{d - 1} x_3^{d - 1}}
		= \frac{P}{Q},
	\]
	hence $P = P'$.
\end{proof}

\subsection{The case \texorpdfstring{$d \geq 5$}{d >= 5}}
\label{sec:dge5}

Fixing $Q$ lets us bound the number of periodic points
by summing the rational point counts on the curves
$C_{a_1, a_2, a_3}$.

\begin{proposition}
\label{prop:dge5}
	Let $K$ be a number field and let $d \geq 5$.
	There is an effectively computable constant
	$B_{\Per}(K, d)$ such that
	\[
		\#\Per_K(z^d + c)
		\leq B_{\Per}(K, d)
	\]
	for every $c \in K$.
\end{proposition}

\begin{proof}
	We count rational points on the fixed finite family
	from Definition~\ref{def:twisted-fermat-family}.
	If there are at least two periodic points, choose
	a nonzero periodic point $Q$.
	Proposition~\ref{prop:quotient-curves} gives
	\[
		\#\Per_K(f_{d, c}) \leq 1 +
		\sum_{(a_1, a_2, a_3) \in A^3}
		\#C_{a_1, a_2, a_3}(K(\mu_d)).
	\]
	The same bound holds if there is at most one periodic
	point. Every curve in the sum depends only on $K, d$
	and has genus $g_d \geq 2$.
	Faltings's theorem~\cite{faltings-mordell} makes the
	sum finite, and
	Yu--Yuan--Zhou~\cite[Theorem~1.6]{yu-yuan-zhou}
	give effective bounds for its summands.
\end{proof}

\begin{remark}
	Section~\ref{sec:quantitative} gives a smaller bound
	by dividing the Fermat equation by $a_3$ and rescaling
	the first two coordinates.
	It also writes the resulting constant
	$B_{\Per}(K, d)$ explicitly 
	using the quantitative point bounds
	of Yu--Yuan--Zhou~\cite{yu-yuan-zhou}.
\end{remark}

\subsection{The case \texorpdfstring{$d = 4$}{d = 4}}
\label{sec:d4}

Let $K$ be a number field and let $d = 4$.
We now prove the $d = 4$ case of Theorem~\ref{thm:main}.
The twisted Fermat curves $C_{a_1, a_2, a_3}$ of
Section~\ref{sec:twisted-fermat} have degree
$d - 1 = 3$ and hence genus $g_4 = 1$.
They may therefore have infinitely many rational points.
Instead of counting these points separately, we will
use two distinct periodic points to place every other
periodic point on a curve inside $E \times E$, for one
fixed elliptic curve $E$ over a fixed extension of $K$.

For each $a \in A$, choose a cube root $\sqrt[3]{a}$
in $\ol{K}$, and set
$L := K(\mu_4, \sqrt[3]{a} : a \in A)$.
The field $L$ depends only on $K$ because $A$ is finite.
Let $E/L$ be the smooth plane cubic
\[
	E : (\zeta_2 - \zeta_3)U^3
	+ (\zeta_3 - \zeta_1)V^3
	+ (\zeta_1 - \zeta_2)W^3 = 0.
\]
Take $[1 : 1 : 1]$ as the origin.
For every $(a_1, a_2, a_3) \in A^3$, the coordinate change
\[
	[X_1 : X_2 : X_3] \longmapsto
	[\sqrt[3]{a_1}X_1 : \sqrt[3]{a_2}X_2 : \sqrt[3]{a_3}X_3]
\]
is an $L$-isomorphism from $C_{a_1, a_2, a_3}$ to $E$.
The expression for $P/Q$ in the proof of
Proposition~\ref{prop:quotient-curves} becomes the value
of the morphism $\phi : E \lra \P^1$ defined by
\[
	\phi([U : V : W])
	:= \sbrac{\zeta_3 U^3 - \zeta_1 W^3 : U^3 - W^3}.
\]
Indeed, if both coordinates defining $\phi$ vanished,
then $(\zeta_3 - \zeta_1)W^3 = 0$. Hence $U = W = 0$,
and the equation of $E$ would also give $V = 0$.
Thus $\phi$ is defined at every point of $E$.

For $t \in L^\times$, let $Z_t$ be the reduced curve
\[
	Z_t := \set{(R, S) \in E \times E \Mid
	\phi(R) = t\phi(S)}.
\]
The equality $\phi(R) = t\phi(S)$ is taken in $\P^1$,
and multiplication by $t$ fixes $\infty$.
The following lemma explains how rational points of
$Z_t$ will bound the number of periodic points.

\begin{lemma}
\label{lem:quartic-periodic-pairs}
	Let $c \in K$, and let $P_0, Q_0 \in \Per_K(f_{4, c})$
	be distinct and nonzero. There is an injective map
	\[
		\Per_K(f_{4, c}) \sm \set{P_0, Q_0}
		\lhook\joinrel\longrightarrow Z_{P_0/Q_0}(L).
	\]
\end{lemma}

\begin{proof}
	Apply Proposition~\ref{prop:quotient-curves} with
	$Q = Q_0$ and with $Q = P_0$, and use the
	$L$-isomorphisms $C_{a_1, a_2, a_3} \lra E$ above.
	For each $P \in \Per_K(f_{4, c}) \sm \set{P_0, Q_0}$,
	choose $R_P, S_P \in E(L)$ with
	$\phi(R_P) = P/Q_0$ and $\phi(S_P) = P/P_0$.
	Then $\phi(R_P) = (P_0/Q_0)\phi(S_P)$, so
	$(R_P, S_P) \in Z_{P_0/Q_0}(L)$.
	Since $P = Q_0\phi(R_P)$, the assignment
	$P \mapsto (R_P, S_P)$ is injective.
\end{proof}

We next bound $\#Z_t(L)$ independently of $t \neq 1$.
The proof uses the ramification of $\phi$: we first
find the images in $\P^1$ of its ramification points,
then compare them with the corresponding images for
$t\phi$. This comparison excludes elliptic components
of $Z_t$.

\begin{proposition}
\label{prop:d4}
	Let $K$ be a number field.
	There is an effectively computable constant
	$B_{\Per}(K, 4)$ such that
	\[
		\#\Per_K(z^4 + c) \leq B_{\Per}(K, 4)
	\]
	for every $c \in K$.
\end{proposition}

\begin{proof}
	We determine the ramification of $\phi$ over $\ol{L}$.
	Cubing the coordinates gives a morphism from $E$ onto
	the line
	\[
		(\zeta_2 - \zeta_3)u
		+ (\zeta_3 - \zeta_1)v
		+ (\zeta_1 - \zeta_2)w = 0
	\]
	in $\P^2$, followed by the isomorphism
	$[u : v : w] \mapsto
	[\zeta_3 u - \zeta_1 w : u - w]$ to $\P^1$.
	The composite of these two morphisms is $\phi$.
	Above a point of this line with $uvw \neq 0$, there
	are nine choices of cube roots up to common scaling,
	and the first morphism is unramified.
	At a point of $E$ with one coordinate zero, its
	quotient by a nonzero coordinate is a local parameter.
	The corresponding quotient on the line pulls back to
	the cube of this parameter. Thus $\phi$ has degree nine
	and ramification index three exactly at $UVW = 0$.
	The points with $U = 0$, $V = 0$, and $W = 0$ map to
	$\zeta_1$, $\zeta_2$, and $\zeta_3$, respectively.

	Fix $t \in L^\times \sm \set{1}$. Both projections
	of $Z_t$ to $E$ are finite. The normalization $T$ of
	any geometric component therefore has nonconstant
	maps $\pi_1, \pi_2 : T \lra E$ and has genus at least
	one. Suppose that $T$ has genus one. By
	Riemann--Hurwitz, both $\pi_1$ and $\pi_2$ are
	\'etale. Consequently, $\phi \circ \pi_1$ is ramified
	precisely over $\zeta_1, \zeta_2, \zeta_3$, while
	$t\phi \circ \pi_2$ is ramified precisely over
	$t\zeta_1, t\zeta_2, t\zeta_3$.
	These two morphisms are equal by the definition of
	$Z_t$, so
	\[
		\set{\zeta_1, \zeta_2, \zeta_3}
		= \set{t\zeta_1, t\zeta_2, t\zeta_3}.
	\]
	Taking the sums of these roots of unity gives
	$\zeta_1 + \zeta_2 + \zeta_3
	= t(\zeta_1 + \zeta_2 + \zeta_3)$.
	Since $\zeta_1 + \zeta_2 + \zeta_3 = -1$, we obtain
	$t = 1$, a contradiction. Hence every component
	of $Z_t$ has geometric genus at least two and cannot
	be a translate of an elliptic curve.

	The degrees of $Z_t$ in a fixed product embedding
	are bounded independently of $t$.
	R\'{e}mond~\cite[Theorem~2.1]{remond-count}, applied
	to the geometric components and the subgroup
	$E(L) \times E(L)$ of $E \times E$, gives
	a constant $N_K \geq 1$ such that
	\begin{equation}
		\#Z_t(L) \leq N_K
		\label{eq:quartic-intersection-bound}
	\end{equation}
	for every $t \in L^\times \sm \set{1}$.
	The constant depends only on $E/L$ and the degree
	bound, hence only on $K$.

	If there are fewer than two nonzero periodic points,
	then $\#\Per_K(f_{4, c}) \leq 2$.
	Otherwise choose distinct nonzero periodic points
	$P_0, Q_0$. Lemma~\ref{lem:quartic-periodic-pairs} gives
	$\#\Per_K(f_{4, c}) \leq 2 + \#Z_{P_0/Q_0}(L)$.
	Thus $B_{\Per}(K, 4) := 2 + N_K$ works.
\end{proof}

\begin{remark}
	Section~\ref{sec:quantitative} computes a degree bound
	for $Z_t$ independent of $t$ and uses an effective
	upper bound for $\rank E(L)$ in
	the explicit form of R\'{e}mond's Mordell--Lang estimate
	in David--Philippon~\cite[p.~643]{david-philippon-ii}.
	This makes $N_K$ and $B_{\Per}(K, 4) = 2 + N_K$
	effectively computable.
\end{remark}

\subsection{The case
	\texorpdfstring{$d = 3$}{d = 3} and
	\texorpdfstring{$\Q(\sqrt{-3}) \not\subset K$}
		{Q(sqrt(-3)) notsubset K}}
\label{sec:d3}

Let $K$ be a number field not containing $\Q(\sqrt{-3})$
and let $d = 3$.
We now prove this case of Theorem~\ref{thm:main}.
The missing $3$-rd roots of unity provide a simpler local argument
than the twisted Fermat curve construction
that we used for the $d \geq 4$ cases.

Since $K(\mu_3)/K$ is a nontrivial quadratic extension,
Chebotarev's density theorem gives a finite place
$v \nmid 3$ that does not split in $K(\mu_3)/K$.
Let $\calo_v$ and $k_v$ be the valuation ring
and residue field of $K_v$.
Since $v \nmid 3$, Hensel's lemma
shows that $k_v$ contains no nontrivial $3$-rd root of unity.

Pezda~\cite[Theorem~1]{pezda-local} gives an effectively
computable integer $C_v \geq 1$ such that every periodic
point in $\calo_v$ of every $g \in \calo_v[z]$ has period
at most $C_v$.
This bounds the periods when $v(c) \geq 0$.
When $v(c) < 0$, Proposition~\ref{prop:tame-distance}
shows that there is at most one periodic point.

\begin{proposition}
\label{prop:d3}
	Let $K$ be a number field satisfying
	$\Q(\sqrt{-3}) \not\subset K$, let $v \nmid 3$ be
	a finite place that does not split in
	$K(\mu_3)/K$, and let $C_v$ be as above.
	Then, for every $c \in K$,
	\[
		\#\Per_K(z^3 + c)
		\leq \frac{3^{C_v + 1} - 3}{2}.
	\]
\end{proposition}

\begin{proof}
	We distinguish whether $c$ is integral at $v$.
	If $v(c) \geq 0$, every periodic point is in
	$\calo_v$ and has period at most $C_v$. Counting
	roots of $f_{3, c}^j(z) - z$ for $1 \leq j \leq C_v$
	gives
	\[
		\#\Per_K(f_{3, c}) \leq \sum_{j = 1}^{C_v}3^j
		= \frac{3^{C_v + 1} - 3}{2}.
	\]

	Suppose that $v(c) < 0$ and a periodic point exists.
	Lemmas~\ref{lem:shell} and \ref{lem:tame-normalization}
	give $v(c) = -3s$ for an integer $s \geq 1$ and
	$\beta \in K_v$ with $\beta^3 = -c$. The map
	$T_\beta(y) = \beta^2(y^3 - 1)$ is conjugate to
	$f_{3, c}$ by $z = \beta y$. Every normalized
	iterate $y$ belongs to $\calo_v^\times$, as does
	$T_\beta(y)$. Since $v(\beta) = -s$, it follows
	that $v(y^3 - 1) = 2s > 0$.
	The residue field $k_v$ has
	no nontrivial third root of unity,
	so every normalized iterate reduces to $1$.

	For periodic points $P, Q$, every normalized iterate
	of $P$ and every normalized iterate of $Q$ therefore
	reduce to $1$. Thus $\ell_v(P, Q) = \infty$.
	Proposition~\ref{prop:tame-distance} gives $P = Q$,
	so there is at most one periodic point.
\end{proof}

The three degree ranges now prove Theorem~\ref{thm:main}.

\begin{proof}[Proof of Theorem~\ref{thm:main}]
	Apply Proposition~\ref{prop:dge5} when $d \geq 5$,
	Proposition~\ref{prop:d4} when $d = 4$, and
	Proposition~\ref{prop:d3} when $d = 3$.
\end{proof}

\begin{remark}
	The explicit bounds and the effectivity assertions
	in Theorem~\ref{thm:main}
	are discussed in
	Section~\ref{sec:quantitative}.
\end{remark}

\subsection{Extension to other polynomials}
\label{sec:related-families}

Looper \cite{looper-2025}
proved the uniform boundedness conjecture
for periodic points of polynomials over number fields
conditional on the $abc$ conjecture and its
higher-dimensional generalizations.
We mildly extend Theorem~\ref{thm:main}
to obtain unconditional results for
two families of polynomials
beyond $z^d + c$
using two standard dynamical observations.

The first observation is that linear conjugacy
preserves dynamical properties
(see \cite[Section~1]{silverman-dynamics} for instance).
A degree-$d$ polynomial over $K$ with one finite
critical point has the general form
$f(z) = a(z - \gamma)^d + b$, where
$a \in K^\times$ and $b, \gamma \in K$.
Letting $u \in \overline{K}$
be a $(d - 1)$-st root of $a$,
we have that $f$ is linearly conjugate
over $K(u)$ to $z^d + u(b - \gamma)$
via $z \mapsto u(z - \gamma)$.
Since Theorem~\ref{thm:main} over $K(u)$
is uniform in the constant term,
it gives the bound $B_{\Per}(K(u), d)$
for all such polynomials when $d \geq 4$,
or when $d = 3$ and
$\Q(\sqrt{-3}) \not\subset K(u)$.

Recall that a semiconjugacy
from $f$ to $h$ is a map
$\pi$ satisfying $\pi \circ f = h \circ \pi$
(see \cite[Chapter~I, Lemma~17]{block-coppel}
for instance).
Hence $\pi \circ f^n = h^n \circ \pi$
for every $n \geq 1$, so $\pi$ sends periodic points of
$f$ to periodic points of $h$.
We apply this observation to a family of polynomials
with several finite critical points.

\begin{corollary}
\label{cor:semiconjugacy-transfer}
	Let $K$ be a number field and
	let $r, s \geq 2$.
	For every $a \in K^\times$, there is a
	constant $B_{\Per}(K, r, s, a)$ such that
	\[
		\#\Per_K\paren{a(z^r - \gamma)^s}
		\leq B_{\Per}(K, r, s, a)
	\]
	for every $\gamma \in K$.
\end{corollary}

\begin{proof}
	Set $\pi(z) := z^r$ and
	$h_\gamma(w) := a^r(w - \gamma)^{rs}$.
	Then $\pi \circ f_\gamma = h_\gamma \circ \pi$, where
	$f_\gamma(z) := a(z^r - \gamma)^s$.
	The preceding linear-conjugacy argument uniformly
	bounds $\#\Per_K(h_\gamma)$ as $\gamma$ varies,
	and the result follows from the semiconjugacy.
\end{proof}

When $\gamma \neq 0$, the polynomial
$f(z) = a(z^r - \gamma)^s$ in
Corollary~\ref{cor:semiconjugacy-transfer}
has a critical point at zero and additional
critical points at the $r$ roots of $z^r = \gamma$.
Thus Theorem~\ref{thm:main} also gives
an effective bound on the periodic points
for a family of non-unicritical polynomials.


\section{Elliptic curves of large rank over
	\texorpdfstring{$K(\mu_d)$}{K(mu\_d)}}
\label{sec:rank}

We now turn to proving
Theorem~\ref{thm:elliptic-curves-large-rank}.
We return to the finite subgroup
$\calc_{K, d}$ of classes modulo
$(d - 1)$-st powers from
Definition~\ref{def:classes-modulo-powers}.
When $d$ is odd, these classes have well-defined
images modulo squares.
The finitely many resulting classes determine
finitely many quadratic twists of an elliptic curve
on which the periodic points can be lifted.
A long cycle therefore gives many rational points
on curves of the form
$\set{(A, B) \mid x_2(B) = f_{d, c}\paren{x_1(A)}}$
inside products of two such elliptic curves.
Uniform Mordell--Lang then forces at least one
of the two Mordell--Weil ranks to be large.

This section has four parts. We first define the
elliptic curves with degree-$2$ maps to $\P^1$
determined by two periodic points, and compute their
variation in moduli.
We next prove a uniform Mordell--Lang bound
for pairs $(A, B) \in E_1 \times E_2$
satisfying $\phi_2(B) = g(\phi_1(A))$, where
$\phi_i : E_i \lra \P^1$ has degree $2$ for
$i = 1, 2$ and $g : \P^1 \lra \P^1$ is a
polynomial map.
We then combine this bound with the Kummer descent to
prove Theorem~\ref{thm:elliptic-curves-large-rank}.
Finally, we specialize the construction to the
remaining cubic case and prove
Corollary~\ref{cor:conditional-d3}.

\subsection{Elliptic curves attached to two periodic points}

When $d \geq 3$ is odd, the finite subgroup
$\calc_{K, d} \subset
K(\mu_d)^\times / (K(\mu_d)^\times)^{d - 1}$
maps naturally to $K(\mu_d)^\times / (K(\mu_d)^\times)^2$.
Fix $c \in K$, distinct nonzero
$Q, R \in \Per_K(f_{d, c})$, and distinct
$\zeta, \xi \in \mu_d \sm \set{1}$.
For a periodic point $P$ different from $Q, R$,
Theorem~\ref{thm:power-classes}(b) controls the classes
of $(P - \zeta Q)/(P - \xi Q)$ and
$(P - \zeta R)/(P - \xi R)$ modulo squares.
Their product has the same class modulo squares as
$(P - \zeta Q)(P - \xi Q)(P - \zeta R)(P - \xi R)$.
This leads to quadratic twists of
a double cover of $\P^1$ branched at
$\zeta Q$, $\xi Q$, $\zeta R$, and $\xi R$.
We first check that these four points
for distinct nonzero periodic points $Q$ and $R$.

\begin{lemma}
\label{lem:four-branch-points}
	Let $d \geq 3$, $L$ be a field containing
	$\mu_d$, and $c \in L$.
	If $Q, R \in \Per_L(f_{d, c})$ are distinct and nonzero
	and $\zeta, \xi \in \mu_d\sm\set{1}$ are distinct,
	then $\zeta Q$, $\xi Q$, $\zeta R$, and $\xi R$ are
	pairwise distinct.
\end{lemma}

\begin{proof}
	Since $Q, R \neq 0$ and $\zeta \neq \xi$,
	the two points in each of the pairs
	$\set{\zeta Q, \xi Q}$ and
	$\set{\zeta R, \xi R}$ are distinct.
	By Lemma~\ref{lem:periodic-orbits}, the distinct
	periodic points $Q, R$ lie in distinct
	$\mu_d$-orbits. Hence the two pairs are disjoint.
\end{proof}

We now define the quadratic twists of
a double cover of $\P^1$ branched at
$\zeta Q$, $\xi Q$, $\zeta R$, and $\xi R$
and then describe it as an elliptic curve.

\begin{definition}
\label{def:four-branch-elliptic-curves}
	Let $L$ be a field of characteristic zero.
	Choose distinct $\zeta, \xi \in L^\times$
	and $Q, R, u \in L^\times$ such that
	$\zeta Q$, $\xi Q$, $\zeta R$, and $\xi R$ are
	pairwise distinct.
	Define $E_{Q, R}^{(u)}$ to be the smooth
	projective curve with affine model
	\[
		E_{Q, R}^{(u)}:\quad uY^2
			= (X - \zeta Q)(X - \xi Q)
				(X - \zeta R)(X - \xi R),
	\]
	and write $E_{Q, R} := E_{Q, R}^{(1)}$.
	Let $x_{Q, R}^{(u)}: E_{Q, R}^{(u)} \lra \P^1$ be the
	morphism induced by the $X$-coordinate.
\end{definition}

Write $F_{Q, R}(X) := (X - \zeta Q)(X - \xi Q)
(X - \zeta R)(X - \xi R)$, so the affine equation of
$E_{Q, R}^{(u)}$ is $uY^2 = F_{Q, R}(X)$.
The four roots of $F_{Q, R}$ give $L$-rational points
with $Y = 0$. We use these points to choose an origin
on each twist and identify its $2$-torsion.

\begin{lemma}
\label{lem:four-branch-elliptic-geometry}
	Let $L$ be a field of characteristic zero, choose
	distinct $\zeta, \xi \in L^\times$, and let
	$Q, R, u \in L^\times$. Suppose that $\zeta Q$,
	$\xi Q$, $\zeta R$, and $\xi R$ are pairwise distinct.
	Then $E_{Q, R}^{(u)}$ has genus $1$ and
	$x_{Q, R}^{(u)}$ has degree $2$.
	Furthermore, choosing any of the points $(\zeta Q, 0)$,
	$(\xi Q, 0)$, $(\zeta R, 0)$, and $(\xi R, 0)$ as
	the origin makes $E_{Q, R}^{(u)}$ an elliptic curve
	over $L$, and these four points form its full
	$L$-rational $2$-torsion.
\end{lemma}

\begin{proof}
	The morphism $x_{Q, R}^{(u)}$ is a double cover of
	$\P^1$ ramified at the four points $(\zeta Q, 0)$,
	$(\xi Q, 0)$, $(\zeta R, 0)$, and $(\xi R, 0)$.
	By Riemann--Hurwitz, $E_{Q, R}^{(u)}$ has genus $1$.
	Each of the four points is $L$-rational.
	After choosing one as the origin,
	the involution $(X, Y) \mapsto (X, -Y)$
	acts as negation.
	Its four fixed points therefore
	form the full $2$-torsion.
\end{proof}

All the twists $E_{Q, R}^{(u)}$ become isomorphic over
$\overline{L}$, so they must have the same $j$-invariant.
To count geometric isomorphism classes as $R$ varies
with $Q$ fixed, we now bound the number of values of
$R/Q$ giving any one $j$-invariant.

\begin{lemma}
\label{lem:four-branch-j-map}
	Let $L$ be a field of characteristic zero
	and choose distinct $\zeta, \xi \in L^\times$
	and $Q, R \in L^\times$
	such that $\zeta Q$, $\xi Q$, $\zeta R$,
	and $\xi R$ are pairwise distinct.
	Then the $j$-invariant
	of the curve $E_{Q, R}$ depends on
	$Q$ and $R$ only through $t := R/Q$, and the
	resulting map from $t$ to the $j$-line has degree
	$12$. In particular, at most $12$ values of $t$
	give one geometric isomorphism class.
\end{lemma}

\begin{proof}
	Set $r := \xi/\zeta$. The change of coordinate
	$X \mapsto X/(\zeta Q)$ takes the branch points to
	$1, r, t, rt$.
	Hence the $j$-invariant depends only on $t$. One
	cross-ratio of these four points is
	$\lambda_r(t) := r(1 - t)^2/
	\paren{(1 - rt)(r - t)}$, which has degree $2$ in
	$t$. The classical map from the cross-ratio line to
	the $j$-line,
	\[
		\lambda	\longmapsto
			256\frac{(1 - \lambda + \lambda^2)^3}
			{\lambda^2(1 - \lambda)^2},
	\]
	has degree six. Their composition therefore has
	degree $12$.
\end{proof}

\subsection{A uniform Mordell--Lang estimate}
\label{sec:elliptic-mordell-lang}

We now prove the estimate on which
Theorem~\ref{thm:elliptic-curves-large-rank}
and Theorem~\ref{thm:fixed-elliptic-curve-rank} rest.
Proposition~\ref{prop:elliptic-mordell-lang-bound}
bounds the number of values $a$ for which there are
$A \in \Gamma_1$ and $B \in \Gamma_2$ with
$a = x_1(A)$ and $g(a) = x_2(B)$.
Here $x_i : E_i \lra \P^1$ is a degree-$2$ map on an
elliptic curve, $g$ is a polynomial of degree $e$, and
$\Gamma_i \subset E_i(\overline L)$ is a subgroup of
finite rank.
The bound is exponential in
$\rank \Gamma_1 + \rank \Gamma_2$, and it depends on
no other feature of $E_1$, $E_2$, and $g$ than $e$.

The quantitative input is the uniform Mordell--Lang
theorem of Gao--Ge--K\"{u}hne
\cite[Theorem~1.1]{gao-ge-kuhne}.
Let $A$ be an abelian variety over $\overline L$
equipped with an ample line bundle, let
$Z \subset A$ be an irreducible closed subvariety, and
let $\Gamma \subset A(\overline L)$ be a subgroup of
finite rank.
The uniform Mordell--Lang theorem
of Gao--Ge--K\"{u}hne
\cite[Theorem~1.1]{gao-ge-kuhne}
supplies a constant
$c_{\UML}\paren{\dim A, \deg Z}$, depending
only on $\dim A$ and on the degree of $Z$ with respect
to the chosen line bundle, such that $Z \cap \Gamma$
is covered by at most
$c_{\UML}\paren{\dim A, \deg Z}^{
1 + \rank \Gamma}$
translates of abelian subvarieties of $A$ that are
contained in $Z$.
If no translate of a
positive-dimensional abelian subvariety of $A$ is
contained in $Z$, then every translate in such a
covering is a single point, and therefore
\begin{equation}
	\label{eq:uniform-mordell-lang-bound}
	\#(Z \cap \Gamma)
		\leq c_{\UML}\paren{\dim A, \deg Z}^{
		1 + \rank \Gamma}.
\end{equation}

We will apply this bound to pairs of points on two
elliptic curves whose images in $\P^1$ satisfy a
polynomial relation.
We first define the curve of such pairs.

\begin{definition}
\label{def:elliptic-relation-curve}
	Let $L$ be a number field, let $E_1, E_2$ be
	elliptic curves over $L$, let
	$\phi_i : E_i \lra \P^1$ be nonconstant for
	$i = 1, 2$, and let $g \in L[X]$ have degree
	$e \geq 1$.
	Write $\scg_g \subset \P^1 \times \P^1$ for the
	graph of $g$, and write $f := (\phi_1, \phi_2)$
	for the morphism
	$E_1 \times E_2 \lra \P^1 \times \P^1$.
	Let $Z_g$ be the support of the divisor
	$f^*\scg_g$, that is,
	\[
		Z_g
		= \set{(A, B) \in E_1 \times E_2 \Mid
		\phi_2(B) = g\paren{\phi_1(A)}}.
	\]
\end{definition}

The bound in \eqref{eq:uniform-mordell-lang-bound}
applies to $Z = Z_g$ inside
$A = E_1 \times E_2$ only when no translate of a
positive-dimensional abelian subvariety of
$E_1 \times E_2$ is contained in $Z_g$.
Because $Z_g$ is a curve,
such a translate would be a translate of
an elliptic curve.
For a nonconstant rational function $\phi$ on a smooth
projective curve, write $(\phi)_\infty$ for its pole
divisor, and write $m_{\min}(\phi)$ and
$m_{\max}(\phi)$ for the smallest and the largest
coefficient occurring in $(\phi)_\infty$.
The following lemma rules out those translates by
comparing $m_{\min}(\phi_1)$ with $m_{\max}(\phi_2)$.

\begin{lemma}
\label{lem:no-elliptic-coset}
	Let $L$ be a number field,
	$E_1$ and $E_2$ be elliptic curves,
	$\phi_i : E_i \lra \P^1$ be nonconstant for
	$i = 1, 2$, and let $g \in L[X]$ have degree
	$e \geq 1$.
	If $e \cdot m_{\min}(\phi_1) > m_{\max}(\phi_2)$,
	then no irreducible component of $Z_g$ contains a
	translate of a positive-dimensional abelian
	subvariety of $E_1 \times E_2$.
\end{lemma}

\begin{proof}
	Let $T$ be an irreducible component of $Z_g$, and
	suppose that $T$ contains a translate of an abelian
	subvariety $S \subseteq E_1 \times E_2$ with
	$\dim S \geq 1$.
	Since $T$ is a curve, that
	translate must be $T$ itself and $\dim S = 1$.
	Hence $T$ is a translate of an elliptic curve $C$,
	and the two projections of $T$ to $E_1$
	and to $E_2$ have
	the form
	\[
		\begin{aligned}
			b
			&\longmapsto
			P_0 + \alpha(b), \\
			b
			&\longmapsto
			Q_0 + \beta(b),
		\end{aligned}
	\]
	for $b \in C\paren{\overline L}$, where
	$P_0 \in E_1(\overline L)$,
	$Q_0 \in E_2(\overline L)$, and
	$\alpha : C \lra E_1$ and $\beta : C \lra E_2$ are
	homomorphisms. Since $T \subseteq Z_g$, the
	identity
	\begin{equation}
		\phi_2\paren{Q_0 + \beta(b)}
		= g\paren{\phi_1\paren{P_0 + \alpha(b)}}
		\label{eq:translate-identity}
	\end{equation}
	holds for every $b \in C\paren{\overline L}$.

	Neither $\alpha$ nor $\beta$ vanishes. Indeed, if
	$\alpha = 0$, then the right-hand side of
	\eqref{eq:translate-identity} is constant, hence so
	is $\phi_2\paren{Q_0 + \beta(b)}$, and $\beta = 0$
	because $\phi_2$ is nonconstant; if $\beta = 0$,
	then the left-hand side of
	\eqref{eq:translate-identity} is constant, hence so
	is $\phi_1\paren{P_0 + \alpha(b)}$, and
	$\alpha = 0$ because $\phi_1$ is nonconstant.
	In either case $T$ would be a single point rather
	than a curve. Both $\alpha$ and $\beta$ are
	therefore isogenies.

	Write $D_i := (\phi_i)_\infty$, and write $t_P$ for
	translation by $P$. Comparing the pole divisors on
	$C$ of the two sides of
	\eqref{eq:translate-identity} gives
	\[
		\beta^*t_{Q_0}^*D_2
		= e \cdot \alpha^*t_{P_0}^*D_1.
	\]
	Translation permutes the points of a divisor and
	preserves its coefficients, and an isogeny in
	characteristic zero is \'{e}tale, so pullback along
	$\alpha$ and along $\beta$ preserves coefficients
	as well.
	The divisors $\beta^*t_{Q_0}^*D_2$ and
	$\alpha^*t_{P_0}^*D_1$ are nonzero, because $D_1$
	and $D_2$ are nonzero and $\alpha$ and $\beta$ are
	surjective.
	Every coefficient of
	$\beta^*t_{Q_0}^*D_2$ is therefore at most
	$m_{\max}(\phi_2)$, and every coefficient of
	$e \cdot \alpha^*t_{P_0}^*D_1$ is at least
	$e \cdot m_{\min}(\phi_1)$. This contradicts
	$e \cdot m_{\min}(\phi_1) > m_{\max}(\phi_2)$.
\end{proof}

We now take $\phi_1$ and $\phi_2$
to be maps $x_1$ and $x_2$ of degree $2$, respectively,
and we record the values $x_1(A)$
at the points $(A, B) \in Z_g$ whose two coordinates
lie in prescribed subgroups of finite rank.

\begin{definition}
\label{def:elliptic-value-set}
	Let $L$ be a number field. For $i = 1, 2$, let
	$E_i/L$ be an elliptic curve, let
	$x_i : E_i \lra \P^1$ have degree $2$, and let
	$\Gamma_i \subset E_i(\overline L)$ be a subgroup
	of finite rank. Let $g \in L[X]$ have degree
	$e \geq 2$. Define
	\[
		\calp_g(\Gamma_1, \Gamma_2)
		:= \set{a \in \overline L \Mid
		\begin{array}{l}
		a = x_1(A)\text{ for some }A \in \Gamma_1, \\
		g(a) = x_2(B)\text{ for some }B \in \Gamma_2.
		\end{array}}
	\]
\end{definition}

Every $a \in \calp_g(\Gamma_1, \Gamma_2)$ is the first
coordinate of at least one point of
$Z_g \cap \paren{\Gamma_1 \times \Gamma_2}$, and
distinct values of $a$ come from distinct points.
Hence
\begin{equation}
	\#\calp_g(\Gamma_1, \Gamma_2)
	\leq \#\paren{Z_g \cap
	\paren{\Gamma_1 \times \Gamma_2}}.
	\label{eq:elliptic-value-count}
\end{equation}

Bounding the right-hand side of
\eqref{eq:elliptic-value-count} through uniform
Mordell--Lang requires the degree of $Z_g$.
Write $p_i : E_1 \times E_2 \lra E_i$ for the two
projections, and let
$\calm := p_1^*x_1^*\calo_{\P^1}(1)
\otimes p_2^*x_2^*\calo_{\P^1}(1)$
be the product polarization on $E_1 \times E_2$. Let
$H_i$ be the class of $p_i^*x_i^*\calo_{\P^1}(1)$ in
the N\'{e}ron--Severi group of $E_1 \times E_2$.
Then $H_1^2 = H_2^2 = 0$, and $H_1H_2 = 4$ because
$x_1$ and $x_2$ have degree $2$.
The graph $\scg_g$ has class $eh_1 + h_2$, where $h_i$
is the class of a fiber of the $i$-th projection of
$\P^1 \times \P^1$. Hence $f^*\scg_g$ has class
$eH_1 + H_2$, and
\begin{equation}
	\deg_{\calm}\paren{f^*\scg_g}
	= (eH_1 + H_2)(H_1 + H_2)
	= 4(e + 1).
	\label{eq:elliptic-relation-degree}
\end{equation}
The curve $Z_g$ is the support of $f^*\scg_g$, so
$\deg_\calm Z_g \leq 4(e + 1)$.
Degrees with respect to $\calm$ are positive and
additive over irreducible components, so $Z_g$ has at
most $4(e + 1)$ geometric irreducible components and
each of them has degree at most $4(e + 1)$, whether or
not $Z_g$ is irreducible.
We package the uniform Mordell--Lang constants
$c_{\UML}$ from Gao--Ge--K\"{u}hne
\cite[Theorem~1.1]{gao-ge-kuhne}
for all these degrees into a single number.

\begin{definition}
\label{def:uniform-mordell-lang-constant}
	For every integer $e \geq 2$, set
	\[
		\kappa_e
			:= \max\set{2, c_{\UML}(2, j)
			\Mid 1 \leq j \leq 4(e + 1)}.
	\]
\end{definition}

The constant $\kappa_e$ depends only on $e$. The value
$2$ is included so that $\log \kappa_e > 0$ in the
rank lower bounds derived below.

\begin{proposition}
\label{prop:elliptic-mordell-lang-bound}
	Let $L$ be a number field. For $i = 1, 2$, let
	$E_i/L$ be an elliptic curve with origin $O_i$, let
	$x_i : E_i \lra \P^1$ have degree $2$, and let
	$\Gamma_i \subset E_i(\overline L)$ be a subgroup
	of finite rank.
	Let $g \in L[X]$ have degree $e \geq 2$, and let
	$\calp_g(\Gamma_1, \Gamma_2)$ be the set of
	$a \in \overline L$ for which there are
	$A \in \Gamma_1$ and $B \in \Gamma_2$ satisfying
	$a = x_1(A)$ and $g(a) = x_2(B)$.
	Suppose that either $e \geq 3$, or $e = 2$ and
	$(x_1)_\infty = 2[O_1]$. Then
	\[
		\#\calp_g(\Gamma_1, \Gamma_2)
		\leq 4(e + 1)
		\kappa_e^{1 + \rank \Gamma_1 + \rank \Gamma_2}.
	\]
\end{proposition}

\begin{proof}
	We first verify the hypothesis of
	Lemma~\ref{lem:no-elliptic-coset} in both cases.
	A degree-$2$ map has $m_{\max}(x_2) \leq 2$, and
	every coefficient of a pole divisor is at least
	$1$. If $e \geq 3$, then
	$e \cdot m_{\min}(x_1) \geq 3 > 2
	\geq m_{\max}(x_2)$.
	If instead $e = 2$ and $(x_1)_\infty = 2[O_1]$,
	then $m_{\min}(x_1) = 2$, so
	$e \cdot m_{\min}(x_1) = 4 > 2 \geq m_{\max}(x_2)$.
	In both cases, no irreducible component of $Z_g$
	contains a translate of a positive-dimensional
	abelian subvariety of $E_1 \times E_2$.

	The subgroup $\Gamma_1 \times \Gamma_2$ of
	$\paren{E_1 \times E_2}\paren{\overline L}$ has
	finite rank equal to $\rank \Gamma_1 + \rank \Gamma_2$.
	By \eqref{eq:elliptic-relation-degree}, the curve
	$Z_g$ has at most $4(e + 1)$ geometric irreducible
	components, each of degree at most $4(e + 1)$ with
	respect to $\calm$.
	Apply uniform Mordell--Lang
	\cite[Theorem~1.1]{gao-ge-kuhne} to
	$\Gamma_1 \times \Gamma_2$ and to each such
	component. Each component contains at most
	$\kappa_e^{1 + \rank \Gamma_1 + \rank \Gamma_2}$
	points of $\Gamma_1 \times \Gamma_2$, so
	$\#\paren{Z_g \cap
	\paren{\Gamma_1 \times \Gamma_2}}$
	is at most
	$4(e + 1)\kappa_e^{
	1 + \rank \Gamma_1 + \rank \Gamma_2}$.
	Now apply \eqref{eq:elliptic-value-count}.
\end{proof}

Sections~\ref{sec:odd-degree-rank}
and~\ref{sec:fixed-elliptic-fields} use
Proposition~\ref{prop:elliptic-mordell-lang-bound} in
the equivalent form obtained by solving for the two
ranks:
\begin{equation}
	\rank \Gamma_1 + \rank \Gamma_2
	\geq \frac{\log\paren{
	\#\calp_g(\Gamma_1, \Gamma_2)}
	- \log\paren{4(e + 1)}}
	{\log \kappa_e} - 1.
	\label{eq:elliptic-rank-lower}
\end{equation}
A large set of values $\calp_g(\Gamma_1, \Gamma_2)$
therefore forces one of the two Mordell--Weil ranks to
be large.

\begin{remark}
	Garcia-Fritz--Pasten~\cite[Theorem~2.4]{garcia-fritz-pasten}
	also use uniform Mordell--Lang to bound intersections
	$g_1(\Gamma_1) \cap g_2(\Gamma_2)$ for subgroups
	$\Gamma_i \subset E_i(\overline L)$ of finite rank and
	morphisms $g_i : E_i \lra \P^1$ of bounded degree
	with different branch loci.
	For the polynomial relations considered here,
	Lemma~\ref{lem:no-elliptic-coset} instead
	excludes elliptic translates
	by comparing pole multiplicities.
	The degree calculation in
	\eqref{eq:elliptic-relation-degree} then gives the bound
	in Proposition~\ref{prop:elliptic-mordell-lang-bound}
	in the form used in \eqref{eq:elliptic-rank-lower}.
\end{remark}

\subsection{Odd degrees}
\label{sec:odd-degree-rank}

Throughout this subsection, let $K$ be a number field,
let $d \geq 3$ be odd, and choose distinct
$\zeta, \xi \in \mu_d \sm \set{1}$.
Recall the finite subgroup
$\calc_{K, d} \subset
K(\mu_d)^\times / (K(\mu_d)^\times)^{d - 1}$
and its representative set $A$ from
Definition~\ref{def:classes-modulo-powers}.
Since $d - 1$ is even, taking classes modulo squares
maps $\calc_{K, d}$ to
$K(\mu_d)^\times / (K(\mu_d)^\times)^2$.
Distinct elements of $A$ may have the same image.
We retain one representative from $A$ for each image
when choosing the quadratic twists.

\begin{definition}
\label{def:odd-square-classes}
	Choose a subset $A_2 \subseteq A$ containing $1$ and exactly
	one representative of each class represented by $A$
	in $K(\mu_d)^\times / (K(\mu_d)^\times)^2$.
\end{definition}

Thus every $a \in A$ can be written as $a = ub^2$
for a unique $u \in A_2$ and some
$b \in K(\mu_d)^\times$.
The set $A_2$ is finite and depends only on $K, d$,
with $\#A_2 \leq \#A = \#\calc_{K, d}$.
For $d = 3$, the exponent $d - 1$ equals $2$, so
$A_2 = A$.
Moreover, changing a representative by a square
produces an isomorphic quadratic twist over $K(\mu_d)$.

The next lemma places every periodic point on one of
the elliptic curves $E_{Q, R}^{(u)}$
from Definition \ref{def:four-branch-elliptic-curves}.

\begin{lemma}
\label{lem:odd-degree-lift}
	Let $K$ be a number field, let $d \geq 3$ be odd,
	and choose distinct
	$\zeta, \xi \in \mu_d \sm \set{1}$.
	Let $c \in K$, set
	$f_{d, c}(z) := z^d + c$, and let
	$Q, R \in \Per_K(f_{d, c})$ be distinct and
	nonzero.
	For every $P \in \Per_K(f_{d, c})$, there are
	$u \in A_2$ and
	$A \in
	E_{Q, R}^{(u)}\paren{K(\mu_d)}$
	such that $x_{Q, R}^{(u)}(A) = P$.
\end{lemma}

\begin{proof}
	For $T \in \set{Q, R}$, set
	$r_T(P) := (P - \zeta T)/(P - \xi T)$.
	If $P \neq T$, Theorem~\ref{thm:power-classes}
	shows that the square class of $r_T(P)$ is represented
	by an element of $A_2$. If $P = T$, then
	$r_T(T) = (1 - \zeta)/(1 - \xi)$.
	This quotient is a unit away from the places above
	$d$. Therefore its class belongs to
	$\calc_{K, d}$. Its
	square class is therefore represented by an element
	of $A_2$.

	Choose $u \in A_2$ to represent the square
	class of $r_Q(P)r_R(P)$. Since
	$r_Q(P)r_R(P) = F_{Q, R}(P)/
	\paren{(P - \xi Q)^2(P - \xi R)^2}$,
	the value $F_{Q, R}(P)/u$
	is a square in $K(\mu_d)$.
	Thus $P$ is the $X$-coordinate
	of a $K(\mu_d)$-point on
	$E_{Q, R}^{(u)}$.
	Lemmas~\ref{lem:four-branch-points}
	and~\ref{lem:four-branch-elliptic-geometry} show
	that it is an elliptic curve with full rational
	$2$-torsion.
\end{proof}

We write $u(P)$ for a choice of such $u$ from
Lemma~\ref{lem:odd-degree-lift}.
We first use this construction to deduce conditional
uniform boundedness in every fixed odd degree.

\begin{corollary}
\label{cor:conditional-odd-ubc}
	Let $K$ be a number field and let $d \geq 3$ be odd.
	Suppose that there is an integer $r \geq 0$
	such that
	$\rank E_{Q, R}^{(u)}\paren{K(\mu_d)} \leq r$
	for every $c \in K$,
	every pair of distinct nonzero points
	$Q, R \in \Per_K(f_{d, c})$, and every
	$u \in A_2$. Then, for every $c \in K$,
	\[
		\#\Per_K(f_{d, c})
		\leq \max\set{d^{[K:\Q]},
			4(d + 1)
			\paren{\#A_2}^2
			\kappa_d^{1 + 2r}}.
	\]
\end{corollary}

\begin{proof}
	Choose a prime $p \mid d$ and a place
	$\mfp \mid p$ of $K$.
	If $c$ is integral at $\mfp$,
	then apply \cite[Theorem~3]{rajagopal-zhang},
	which says that
	\[
		\#\Per_K(f_{d, c})
		\leq \#k_\mfp
		\leq d^{[K:\Q]}.
	\]

	Now suppose that $c$ is nonintegral at $\mfp$.
	If the periodic set has at most one point,
	the bound is automatically satisfied.
	Otherwise, choose distinct periodic points
	$Q$ and $R$.
	Lemma~\ref{lem:shell} shows that they are nonzero.

	For every periodic point $P$, apply
	Lemma~\ref{lem:odd-degree-lift} to $P$ and to
	$f_{d, c}(P)$. The pair $\paren{u(P),u(f_{d, c}(P))}$
	has at most $\paren{\#A_2}^2$
	possibilities. For each fixed pair, apply
	Proposition~\ref{prop:elliptic-mordell-lang-bound}
	with $g = f_{d, c}$ and with the full Mordell--Weil groups
	of the two corresponding twists. The first
	$X$-coordinate recovers $P$, and hence the assignment
	is injective. Each pair contributes at most
	$4(d + 1)\kappa_d^{1 + 2r}$ points. Summing over the
	pairs proves the claimed bound.
\end{proof}

Without a rank bound, the same estimate gives a lower
bound for rank.

\begin{proof}[Proof of
	Theorem~\ref{thm:elliptic-curves-large-rank}]
	We first apply
	Proposition~\ref{prop:elliptic-mordell-lang-bound} to
	every pair $Q, R$ to obtain a twist of large rank. We
	then fix $Q$ and use the degree-$12$ map to the
	$j$-line to count distinct geometric isomorphism
	classes. Let
	$c \in K$ and suppose that $f_{d, c}$ has a
	$K$-rational cycle $\sco \subset K$ of exact
	length $n \geq 5$.
	At most one point of $\sco$ is zero.

	Choose distinct nonzero points $Q, R \in \sco$.
	For $P \in \sco$, let $u(P)$ be the
	representative supplied by
	Lemma~\ref{lem:odd-degree-lift}.
	Summing
	Proposition~\ref{prop:elliptic-mordell-lang-bound}
	over the possible pairs
	$\paren{u(P), u(f_{d, c}(P))}$ gives
	\[
		n
			\leq 4(d + 1)
			\paren{\#A_2}^2
			\kappa_d^{
			1 + 2\max_{u \in A_2}
			\rank E_{Q, R}^{(u)}\paren{K(\mu_d)}}.
	\]
	Consequently,
	\begin{equation}
		\max_{u \in A_2}
		\rank E_{Q, R}^{(u)}\paren{K(\mu_d)}
			\geq \frac{\log n - \log\paren{
				4(d + 1)\paren{\#A_2}^2}}
				{2 \log \kappa_d} - \frac{1}{2}.
		\label{eq:odd-rank-lower}
	\end{equation}
	This estimate holds for every pair of distinct
	nonzero points in the cycle.

	Fix one nonzero $Q \in \sco$. Apart from $Q$
	and the possible zero point, there are at least
	$n - 2$ choices of $R$. Lemma~\ref{lem:four-branch-j-map}
	shows that at most $12$ values of $R/Q$ give the
	same $j$-invariant. Hence the choices of $R$ produce
	at least $\lceil(n - 2)/12\rceil$ geometric
	isomorphism classes. For every choice of
	$R$, Equation~\eqref{eq:odd-rank-lower} supplies a
	quadratic twist in that class with the required rank
	lower bound.
\end{proof}

Consequently, a long odd-degree cycle forces both
rank growth and variation in moduli. The first
conclusion comes from uniform Mordell--Lang, while the
second comes from the explicit degree of the $j$-map.

The twists can also be untwisted over one fixed
extension of $K(\mu_d)$.
This gives a complementary form of the rank-growth
conclusion.

\begin{corollary}[Rank growth after a fixed extension]
\label{cor:odd-fixed-extension-rank}
	Let $K$ be a number field, and let $d \geq 3$
	be odd.
	There is a finite multiquadratic extension
	$M/K(\mu_d)$, depending only on $K$ and $d$,
	with the following property.
	If $c \in K$ and $f_{d, c}$ has a
	$K$-rational cycle $\sco$ of length $n \geq 3$,
	then, for every pair of distinct nonzero points
	$Q, R \in \sco$,
	\[
		\rank E_{Q, R}(M)
			\geq \frac{\log n - \log\paren{4(d + 1)}}
				{2 \log \kappa_d} - \frac{1}{2}.
	\]
\end{corollary}

\begin{proof}
	Take
	$M := K(\mu_d)\paren{
	\sqrt{u} \Mid u \in A_2}$.
	This is a fixed multiquadratic extension because
	$A_2$ is finite and depends only on
	$K$ and $d$.
	For every $u \in A_2$, the twist
	$E_{Q, R}^{(u)}$ becomes isomorphic to
	$E_{Q, R}$ over $M$ by an isomorphism preserving
	the $X$-coordinate.

	Lemma~\ref{lem:odd-degree-lift} therefore shows
	that every point $P \in \sco$, as well as
	$f_{d, c}(P)$, is the $X$-coordinate of an
	$M$-rational point of $E_{Q, R}$.
	Apply
	Proposition~\ref{prop:elliptic-mordell-lang-bound}
	with $E_{Q, R}$ in both factors and rearrange
	the resulting inequality.
\end{proof}

\begin{remark}
	For an elliptic curve $E/K(\mu_3)$, let
	$\operatorname{Sel}_2\paren{E/K(\mu_3)}$
	denote its $2$-Selmer group.
	When $d = 3$, every twist supplied by
	Theorem~\ref{thm:elliptic-curves-large-rank}
	has full $K(\mu_3)$-rational $2$-torsion.
	The Kummer exact sequence gives
	\[
		\dim_{\F_2}
		\operatorname{Sel}_2\paren{E/K(\mu_3)}
			\geq \rank E\paren{K(\mu_3)} + 2;
	\]
	see Silverman
	\cite[Section~X.4]{silverman-aec}.
	Hence a cycle of length $n$ produces at least
	$\lceil(n - 2)/12\rceil$ geometric isomorphism
	classes containing twists for which the
	dimensions of the $2$-Selmer groups grow at
	least logarithmically with $n$.
\end{remark}

\subsection{The remaining cubic case}
\label{sec:d3-with-roots}

The construction in
Section~\ref{sec:odd-degree-rank}
is defined over $K$ itself when
$\Q(\sqrt{-3}) \subset K$,
which is equivalent to
$K$ containing $\mu_3$.
Consequently, every curve and every
Kummer class in Section~\ref{sec:odd-degree-rank} is defined
over $K$.
We now identify the resulting family and state the precise
conditional bound.

The two nontrivial third roots of unity $\zeta, \xi$ satisfy
$(X - \zeta Q)(X - \xi Q) = X^2 + QX + Q^2$.
Hence, for distinct nonzero periodic points $Q$ and $R$,
the curves defined in
Definition~\ref{def:four-branch-elliptic-curves}
have the form
\begin{equation}
	E_{Q, R}^{(u)}:
	\quad
	uY^2
	= (X^2 + QX + Q^2)(X^2 + RX + R^2).
	\label{eq:cubic-double-cover-family}
\end{equation}
They have full $K$-rational $2$-torsion, and every periodic
point is the $X$-coordinate of a point on one of finitely
many such quadratic twists.

Let $S$ consist of the archimedean places of $K$ and the
finite places above $3$. Since $K(\mu_3) = K$,
Definition~\ref{def:classes-modulo-powers} gives
\[
	\calc_{K, 3}
	= \frac{
	\set{a \in K^\times \Mid
	v(a) \equiv 0 \pmod 2
	\text{ for every }v \notin S}}
	{(K^\times)^2}.
\]
Since $d - 1 = 2$,
we have $A_2 = A$ and
$\#A_2 = \#\calc_{K, 3}$.
If $S_f$ is the finite part of $S$ and
$r_S := r_1(K) + r_2(K) + \#S_f - 1$, then
Lemma~\ref{lem:finite-power-classes} and the $S$-unit
theorem give
\[
	\#\calc_{K, 3}
		= 2^{r_S} \#\frac{\mu(K)}{\mu(K)^2}
			\cdot \#\Cl\paren{\calo_{K, S}}[2].
\]

We use this expression for $\#\calc_{K, 3}$ in the
following conditional bound for $d = 3$.

\begin{corollary}[Conditional uniform boundedness when
	$d = 3$]
\label{cor:conditional-d3}
	Let $K$ be a number field containing
	$\Q(\sqrt{-3})$.
	Suppose that there is an integer $r_K$
	such that $\rank E_{Q, R}^{(u)}(K) \leq r_K$ for every
	elliptic curve in the family
	\eqref{eq:cubic-double-cover-family}. There is a
	constant $B_{\Per}(K, 3)$ such that
	$\#\Per_K(z^3 + c) \leq B_{\Per}(K, 3)$ for every
	$c \in K$. More precisely, one may take
	\[
		B_{\Per}(K, 3)
		= \max\set{
		3^{[K : \Q]},
		16\paren{\#\calc_{K, 3}}^2
		\kappa_3^{1 + 2r_K}}.
	\]
\end{corollary}

\begin{proof}[Proof of
	Corollary~\ref{cor:conditional-d3}]
	Every curve occurring in
	Corollary~\ref{cor:conditional-odd-ubc} belongs to
	the family \eqref{eq:cubic-double-cover-family}.
	Hence its Mordell--Weil rank is at most $r_K$.
	Apply Corollary~\ref{cor:conditional-odd-ubc}
	with $d = 3$ and $r = r_K$;
	since $4(d + 1) = 16$,
	the bound becomes
	\[
		\#\Per_K(z^3 + c)
			\leq \max\set{3^{[K : \Q]},
				16\paren{\#\calc_{K, 3}}^2
				\kappa_3^{1 + 2r_K}},
	\]
	which is the asserted bound.
\end{proof}

This completes the proof of the conditional assertion in
Theorem~\ref{thm:elliptic-curves-large-rank}.


\section{Elliptic curves of large rank over varying fields}
\label{sec:fixed-elliptic-fields}

Theorem~\ref{thm:elliptic-curves-large-rank} varies the
elliptic curve with two points in a rational cycle while
keeping the field fixed.
The complementary problem is to
fix both $f_{d, c}$ and an elliptic curve $E$
together with a degree-$2$ map
$x:E \lra \P^1$,
while allowing the number field to vary with the period.
We lift the points of exact period $n$
through $x$ and bound the rank of the subgroup
they generate on $E$.
The field of definition is obtained by adjoining
these lifts to the field generated by the periodic points.

For $n \geq 1$, write
$\Per_{\overline{K}}(f_{d, c}, n)$ for the set of
finite points of exact period $n$ under $f_{d, c}$,
and write $\kappa_d$ for the uniform Mordell--Lang
constant of
Definition~\ref{def:uniform-mordell-lang-constant}.

\begin{theorem}
\label{thm:fixed-elliptic-curve-rank}
	Let $K$ be a number field, let $d \geq 2$, and fix
	$c \in K$. There exist a finite extension $L/K$,
	an elliptic curve $E/L$ with full $L$-rational
	$2$-torsion, and a degree-$2$ morphism
	$x: E \lra \P^1$ defined over $L$ such that the
	following holds.
	For every $n \geq 1$ with
	$\#\Per_{\overline{K}}(f_{d, c}, n) > 0$
	there is a finite extension $F_n/L$
	and lifts $\widetilde{P} \in E(F_n)$ of each
	$P \in \Per_{\overline K}(f_{d,c},n)$ such that
	\begin{align*}
		\rank\abrac{\widetilde P
			\Mid P\in\Per_{\overline K}(f_{d,c},n)}
			\geq \frac{\log\paren{\#\Per_{\overline{K}}(f_{d, c}, n)}
				- \log\paren{4(d + 1)}}
				{2 \log \kappa_d} - \frac{1}{2}.
	\end{align*}
\end{theorem}

\begin{remark}
	Since $\langle\widetilde P
	\mid P\in\Per_{\overline K}(f_{d,c},n)\rangle$
	is a subgroup of $E(F_n)$,
	Theorem~\ref{thm:fixed-elliptic-curve-rank}
	immediately gives a lower bound on
	$\rank E(F_n)$.
	The fields in
	Theorem~\ref{thm:fixed-elliptic-curve-rank}
	may be chosen with a simple structure. If
	$L_n :=L\paren{\Per_{\overline K}(f_{d,c},n)}$,
	then $F_n/L_n$ may be taken to be multiquadratic,
	obtained by adjoining precisely the square roots needed
	to lift the points of exact period $n$ through $x$.
	In particular, if $E$ is written as
	$Y^2 = h(X)$, then one may take
	\[
		F_n = L_n\paren{\sqrt{h(P)}
			\Mid P\in\Per_{\overline K}(f_{d,c},n)}.
	\]
\end{remark}

\begin{remark}
	Frey--Jarden
	\cite[Theorem~2.2]{frey-jarden}
	showed that every elliptic curve over $\Q$
	has infinite rank over the infinite extension
	$\Q(\sqrt{a} \mid a \in \Q^\times)/\Q$.
	Park--Poonen--Voight--Wood
	\cite[Theorems~12.4.1 and~12.4.2]{park-poonen-voight-wood}
	construct sequences of number fields $M$ of
	unbounded degree, each admitting infinitely many
	$M$-isomorphism classes of elliptic curves of rank
	at least $[M : \Q]/2$.
	Theorem~\ref{thm:fixed-elliptic-curve-rank}
	likewise gives unbounded rank as $n$ grows,
	over finite extensions obtained by adjoining square
	roots to the corresponding dynatomic fields;
	its additional feature is that the subgroup
	of increasing rank is generated explicitly
	by lifts of the periodic points
	of the fixed dynamical system
	$f_{d, c}$ through a fixed map $x: E \lra \P^1$.
\end{remark}

The proof of Theorem~\ref{thm:fixed-elliptic-curve-rank}
separates into two parts according to the degree.
When $d \geq 3$, two fixed points $Q$ and $R$
of $f_{d, c}$ and two distinct nontrivial roots
$\zeta, \xi \in \mu_d$ determine the elliptic
curve.
When $d = 2$, three periodic points
$Q_1, Q_2, Q_3$ determine the elliptic curve.
We begin with the case $d \geq 3$.

\subsection{Rank growth in degrees \texorpdfstring{$d \geq 3$}
	{d >= 3}}
\label{sec:fixed-elliptic-fields-dge3}

Fix a number field $K$, an integer $d \geq 3$, and a
parameter $c \in K$. We first verify that the required
fixed points exist. The polynomial $f_{d, c}(z) - z$ has
degree $d$ and at most one multiple root. Indeed, if
$\alpha$ is a multiple root, then
$d\alpha^{d - 1} = 1$ and
$\alpha^d + c = \alpha$, which give
$\alpha = dc/(d - 1)$. Thus $\alpha$ is uniquely
determined by $c$. Moreover, $d\alpha^{d - 1} = 1$
gives $\alpha \neq 0$, so
$d(d - 1)\alpha^{d - 2} \neq 0$. Hence any multiple
root has multiplicity
exactly two, and $f_{d, c}(z) - z$ has at least
$d - 1 \geq 2$ distinct roots.

If $c \neq 0$, none of these roots is zero. If $c = 0$,
the nonzero fixed points are the elements of
$\mu_{d - 1}$. In either case, $f_{d, c}$ has at least two
distinct nonzero fixed points.

\begin{definition}
\label{def:exact-period-elliptic-curve}
	Choose distinct nonzero fixed points
	$Q, R \in \overline{K}$ of $f_{d, c}$ and distinct roots
	$\zeta, \xi \in \mu_d\sm\set{1}$, and set
	$L_0 := K(\mu_d, Q, R)$.
\end{definition}

For convenience, recall from
Definition~\ref{def:four-branch-elliptic-curves} that
\[
	F_{Q, R}(X) := (X - \zeta Q)(X - \xi Q)
	(X - \zeta R)(X - \xi R)
\]
and that $E_{Q, R}^{(u)}$ is the smooth projective model of
$uY^2 = F_{Q, R}(X)$. When $u = 1$, we write
$E_{Q, R}$. The field $L_0$, the polynomial $F_{Q, R}$,
and the curve $E_{Q, R}$ are all independent of $n$.

We next verify that $F_{Q,R}$ has four distinct roots,
like in Lemma~\ref{lem:four-branch-points},
and that its double cover of $\P^1$
has full rational $2$-torsion
after a ramification point is chosen as the origin.

\begin{lemma}
\label{lem:exact-period-elliptic-geometry}
	Let $K$ be a number field, let $d \geq 3$, and let
	$c \in K$. Choose distinct
	nonzero fixed points $Q, R \in \overline{K}$
	of $f_{d, c}(z)$
	and distinct roots
	$\zeta, \xi \in \mu_d\sm\set{1}$, and set
	$L_0 := K(\mu_d, Q, R)$.
	Then the four points
	$\set{\zeta Q, \xi Q, \zeta R, \xi R}$
	are pairwise distinct.
	Furthermore, for every $u \in L_0^\times$,
	choosing any ramification point as the origin makes
	$E_{Q,R}^{(u)}/L_0$ an elliptic curve with full
	$L_0$-rational $2$-torsion.
\end{lemma}

\begin{proof}
	We first rule out collisions between the two pairs
	$\{\zeta Q, \xi Q\}$ and $\{\zeta R, \xi R\}$.
	Suppose that an element of
	$\set{\zeta Q, \xi Q}$ is equal to an element of
	$\set{\zeta R, \xi R}$. Then $Q^d = R^d$.
	But $Q$ and $R$ are fixed points of $f_{d, c}$,
	so $Q = Q^d + c = R^d + c = R$
	contradicts the assumption that $Q$ and $R$
	are distinct.

	The remaining assertions follow from
	Lemma~\ref{lem:four-branch-elliptic-geometry}.
\end{proof}

For $n \geq 1$, set
$L_n := L_0\paren{\Per_{\overline{K}}(f_{d, c},n)}$,
the field generated over $L_0$ by all points
of exact period $n$;
the extension $L_n/L_0$ is commonly called the
\emph{$n$-th dynatomic field for $f_{d, c}$ over $L_0$}
(see \cite[Section 3.9]{silverman-dynamics}).

Two elements of $L_n^\times$
have the same class
in $L_n^\times / (L_n^\times)^2$
if their quotient belongs to $(L_n^\times)^2$.
For a point $P$ of exact period $n$, the
$X$-coordinate $P$ lifts to an $L_n$-rational point
of $E_{Q, R}^{(u)}$ precisely when
$F_{Q, R}(P)/u$ is a square in $L_n$.
We also note that
$F_{Q, R}(P) \neq 0$ for every periodic point $P$
since the only periodic point in $\mu_d \cdot Q$
(resp. $\mu_d \cdot R$) is $Q$ (resp. $R$)
by Lemma~\ref{lem:periodic-orbits}.
Thus the class of $F_{Q, R}(P)$
in $L_n^\times / (L_n^\times)^2$ determines
which quadratic twist $E_{Q,R}^{(u)}$ receives the point $P$.

\begin{definition}
\label{def:exact-period-data}
	Suppose that
	$\#\Per_{\overline{K}}(f_{d, c}, n) > 0$.
	Define
	\[
		\sct_n^{Q, R} := \set{
			\sbrac{F_{Q, R}(P)} \in L_n^\times/(L_n^\times)^2
			\Mid P \in \Per_{\overline{K}}(f_{d, c}, n)}.
	\]
	For every $\tau \in \sct_n^{Q, R}$,
	choose a representative $u_\tau \in L_n^\times$.
\end{definition}

Changing $u_\tau$ by a square in $(L_n^\times)^2$
gives a new quadratic twist that is isomorphic
to the original $E_{Q, R}^{(u_\tau)}$ over $L_n$.
Consequently, the maximal rank of
$E_{Q, R}^{(u_\tau)}(L_n)$ as
$\tau$ varies in $\sct_n^{Q, R}$
is independent of the chosen representatives.
Moreover, $\#\sct_n^{Q, R} = 1$ precisely when
all the values $F_{Q, R}(P)$ represent the same
square class.
We continue to write $\kappa_d$ for the uniform
Mordell--Lang constant of
Definition~\ref{def:uniform-mordell-lang-constant}.

The next theorem bounds the number of points of exact
period $n$ in terms of the square classes represented
by the values $F_{Q, R}(P)$ and the Mordell--Weil
ranks of the corresponding quadratic twists.

\begin{theorem}
\label{thm:exact-period-rank-twists}
	Let $K$ be a number field, let $d \geq 3$, and let
	$c \in K$.
	Choose distinct nonzero fixed points $Q, R \in \overline{K}$
	of $f_{d, c}$ and distinct roots
	$\zeta, \xi \in \mu_d\sm\set{1}$, and set
	$L_0 := K(\mu_d, Q, R)$.
	For $n \geq 1$ such that
	$\#\Per_{\overline{K}}(f_{d, c}, n) > 0$,
	\[
		\#\Per_{\overline{K}}(f_{d, c}, n)
			\leq 4(d + 1)
			\paren{\#\sct_n^{Q, R}}^2
			\kappa_d^{
			1 + 2\max_{\tau \in \sct_n^{Q, R}}
			\rank E_{Q, R}^{(u_\tau)}(L_n)}.
	\]
	Equivalently,
	\[
		\max_{\tau \in \sct_n^{Q, R}}
			\rank E_{Q, R}^{(u_\tau)}(L_n)
			\geq \frac{\log\paren{
				\#\Per_{\overline{K}}(f_{d, c}, n)}
				- \log\paren{4(d + 1)
				\paren{\#\sct_n^{Q, R}}^2}}
				{2 \log \kappa_d} - \frac{1}{2}.
	\]
\end{theorem}

\begin{proof}
	By Lemma~\ref{lem:periodic-orbits}, the only
	periodic point in $\mu_d Q$ is $Q$, and the only
	periodic point in $\mu_d R$ is $R$.
	Since $Q, R \neq 0$ and $\zeta, \xi \neq 1$,
	none of $\zeta Q$, $\xi Q$, $\zeta R$, and
	$\xi R$ is periodic. Hence $F_{Q, R}(P) \neq 0$
	for every $P \in \Per_{\overline{K}}(f_{d, c}, n)$.

	If $\sbrac{F_{Q, R}(P)} = \tau$, then the
	choice of $u_\tau$ shows that $P$ is the
	$X$-coordinate of an $L_n$-point on
	$E_{Q, R}^{(u_\tau)}$.
	Group the points $P$ according to the ordered pair
	\[
		\big(\sbrac{F_{Q, R}(P)},
		\sbrac{F_{Q, R}(f_{d, c}(P))} \big)
		\in
		\sct_n^{Q, R}
		\times \sct_n^{Q, R}.
	\]
	There are at most
	$\paren{\#\sct_n^{Q, R}}^2$
	such pairs.

	For each fixed pair, apply
	Proposition~\ref{prop:elliptic-mordell-lang-bound}
	with $g = f_{d, c}$ and with the full
	Mordell--Weil groups of the two corresponding
	twists.
	Summing over the represented pairs gives
	the desired upper bound
	on $\#\Per_{\overline{K}}(f_{d, c}, n)$.
	Rearranging that inequality gives the
	desired lower bound on
	$\max_{\tau \in \sct_n^{Q, R}}
	\rank E_{Q, R}^{(u_\tau)}(L_n)$.
\end{proof}

Theorem~\ref{thm:exact-period-rank-twists}
identifies two possible reasons why there
can be many points of exact period $n$:
either the values $F_{Q, R}(P)$
represent many square classes
in $\sct_n^{Q, R}$,
or one of the corresponding twists has large
Mordell--Weil rank.
We now remove the first possibility
by adjoining all the relevant square roots.

\begin{definition}
\label{def:exact-period-lifts}
	Suppose that $n \geq 1$ and
	$\#\Per_{\overline{K}}(f_{d, c},n) > 0$.
	Define the field extension
	\[
		F_n := L_n\paren{
		\sqrt{F_{Q, R}(P)} \Mid
		P \in \Per_{\overline{K}}(f_{d, c}, n)}.
	\]
\end{definition}

By construction, $F_{Q, R}(P)$ is a square in $F_n$
for every point $P$ of exact period $n$.
Consequently, both $P$ and $f_{d, c}(P)$ are
$X$-coordinates of $F_n$-rational points on
$E_{Q, R}$.
We may therefore use $E_{Q, R}$ in both factors
of the Mordell--Lang argument;
this gives the following rank bound.
For each $P \in \Per_{\overline K}(f_{d,c},n)$,
choose $y_P \in F_n$ with $y_P^2 = F_{Q,R}(P)$,
and let $\Gamma_n \leq E_{Q,R}(F_n)$ be the
subgroup generated by the points $(P, y_P)$.

\begin{corollary}
\label{cor:fixed-elliptic-rank}
	Let $K$ be a number field, let $d \geq 3$, and let
	$c \in K$.
	Choose distinct nonzero fixed points
	$Q, R \in \overline{K}$ of $f_{d, c}$
	and distinct roots $\zeta, \xi \in \mu_d\sm\set{1}$,
	and set $L_0 := K(\mu_d, Q, R)$.
	Let $n \geq 1$ such that
	$\#\Per_{\overline{K}}(f_{d, c}, n) > 0$
	and let $F_n$ be the field in
	Definition~\ref{def:exact-period-lifts}.
	Then
	\[
		\rank \Gamma_n
			\geq \frac{ \log\paren{\#\Per_{\overline{K}}(f_{d, c},n)}
				- \log\paren{4(d + 1)}}
				{2 \log \kappa_d} - \frac{1}{2}.
	\]
\end{corollary}

\begin{proof}
	Since $f_{d, c}$ permutes the points
	of exact period $n$,
	both $P$ and $f_{d, c}(P)$ are $X$-coordinates
	of points in $\Gamma_n$
	for every point $P$ of exact period $n$.
	Apply
	Proposition~\ref{prop:elliptic-mordell-lang-bound}
	with $g = f_{d,c}$,
	$E_1 = E_2 = E_{Q,R}$, and
	$\Gamma_1 = \Gamma_2 = \Gamma_n$. It gives
	\[
		\#\Per_{\overline K}(f_{d,c},n)
		\leq4(d+1)\kappa_d^{1+2\rank\Gamma_n}.
	\]
	Taking logarithms and rearranging proves
	the lower bound for $\rank \Gamma_n$.
\end{proof}

\begin{proof}[Proof of
	Theorem~\ref{thm:fixed-elliptic-curve-rank} when
	$d \geq 3$]
	Let $L_0$ be as in
	Definition~\ref{def:exact-period-elliptic-curve}, and
	set $L := L_0$ and $E := E_{Q, R}$,
	and let $x: E \lra \P^1$ be the $X$-coordinate map.
	Choose a ramification point of $x$ as the origin on $E$.
	Lemma~\ref{lem:exact-period-elliptic-geometry}
	shows that $E/L$ has full $L$-rational $2$-torsion.
	For each $n$ with
	$\#\Per_{\overline{K}}(f_{d, c},n) > 0$,
	the field
	generated over $L_n$ by the chosen lifts is
	exactly the field $F_n$ of
	Definition~\ref{def:exact-period-lifts}.
	This extension is multiquadratic, and
	Corollary~\ref{cor:fixed-elliptic-rank}
	gives the lower bound for $\rank \Gamma_n$.
\end{proof}

The degree of $F_n/L_n$ is determined by the
classes in $L_n^\times / (L_n^\times)^2$
represented by its generators.
If all the corresponding twists
have bounded rank over $L_n$, then
Theorem~\ref{thm:exact-period-rank-twists} forces
many distinct classes in
$L_n^\times / (L_n^\times)^2$
and hence a large extension.

\begin{corollary}
\label{cor:auxiliary-extension-degree}
	Let $d \geq 3$ and $r_0 > 0$.
	For every $n \geq 1$ such that
	$\#\Per_{\overline{K}}(f_{d, c}, n) > 0$.
	and $\rank E_{Q, R}^{(u_\tau)}(L_n) \leq r_0$
	for every $\tau\in\sct_n^{Q,R}$,
	\[
		[F_n : L_n]
			\geq \sqrt{\frac{\#\Per_{\overline K}(f_{d, c}, n)}
				{4(d + 1)\kappa_d^{1 + 2r_0}}}.
	\]
\end{corollary}

\begin{proof}
	The degree $[F_n : L_n]$ is the order of the subgroup
	of $L_n^\times / (L_n^\times)^2$ generated by
	$\sct_n^{Q,R}$. In particular,
	$[F_n : L_n] \geq \#\sct_n^{Q,R}$.
	Theorem~\ref{thm:exact-period-rank-twists} gives
	\[
		\#\Per_{\overline K}(f_{d, c}, n)
		\leq 4(d + 1)\paren{\#\sct_n^{Q, R}}^2
			\kappa_d^{1+2r_0},
	\]
	and the result follows.
\end{proof}

\subsection{Rank growth in degree \texorpdfstring{$d = 2$}
	{d = 2}}
\label{sec:quadratic-dynatomic-rank}

Suppose now that $d = 2$.
Since $-1$ is the only
nontrivial element of $\mu_2$,
the four-point construction $E_{Q, R}^{(u)}$ above is unavailable.
We replace it by degree-$2$ cover of $\P^1$ branched
at three distinct points and infinity.
The point above infinity will be the origin $O$,
so the coordinate map will satisfy $(x)_\infty = 2[O]$
as required by
Proposition~\ref{prop:elliptic-mordell-lang-bound}
when $d=2$.

A quadratic polynomial has infinitely many periodic
points over an algebraic closure; see
Silverman~\cite[Section~3.9]{silverman-dynamics}.
Thus we may choose three distinct periodic points
$Q_1, Q_2, Q_3$ of $f_{2, c}$
independently of $n$.

\begin{definition}
\label{def:quadratic-fixed-elliptic-curve}
	Choose distinct nonzero periodic points
	$Q_1, Q_2, Q_3 \in \overline{K}$ of $f_{2, c}$,
	and set $L_0 := K(Q_1, Q_2, Q_3)$.
	Let $E/L_0$ be the smooth
	projective curve with affine model
	\[
		E:
		\quad
		Y^2 = (X + Q_1)(X + Q_2)(X + Q_3),
	\]
	and let $x:E \lra \P^1$
	be the morphism induced by
	projection to the $X$-coordinate.
\end{definition}

The three finite branch points are distinct and
$L_0$-rational. The point at infinity gives a fourth
branch point and a natural origin
for an elliptic curve structure.

\begin{lemma}
\label{lem:quadratic-fixed-elliptic-geometry}
	The curve $E/L_0$ is an elliptic curve with full
	$L_0$-rational $2$-torsion, and
	$(x)_\infty = 2[O]$, where $O$ is the point at
	infinity.
\end{lemma}

\begin{proof}
	We apply Riemann--Hurwitz and
	identify the ramification points.
	The curve $E$ is a double cover of $\P^1$
	branched at $-Q_1, -Q_2, -Q_3$, and infinity.
	It therefore has genus $1$.
	Taking the point at infinity as the origin makes
	the other three ramification points
	the nonzero points of order two.
	The function $x$ has its unique pole at $O$
	with order two.
\end{proof}

The next fields make every point of exact period $n$
lift to this elliptic curve $E$.

\begin{definition}
\label{def:quadratic-dynatomic-lifts}
	Suppose that $n \geq 1$ and
	$\#\Per_{\overline{K}}(f_{2, c}, n) > 0$.
	Set
	$L_n := L_0\paren{
	\Per_{\overline{K}}(f_{2, c}, n)}$,
	and define
	\[
		F_n
			:= L_n\paren{
			\sqrt{(P + Q_1)(P + Q_2)(P + Q_3)}
			\Mid
			P \in \Per_{\overline{K}}(f_{2, c}, n)}.
	\]
\end{definition}

The field $L_n$ is the $n$-th dynatomic field for
$f_{2, c}$ over the field
$L_0 = K(Q_1, Q_2, Q_3)$ fixed in
Definition~\ref{def:quadratic-fixed-elliptic-curve}.
The extension $F_n/L_n$ is multiquadratic.
Every radicand in the definition of $F_n$ is
nonzero.
Indeed, if $P = -Q_i$, then
$f_{2, c}(P) = f_{2, c}(Q_i)$.
Since both $P$ and $Q_i$ are periodic, applying
a sufficiently high iterate gives $P = Q_i$,
contrary to $Q_i \neq 0$.

\begin{proof}[Proof of
	Theorem~\ref{thm:fixed-elliptic-curve-rank} when
	$d=2$]
	Take $L := L_0$ and let $E$ and $x$ be as in
	Definition~\ref{def:quadratic-fixed-elliptic-curve}.
	By Lemma~\ref{lem:quadratic-fixed-elliptic-geometry},
	$E/L$ has full $L$-rational $2$-torsion and
	$(x)_\infty = 2[O]$.
	For $n \geq 1$ with
	$\#\Per_{\overline K}(f_{2, c}, n) > 0$, the field
	generated over $L_n$ by the chosen lifts is
	exactly the multiquadratic extension $F_n$ of
	Definition~\ref{def:quadratic-dynatomic-lifts}.
	Write these lifts as $(P, y_P)$, where
	\[
		y_P^2 = (P + Q_1)(P + Q_2)(P + Q_3),
	\]
	and let $\Gamma_n \subseteq E(F_n)$ be the subgroup
	they generate.

	Since $f_{2, c}$ permutes
	the points of exact period $n$,
	both $P$ and $f_{2, c}(P)$ are $X$-coordinates
	of points in $\Gamma_n$ for every such $P$.
	Apply
	Proposition~\ref{prop:elliptic-mordell-lang-bound}
	with $g = f_{2,c}$, $E_1 = E_2 = E$, and
	$\Gamma_1 = \Gamma_2 = \Gamma_n$ to get
	\[
		\#\Per_{\overline K}(f_{2, c}, n)
			\leq 12\kappa_2^{1 + 2 \rank \Gamma_n}.
	\]
	Taking logarithms and rearranging proves the
	lower bound on $\rank \Gamma_n$.
\end{proof}

\subsection{Rank growth with the period}
\label{sec:period-rank-growth}

The constructions above prove
Theorem~\ref{thm:fixed-elliptic-curve-rank} for every
$d \geq 2$, with both $E$ and $x$ independent of $n$.
Let
\[
	\delta_d(n) := \sum_{m \mid n} \mu(n / m) d^m,
\]
where $\mu$ is the M\"{o}bius function;
this is the degree of the $n$-th dynatomic polynomial
(see \cite[Remark 4.3]{silverman-dynamics}).
Whenever its roots are simple and have exact period $n$,
\[
	\#\Per_{\overline K}(f_{d,c}, n) = \delta_d(n).
\]
Since $\delta_d(n) \asymp_d d^n$,
Theorem~\ref{thm:fixed-elliptic-curve-rank} gives
\begin{equation}
	\label{eq:linear-exact-period-rank}
	\rank\Gamma_n
		\geq \frac{\log d}{2\log\kappa_d}n - O_d(1).
\end{equation}
More generally, the same estimate holds along any
sequence of periods satisfying
\[
	\#\Per_{\overline K}(f_{d,c},n)\gg_d d^n.
\]
It also holds for $\rank E(F_n)$,
since $\Gamma_n \leq E(F_n)$.

\begin{remark}
	The fields $F_n$ need not be nested.
	Let $F_{\leq n}$ be the compositum of the fields
	$F_j$ with $1 \leq j \leq n$ and
	$\#\Per_{\overline K}(f_{d, c}, j)>0$.
	These fields form an increasing tower.
	Since $\Gamma_n\subseteq E(F_{\leq n})$,
	the lower bound in
	\eqref{eq:linear-exact-period-rank}
	also holds for $\rank E(F_{\leq n})$
	whenever its hypothesis on the number of points
	of exact period $n$ holds.
\end{remark}

\subsection{An explicit example using iterated preimages}
\label{sec:iterated-splitting-rank}

The preceding construction adjoins square roots to the
fields generated by periodic points.
We conclude this section with 
an explicit example of rank growth
in which the required square roots
already lie in a tower
of splitting fields of iterates.

We use the preimages of $0$ under
$f_{2, i}(z) = z^2 + i$.
Consider the elliptic curve
$E : Y^2 = (X^2 + 1)(X + 1 - i)$ over $\Q(i)$.
Observe that $E$ is non-CM since
$j(E) = (-10048 - 11136i)/25 \notin \Z[i]$.

For each integer $m \geq 0$,
let $K_m := \Q(i)\paren{f_{2, i}^{-m}(0)}$ be
the splitting field of $f_{2, i}^m$ over $\Q(i)$
inside $\overline{\Q}$.
Then we have a tower of number fields
with $K_0 = \Q(i)$ and
$K_m \subseteq K_{m + 1}$.
For $j \geq 1$, the value $f_{2, i}^j(0)$ belongs
to $\set{i, i - 1, -i}$ and is nonzero modulo every
nonzero prime ideal $\mfp \neq (1 + i)$ of $\Z[i]$.
The chain rule therefore shows that $f_{2, i}^m$
has separable reduction modulo $\mfp$,
so $K_m/\Q(i)$ is unramified outside $(1 + i)$.
The following bound gives linear rank growth over $K_m$.

\begin{proposition}
	\label{prop:iterated-splitting-rank}
	For every integer $m \geq 3$,
	\[
		\rank E(K_m)
		\geq \frac{m\log 2 - \log(256\cdot 10^{13})}
		{2\log\paren{1 + \frac{5}{4\sqrt{2}}}}.
	\]
\end{proposition}

\begin{proof}
	We count points on a genus-two curve whose Jacobian
	is isogenous to $E^2$ and apply the
	explicit point bound of
	Yu--Yuan--Zhou \cite[Theorem~1.3]{yu-yuan-zhou}.
	Let $C/\Q(i)$ be the smooth projective curve with
	affine model
	\[
		C : \quad Y^2 =
		(T^2 + 1)(T^2 + 1 - 2i)(T^2 + 2 - 2i).
	\]
	The sextic is square-free, so $C$ has genus two.

	Fix $m \geq 3$.
	Since $f_{2, i}^{m - 2}$ is even,
	for each $a \in f_{2, i}^{-(m - 2)}(0)$ we may choose
	$u, v \in K_{m - 1}$ with $u^2 = a - i$ and
	$v^2 = -a - i$. The preimages of $\pm u$ and
	$\pm v$ under $f_{2, i}$ lie in $K_m$, so there are
	$t, y \in K_m$ satisfying
	\begin{align*}
		t^2 & = (v - i)(-v - i) = a + i - 1, \\
		y^2 & = u^2v^2(u - i)(-u - i)
			= (a^2 + 1)(a + 1 - i).
	\end{align*}
	Substitution shows that $(t, y) \in C(K_m)$.
	None of $i$, $-i$, and $\pm(1 - i)$ maps to $0$
	under an iterate of $f_{2, i}$, so $ty \neq 0$.
	Thus $(t, y)$, $(t, -y)$, $(-t, y)$, and $(-t, -y)$
	are four distinct points of $C(K_m)$ satisfying
	$T^2 + 1 - i = a$. The sets of points corresponding
	to distinct values of $a$ are therefore disjoint.
	Summing over the $2^{m - 2}$ distinct roots of
	$f_{2, i}^{m - 2}$ yields
	$\#C(K_m) \geq 4\cdot 2^{m - 2} = 2^m$.

	To relate $J(C)$ to $E$, consider the maps
	\begin{align*}
		(T, Y) &\longmapsto (T^2 + 1 - i, Y), \\
		(T, Y) &\longmapsto
			\paren{f_{2, i}(T^2 + 1 - i),
			(T^2 + 1 - i)TY}.
	\end{align*}
	Substituting each image into the equation of
	$E$ and using the equation of $C$ verifies that both
	maps take values in $E$. The maps extend to
	morphisms $C \lra E$.
	The pullbacks of $dX/Y$ are $2T\,dT/Y$ and
	$4\,dT/Y$, respectively, and form a basis of the
	regular differentials on $C$. Hence the induced
	homomorphism $J(C) \lra E^2$ is an isogeny over
	$\Q(i)$, so $\rank J(C)(K_m) = 2\rank E(K_m)$.
	Yu--Yuan--Zhou \cite[Theorem~1.3]{yu-yuan-zhou} give
	\[
		2^m \leq \#C(K_m)
		\leq 256\cdot 10^{13}
		\paren{1 + \frac{5}{4\sqrt{2}}}^{
		2\rank E(K_m)}.
	\]
	Taking logarithms of both sides yields
	the desired lower bound on $\rank E(K_m)$.
\end{proof}


\section{Jacobians of large rank over
	\texorpdfstring{$K(\mu_d)$}{K(mu\_d)}}
\label{sec:higher-jacobians}

This section gives two
different constructions of Jacobians
using periodic points of $f_{d, c}$.
The first applies in every degree $d \geq 2$.
From a $K$-rational point $P$ of exact period $n$
under $f_{d, c}$,
we obtain a point
$x_P \in X_{1, d}(n)(K)$
and a different point
$s_P \in X_{1, d}(n)(K(\mu_d))$
on the smooth projective dynatomic curve;
under the map $X_{1, d}(n) \lra \P^1$ induced
by the $c$-coordinate, the points $x_P$ and $s_P$
lie over $c$ and $\infty$, respectively.
The divisor $x_P - s_P$ defines a class
$\delta_P := [x_P - s_P]$ in
$J_{1, d}(n)(K(\mu_d))$.
We show that the images of $\delta_P$ under the order-$n$
action induced by iteration of $f_{d, c}$
generate a torsion-free subgroup
$\Lambda_P$ of $J_{1, d}(n)(K(\mu_d))$.
The gonality of $X_{1, d}(n)$ and the cyclotomic
decomposition of $\Lambda_P \otimes_{\Z} \Q$
then give
Theorem~\ref{thm:unicritical-dynatomic-rank}.

The second construction applies when $d \geq 4$
and uses degree-$(d - 1)$ Galois covers of
$\P^1$ with cyclic deck group.
After one fixed extension of $K$,
a long rational cycle gives many rational points
on curves of genus $d - 2$.
Combining the geometry of these curves with the
quantitative rational point bound of Yu--Yuan--Zhou
\cite{yu-yuan-zhou} gives
Theorem~\ref{thm:jacobians-large-rank},
which can be viewed as an analogue of
Theorem~\ref{thm:elliptic-curves-large-rank}
for certain principally polarized Jacobians.

\subsection{Dynatomic curves}
\label{sec:dynatomic-curves}

We now recall the properties of dynatomic curves
needed to prove
Theorem~\ref{thm:unicritical-dynatomic-rank}.

Fix $d \geq 2$.
For $n \geq 1$, let
$Y_{1, d}(n) \subset \A^2$
be the affine curve defined by the $n$-th dynatomic
polynomial for the family
$f_{d, c}(z) = z^d + c$, as in Silverman
\cite[Sections~4.1--4.2]{silverman-dynamics}.
A point $(c, P)$ of $Y_{1, d}(n)$ records
a point $P$ of formal period $n$ under $f_{d, c}$.
Let $X_{1, d}(n)$ be the normalization of the
projective closure of $Y_{1, d}(n)$,
and let $J_{1, d}(n)$ be its Jacobian.
The projection to the parameter $c$ extends to
a morphism
$\pi_n : X_{1, d}(n) \lra \P^1$.

Every point of exact period $n$ lies on the
affine dynatomic curve $Y_{1, d}(n)$.
The converse can fail
at special parameters: a point of exact period
$m < n$ can lie on this curve when $m \mid n$
and the multiplier of its cycle has order
$n/m$;
see \cite[Section 4.2.1]{silverman-dynamics}
and \cite[Theorem~2.1]{gao-ou-dynatomic}.
For example, $f_{2, -3/4}$ fixes $-1/2$
with multiplier $-1$, so
$(-3/4, -1/2)$ lies on the second dynatomic
curve although $-1/2$ has exact period $1$.

For $d = 2$, the smoothness of
the affine dynatomic curve $Y_{1, 2}(n)$
is due to Douady--Hubbard
\cite{douady-hubbard-dynamics-i, douady-hubbard-dynamics-ii},
while its irreducibility was proved independently
by Bousch~\cite{bousch-dynamics} and
Lau--Schleicher~\cite{lau-schleicher-dynatomic}.
Gao--Ou
\cite[Theorems~1.1 and~1.2]{gao-ou-dynatomic}
prove that $Y_{1, d}(n)$ is smooth and
geometrically irreducible for every $d \geq 2$;
in particular, the projective curve
$X_{1, d}(n)$ is geometrically irreducible.
Doyle--Krieger--Obus--Pries--Rubinstein-Salzedo--West
\cite[Section~3.2]{dkoprw}
record the cyclic automorphism
$\sigma_n(c, z) := (c, f_{d, c}(z))$
of $Y_{1, d}(n)$.
This automorphism extends uniquely to an
automorphism of order $n$ on $X_{1, d}(n)$.
Since $\sigma_n$ does not change $c$, one has
$\pi_n \circ \sigma_n = \pi_n$.
Doyle--Poonen \cite[Theorem~1.1]{doyle-poonen}
prove that the geometric gonality of
$X_{1, d}(n)$ tends to infinity with $n$.

The automorphism $\sigma_n$ also acts
on $J_{1, d}(n)$.
We use the group-algebra decomposition of an
abelian variety with a finite group action in
the form given by Lange--Recillas
\cite[Section~1]{lange-recillas-group-action}.

For each divisor $m \mid n$,
let $e_{n, m}$ be the idempotent
corresponding to $\Q(\zeta_m)$ under
\[
	\Q[\abrac{\sigma_n}]
		\simeq \prod_{r \mid n}\Q(\zeta_r).
\]
Choose an integer $q_{n, m} \geq 1$ such that
$q_{n, m}e_{n, m}
\in \Z[\abrac{\sigma_n}]$, and let
$A_{n, m}$ be the image of the endomorphism
$q_{n, m}e_{n, m}$ of $J_{1, d}(n)$.
Lange--Recillas
\cite[Section~1]{lange-recillas-group-action}
show that $A_{n, m}$ is an abelian subvariety
independent of the choice of $q_{n, m}$.
Since $\sigma_n$ is defined over $\Q$,
the endomorphism $q_{n, m}e_{n, m}$ and the
subvariety $A_{n, m}$ are defined over $\Q$.

The following proposition identifies
a cyclotomic factor of
the dynatomic Jacobian $J_{1, d}(n)$
of large rank,
and is a stronger version
of Theorem~\ref{thm:unicritical-dynatomic-rank}.
Write $\varphi$
for Euler's totient function

\begin{proposition}
\label{prop:unicritical-cyclotomic-rank}
	Let $K$ be a number field. There is an integer
	$N_{K, d} \geq 1$ with the following property. If
	$c \in K$ and $f_{d, c}$ has a $K$-rational point of exact
	period $n > N_{K, d}$, then, for every prime power
	divisor $\ell^a \mid n$, there is a divisor $m \mid n$
	with $\ell^a \mid m$ such that
	\[
		\rank A_{n, m}\paren{K(\mu_d)}
		\geq \varphi\paren{m}
		\geq \varphi\paren{\ell^a}.
	\]

	Furthermore, if there is an integer $r_{K, d}$
	such that the rank of
	$A_{n, m}\paren{K(\mu_d)}$ is at most
	$r_{K, d}$ for every $n \geq 1$ and every
	$m \mid n$,
	then there is a constant $B_{\Per}(K, d)$
	that is effectively computable from $K$, $d$, and
	$r_{K, d}$ such that
	\[
		\#\Per_K(z^d + c) \leq B_{\Per}(K, d)
	\]
	for every $c \in K$.
\end{proposition}

To prove
Proposition~\ref{prop:unicritical-cyclotomic-rank},
we compare the point
$x_P \in X_{1, d}(n)(K)$ corresponding to the
pair $(c, P)$, where $P$ has exact period $n$,
with a point
$s_P \in X_{1, d}(n)(K(\mu_d))$ satisfying
$\pi_n(s_P) = \infty$.
We briefly recall that
Morton \cite{morton-curves} constructs points in
$\pi_n^{-1}(\infty)$ by Laurent series and that
Doyle--Krieger--Obus--Pries--Rubinstein-Salzedo--West
\cite{dkoprw} provide a mixed-characteristic
model in which the reductions of $x_P$ and $s_P$
can be compared.

Set $F_u(z) := z^d - u^d$, so that
$F_u = f_{d, -u^d}$.
Let
$\boldsymbol{\zeta}
:= (\zeta_j)_{j \in \Z/n\Z}$
be a sequence in $\mu_d$ such that no integer
$r$ with $0 < r < n$ satisfies
$\zeta_{j + r} = \zeta_j$ for every $j$.
Morton \cite[Lemmas~1 and~2]{morton-curves}
constructs a unique Laurent series
$z_{\boldsymbol{\zeta}}(u)$ of exact period $n$
under $F_u$ whose $j$-th iterate has leading
term $\zeta_j u$.
Morton \cite[Proposition~10]{morton-curves}
then associates to $z_{\boldsymbol{\zeta}}(u)$
a point
$s_{\boldsymbol{\zeta}}
\in X_{1, d}(n)(K(\mu_d))$
satisfying
$\pi_n(s_{\boldsymbol{\zeta}}) = \infty$.
Doyle--Krieger--Obus--Pries--Rubinstein-Salzedo--West
\cite[Remark~3.9, Definition~3.10, and
	Proposition~6.4]{dkoprw}
realize $s_{\boldsymbol{\zeta}}$ on a
mixed-characteristic model and prove that
$s_{\boldsymbol{\zeta}}$ is smooth in every
special fiber whose characteristic does not
divide $d$.

The other ingredient that we use to
prove Proposition~\ref{prop:unicritical-cyclotomic-rank}
is the construction of a sequence
$\boldsymbol{\epsilon} := (\epsilon_j)_{j \in \Z/n\Z}$
from the orbit of $P$ and
a local comparison of $x_P$
and $s_{\boldsymbol{\epsilon}}$
after reduction modulo a prime.
Let $K$ be a number field.
Fix a rational prime $p \geq 5$ such that
$p \nmid d$ and $p$ splits completely in
$K(\mu_d)$, and fix a place
$v \mid p$ of $K(\mu_d)$.

\begin{lemma}
\label{lem:unicritical-dynatomic-disc}
	Suppose that $c \in K$ satisfies $v(c) < 0$,
	let $P \in \Per_K(f_{d, c}, n)$,
	and let $x_P \in X_{1, d}(n)(K)$ be the point
	corresponding to the pair $(c, P)$.
	Then there is a point
	$s_P \in X_{1, d}(n)(K(\mu_d))$ satisfying
	$\pi_n(s_P) = \infty$ such that
	$\sigma_n^j x_P$ and $\sigma_n^j s_P$ reduce
	to the same smooth point modulo $v$ for every
	$j \in \Z$.
\end{lemma}

\begin{proof}
	The proof has three steps.
	First, we associate to $P$ a sequence
	$\boldsymbol{\epsilon} := (\epsilon_j)_{j \in \Z/n\Z}$
	in $\mu_d$ and prove that no nontrivial cyclic
	shift fixes this sequence.
	Second, we evaluate Morton's Laurent series
	$z_{\boldsymbol{\epsilon}}(u)$ at $u = \beta$ and prove that
	the resulting periodic point is $P$.
	Finally, the formal section determined by
	$z_{\boldsymbol{\epsilon}}(u)$ specializes to $x_P$ at
	$u = \beta$ and to $s_P$ at $u = \infty$.

	\emph{Step 1: the sequence $\boldsymbol{\epsilon}
		:= (\epsilon_j)_{j \in \Z/n\Z}$.}
	Write $P_j := f_{d, c}^j(P)$,
	with indices $j \in \Z/n\Z$.
	Lemma~\ref{lem:shell} gives an integer $s \geq 1$
	such that $v(c) = -ds$ and
	$v(P_j) = -s$ for every $j$.
	Since $K(\mu_d)_v = \Q_p$ and $p \nmid d$,
	Lemma~\ref{lem:tame-normalization} gives
	$\beta \in \Q_p$ such that $\beta^d = -c$.

	Set $t_0 := \beta^{-1}$ and
	$y_j := t_0P_j$.
	Thus $t_0 \in p\Z_p$, and every $y_j$ is a unit.
	Substituting $P_j = t_0^{-1}y_j$ and
	$c = -t_0^{-d}$ into
	$P_{j + 1} = P_j^d + c$ gives
	$t_0^{d - 1}y_{j + 1} = y_j^d - 1$
	for every $j \in \Z/n\Z$.
	Therefore $\overline{y_j}^d = 1$.
	Reduction gives a bijection
	$\mu_d \lra \mu_d(\F_p)$ because
	$\mu_d \subset \Q_p$ and $p \nmid d$.
	Let $\epsilon_j \in \mu_d$ be the unique element
	satisfying $y_j \equiv \epsilon_j \pmod p$.

	The sequence
	$(\epsilon_j)_{j \in \Z/n\Z}$
	has exact cyclic period $n$.
	Indeed, suppose that
	$\epsilon_{j + m} = \epsilon_j$ for every $j$
	and some $m$ with $0 < m < n$.
	The $j$-th normalized iterates of $P$ and of
	$f_{d, c}^m(P)$ would then have the same reduction
	for every $j \geq 0$.
	Proposition~\ref{prop:tame-distance} would then
	give $P = f_{d, c}^m(P)$, contradicting the
	exact period of $P$.

	\emph{Step 2: the Laurent series
		$z_{\boldsymbol{\epsilon}}(u)$.}
	Extend $(\epsilon_j)_{j \in \Z/n\Z}$
	periodically to all $j \geq 0$.
	Morton \cite[Lemmas~1 and~2]{morton-curves}
	gives a unique Laurent series
	$z_{\boldsymbol{\epsilon}}(u)
	\in K(\mu_d)((u^{-1}))$
	of exact period $n$ under $F_u$ such that
	$F_u^j(z_{\boldsymbol{\epsilon}}(u))
	= \epsilon_j u + O(1)$ for every $j \geq 0$,
	where $O(1)$ denotes a power series in $u^{-1}$.
	The coefficient calculation in
	Morton \cite[Lemma~1]{morton-curves}
	shows that the coefficients are $v$-integral
	because $p \nmid d$.
	Hence $z_{\boldsymbol{\epsilon}}(u)$ converges at
	$u = \beta$.

	Set $P^\ast := z_{\boldsymbol{\epsilon}}(\beta)$.
	Since $F_\beta = f_{d, c}$, the leading-term
	condition gives
	$\overline{F_\beta^j(P^\ast)/\beta}
	= \epsilon_j
	= \overline{P_j/\beta}$
	for every $j \geq 0$.
	Both $P^\ast$ and $P$ are periodic under
	$f_{d, c}$.
	Proposition~\ref{prop:tame-distance}
	therefore gives $P^\ast = P$.

	\emph{Step 3: the point $s_P$ and its reduction.}
	Recall that Morton \cite[Proposition~10]{morton-curves}
	and Doyle--Krieger--Obus--Pries--
	Rubinstein-Salzedo--West
	\cite[Remark~3.9]{dkoprw}
	associate to $z_{\boldsymbol{\epsilon}}(u)$
	the point
	$s_{\boldsymbol{\epsilon}}
	\in X_{1, d}(n)(K(\mu_d))$
	satisfying
	$\pi_n(s_{\boldsymbol{\epsilon}}) = \infty$.
	Set $s_P := s_{\boldsymbol{\epsilon}}$.

	Write $t := u^{-1}$.
	The tuple
	$\paren{tF_u^j(z_{\boldsymbol{\epsilon}}(u))}
	_{j \in \Z/n\Z}$
	has entries in $\Z_p[[t]]$ and has value
	$(\epsilon_j)_{j \in \Z/n\Z}$ at $t = 0$.
	By
	Doyle--Krieger--Obus--Pries--
	Rubinstein-Salzedo--West
	\cite[Definition~3.10]{dkoprw},
	this tuple defines a formal section
	$B_{\boldsymbol{\epsilon}}(t)$ of their
	mixed-characteristic model.
	The equalities $P^\ast = P$ and
	$t_0 = \beta^{-1}$ give
	$B_{\boldsymbol{\epsilon}}(t_0) = x_P$, while
	$B_{\boldsymbol{\epsilon}}(0) = s_P$.
	Proposition~6.4 of \cite{dkoprw} shows that
	$s_P$ is a smooth point of the special fiber.
	Hence $x_P$ and $s_P$ have the same smooth
	reduction modulo $v$.

	For $0 \leq k < n$, the sequence associated to
	$P_k$ is
	$(\epsilon_{j + k})_{j \in \Z/n\Z}$.
	The uniqueness in Morton
	\cite[Lemma~1]{morton-curves}
	identifies its Laurent series with
	$F_u^k(z_{\boldsymbol{\epsilon}}(u))$.
	The corresponding formal section specializes
	to $\sigma_n^k x_P$ at $t = t_0$ and to
	$\sigma_n^k s_P$ at $t = 0$.
	These two points therefore have the same smooth
	reduction.
	Since $\sigma_n^n = 1$, the same conclusion
	holds for every $k \in \Z$.
\end{proof}

Lemma~\ref{lem:unicritical-dynatomic-disc}
shows that $\sigma_n^j x_P$ and
$\sigma_n^j s_P$ have the same smooth reduction
for every $j$.
We now prove
Proposition~\ref{prop:unicritical-cyclotomic-rank}
by showing
that Abel--Jacobi morphisms therefore
place the divisor classes
$[\sigma_n^j x_P - \sigma_n^j s_P]$
in the kernel of reduction on $J_{1, d}(n)$
and that the cyclic action on these divisor classes
forces a cyclotomic contribution to the
Mordell--Weil rank.

\begin{proof}[Proof of
	Proposition~\ref{prop:unicritical-cyclotomic-rank}]
	The proof has five steps.
	We first choose a place $v$ and a period bound
	$N_{K, d}$.
	For a cycle of length $n > N_{K, d}$,
	Lemma~\ref{lem:unicritical-dynatomic-disc}
	provides the point $s_P$.
	We set $\delta_P := [x_P - s_P]$ and prove that
	the classes $\sigma_n^j\delta_P$ generate a
	torsion-free subgroup
	$\Lambda_P := \Z[\sigma_n]\delta_P$.
	Gonality shows that $\sigma_n$ has exact order
	$n$ on $\Lambda_P$.
	The cyclotomic components of
	$V_P := \Lambda_P \otimes_{\Z}\Q$
	give the rank lower bound.
	Finally, the rank lower bound converts
	a uniform bound for
	$\rank J_{1, d}(n)(K(\mu_d))$
	as $n$ varies
	into a uniform bound for rational periods.

	\emph{Step 1: the place $v$ and
		the period bound $N_{K, d}$.}
	Retain the prime $p$ and the place $v$
	fixed before
	Lemma~\ref{lem:unicritical-dynatomic-disc}.
	Thus $K(\mu_d)_v = \Q_p$.
	The prime $p$ and the place $v$ can be found
	effectively by enumerating rational primes and
	testing complete splitting in $K(\mu_d)$.
	Let $C_v$ be the local period bound supplied by
	Pezda \cite[Theorem~1]{pezda-local}.
	If $v(c) \geq 0$, every periodic point is
	$v$-integral and has period at most $C_v$.

	The characteristic-zero proof of
	Doyle--Poonen
	\cite[Theorem~1.1(b)]{doyle-poonen}
	reduces the gonality growth of $X_{1, d}(n)$
	to the asymptotically linear lower bound in
	\cite[Theorem~1.4]{doyle-poonen}.
	The degree and genus formulas used in that proof
	are effective; see
	\cite[Proposition~3.1]{doyle-poonen}
	and Morton \cite[Theorem~13(d)]{morton-curves}.
	Hence one can effectively choose an integer
	$N_0 = N_0(d)$ such that
	$\operatorname{gon}(X_{1, d}(n)) > 2$
	whenever $n > N_0$.
	Set $N_{K, d} := \max\set{C_v, N_0}$.
	Suppose that $c \in K$ and that $P \in K$ has
	exact period $n > N_{K, d}$ under $f_{d, c}$.
	Since $n > C_v$, one has $v(c) < 0$.

	\emph{Step 2: the torsion-free subgroup
		$\Lambda_P$ of the Jacobian.}
	Let $x_P \in X_{1, d}(n)(K)$ be the point
	corresponding to $P$.
	Lemma~\ref{lem:unicritical-dynatomic-disc}
	gives a point
	$s_P \in X_{1, d}(n)(K(\mu_d))$
	above $c = \infty$ such that
	$\sigma_n^j x_P$ and $\sigma_n^j s_P$
	have the same smooth reduction for every $j$.
	Set
	$\delta_P := [x_P - s_P]
	\in J_{1, d}(n)(K(\mu_d))$.

	Let $\calj_n$ be the N\'{e}ron model of
	$J_{1, d}(n)$ over $\Z_p$, and set
	\[
		\calj_n^1(\Z_p)
			:= \Ker\Big( \calj_n(\Z_p)
				\lra \calj_n(\F_p) \Big).
	\]
	For each $j$,
	apply the N\'{e}ron mapping property
	of Bosch--L\"utkebohmert--Raynaud
	\cite[Chapter~1, Section~2]
		{bosch-lutkebohmert-raynaud}
	to the Abel--Jacobi morphism based at
	$\sigma_n^j s_P$.
	This morphism sends $\sigma_n^j x_P$ to
	$\sigma_n^j\delta_P$ and sends
	$\sigma_n^j s_P$ to $0$.
	Since $\sigma_n^j x_P$ and $\sigma_n^j s_P$
	have the same smooth reduction,
	$\sigma_n^j\delta_P$ reduces to $0$.
	Hence
	$\sigma_n^j\delta_P \in \calj_n^1(\Z_p)$.

	Clark--Xarles
	\cite[Section~3.1, Proposition~9]
		{clark-xarles-local-torsion}
	show that $\calj_n^1(\Z_p)$ contains no
	nonzero torsion because
	$K(\mu_d)_v = \Q_p$ and $p \geq 5$.
	Therefore
	$\Lambda_P := \Z[\sigma_n] \delta_P$
	is a torsion-free subgroup of
	$J_{1, d}(n)(K(\mu_d))$.

	\emph{Step 3: the order of $\sigma_n$ on
		$\Lambda_P$.}
	Let $h$ be the least positive integer such that
	$\sigma_n^h\delta_P = \delta_P$.
	Since $\Lambda_P$ is generated by the classes
	$\sigma_n^j\delta_P$, the integer $h$ is the
	order of $\sigma_n$ on $\Lambda_P$.
	The relation $\sigma_n^n = 1$ gives $h \mid n$.

	Define the divisors
	$D_h^+ := \sigma_n^h x_P + s_P$ and
	$D_h^- := x_P + \sigma_n^h s_P$
	on $X_{1, d}(n)$.
	The equality
	$\sigma_n^h\delta_P = \delta_P$ says that
	$D_h^+$ and $D_h^-$ are linearly equivalent.
	If $D_h^+ \neq D_h^-$, their linear equivalence
	gives a nonconstant rational function on
	$X_{1, d}(n)$ whose pole divisor has degree
	at most two.
	This function defines a nonconstant morphism
	$X_{1, d}(n) \longrightarrow \P^1$
	of degree at most two, contradicting
	$\operatorname{gon}(X_{1, d}(n)) > 2$.

	Therefore $D_h^+ = D_h^-$.
	Since $\pi_n \circ \sigma_n = \pi_n$, the points
	$\sigma_n^h x_P$ and $x_P$ both map to $c$
	under $\pi_n$, while $s_P$ and
	$\sigma_n^h s_P$ both map to $\infty$.
	Comparing the points in $D_h^+$ and $D_h^-$
	that map to $c$ gives
	$\sigma_n^h x_P = x_P$.
	Since $P$ has exact period $n$, one has
	$n \mid h$.
	Consequently, $h = n$.

	\emph{Step 4: the cyclotomic components of $V_P$.}
	Since $\Lambda_P$ is torsion-free,
	$\sigma_n$ also has exact order $n$ on
	$V_P := \Lambda_P \otimes_{\Z}\Q$.
	The idempotents $e_{n, m}$ give
	\[
		V_P
			= \bigoplus_{m \mid n}e_{n, m}V_P.
	\]
	On every nonzero component $e_{n, m}V_P$,
	the automorphism $\sigma_n$ has order $m$.
	Hence the least common multiple of the divisors
	$m$ satisfying $e_{n, m}V_P \neq 0$ is $n$.

	Let $\ell^a \mid n$.
	There is a divisor $m \mid n$ such that
	$\ell^a \mid m$ and $e_{n, m}V_P \neq 0$.
	The component $e_{n, m}V_P$ is a nonzero vector
	space over $\Q(\zeta_m)$, and therefore
	$\dim_{\Q}e_{n, m}V_P \geq \varphi(m)$.
	Moreover,
	$e_{n, m}V_P$ is contained in
	$A_{n, m}(K(\mu_d)) \otimes_{\Z}\Q$.
	Consequently,
	\[
		\rank A_{n, m}\paren{K(\mu_d)}
			\geq \varphi\paren{m}
			\geq \varphi\paren{\ell^a}.
	\]

	\emph{Step 5: bounded ranks and bounded periods.}
	Suppose that the rank of
	$A_{n, m}\paren{K(\mu_d)}$ is at most
	$r_{K, d}$ whenever $n \geq 1$ and $m \mid n$.
	If an exact period $n > N_{K, d}$ occurs, then
	$\varphi\paren{\ell^a} \leq r_{K, d}$ for every
	prime power $\ell^a \mid n$.
	There are only finitely many such prime powers,
	since
	$\varphi\paren{\ell^a}
	= \ell^{a - 1}(\ell - 1)$.

	Choose an integer $M_{K, d} \geq 1$ such that
	every prime power $\ell^a$ satisfying
	$\varphi\paren{\ell^a} \leq r_{K, d}$ divides
	$M_{K, d}$.
	Every occurring period greater than $N_{K, d}$
	divides $M_{K, d}$.
	Consequently, every occurring period is at most
	$B_{K, d} := \max\set{N_{K, d}, M_{K, d}}$.

	For fixed $c \in K$, every point of exact period
	$j$ is a root of $f_{d, c}^j(z) - z$, which has
	degree $d^j$.
	Therefore
	$\#\Per_K(f_{d, c})
	\leq \sum_{j = 1}^{B_{K, d}}d^j$.
	This proves the conditional uniform boundedness
	assertion.
\end{proof}

Proposition~\ref{prop:unicritical-cyclotomic-rank}
identifies the cyclotomic factors on which a long
rational cycle produces rank.
We now pass from these factors to the full Jacobian.

\begin{proof}[Proof of
	Theorem~\ref{thm:unicritical-dynatomic-rank}]
	Suppose that $c \in K$ and that $f_{d, c}$ has a
	$K$-rational point of exact period
	$n > N_{K, d}$.
	For every prime power divisor
	$\ell^a \mid n$, Proposition
	\ref{prop:unicritical-cyclotomic-rank} gives
	a divisor $m \mid n$ and an abelian subvariety
	$A_{n, m} \subseteq J_{1, d}(n)$ such that
	$\rank A_{n, m}\paren{K(\mu_d)}
	\geq \varphi\paren{m}
	\geq \varphi\paren{\ell^a}$.
	Hence
	$\rank J_{1, d}(n)\paren{K(\mu_d)}
	\geq \max_{\ell^a \mid n}
	\varphi\paren{\ell^a}$.

	A uniform bound for the ranks of
	$J_{1, d}(n)(K(\mu_d))$ also bounds the ranks
	of all $A_{n, m}(K(\mu_d))$.
	The second assertion therefore follows from
	Proposition~\ref{prop:unicritical-cyclotomic-rank}.
\end{proof}

\subsection{Cyclic covers of the projective line}
\label{sec:cyclic-covers}

For the rest of this section, let $K$ be a number field
and let $d \geq 4$. The elliptic curves in
Section~\ref{sec:odd-degree-rank} require $d$ to be odd
because they arise from square roots. In every degree
$d \geq 4$, one can instead adjoin a
$(d - 1)$-st root of the same type of quotient.
The product of the two quotients obtained
from Theorem~\ref{thm:power-classes}
leads to a curve whose function field is
obtained by adjoining such a root.
This motivates the following definition.

\begin{definition}
\label{def:four-branch-cyclic-curves}
	Let $L$ be a field of characteristic zero containing
	both $\mu_d$ and $\mu_{d - 1}$. Choose distinct
	$\zeta, \xi \in \mu_d\sm\set{1}$ and
	$Q, R \in L^\times$. Assume that
	$\zeta Q$, $\xi Q$, $\zeta R$, and $\xi R$ are
	pairwise distinct.
	Let $C_{Q, R}$ be the smooth projective curve
	with affine model
	\[
		C_{Q, R}:\quad
		Y^{d - 1}
			= \frac{(X - \zeta Q)(X - \zeta R)}
			{(X - \xi Q)(X - \xi R)}.
	\]
\end{definition}

The function field of $C_{Q, R}$ is obtained from
$L(X)$ by adjoining a $(d - 1)$-st root of the
displayed rational function.
The $X$-coordinate therefore induces a morphism
$C_{Q, R} \longrightarrow \P^1$.
The next lemma shows that this morphism is Galois
with cyclic deck group and computes the genus
of its source.

\begin{lemma}
\label{lem:four-branch-cyclic-geometry}
	Let $d \geq 4$, and let $L$ be a field of
	characteristic zero containing both $\mu_d$ and
	$\mu_{d - 1}$. Choose distinct
	$\zeta, \xi \in \mu_d\sm\set{1}$ and
	$Q, R \in L^\times$. Assume that
	$\zeta Q$, $\xi Q$, $\zeta R$, and $\xi R$ are
	pairwise distinct.
	Let $C_{Q, R}$ be the curve in
	Definition~\ref{def:four-branch-cyclic-curves}.
	The morphism $C_{Q, R} \longrightarrow \P^1$
	induced by the $X$-coordinate is a connected
	Galois cover of degree $d - 1$.
	Its deck group is $\mu_{d - 1}$, acting on
	the affine model by $Y \mapsto \eta Y$
	for $\eta \in \mu_{d - 1}$.
	The cover is totally ramified at the four
	branch points, and $C_{Q, R}$ has genus $d - 2$.
\end{lemma}

\begin{proof}
	The right-hand side of the equation
	defining $C_{Q, R}$ has two simple zeros
	and two simple poles.
	Since it has a simple zero, none of its powers
	with exponent between $1$ and $d - 2$ is a
	$(d - 1)$-st power in $L(X)$.
	Its class therefore has order $d - 1$ in
	$L(X)^\times/(L(X)^\times)^{d - 1}$.
	Hence adjoining a $(d - 1)$-st root of this
	function gives an extension of $L(X)$ of
	degree $d - 1$.
	Consequently, the covering curve is connected.
	Since $\mu_{d - 1} \subset L$, the transformations
	$Y \mapsto \eta Y$, for $\eta \in \mu_{d - 1}$,
	form the full deck group.
	Thus the cover is Galois with cyclic deck group.
	The simple zeros and poles give total
	ramification at the four branch points.
	By Riemann--Hurwitz,
	$2g(C_{Q, R}) - 2 = -2(d - 1) + 4(d - 2)$.
	Therefore $g(C_{Q, R}) = d - 2$.
\end{proof}

For $d = 4$, the covering morphism has degree $3$
and its source has genus $2$.
In general, both the degree and the genus grow
linearly with $d$.
We will also need that these covers genuinely vary in
moduli.

\begin{definition}
	\label{def:Ct-curve}
	Fix distinct
	$\zeta, \xi \in \mu_d\sm\set{1}$.
	For every $t$ for
	which $\zeta$, $\zeta t$, $\xi$, and $\xi t$ are
	pairwise distinct, let $C_t$ be the smooth projective
	curve with affine model
	\[
		C_t: \quad Y^{d - 1}
			= \frac{(X - \zeta)(X - \zeta t)}
			{(X - \xi)(X - \xi t)}.
	\]
\end{definition}

\begin{proposition}
\label{prop:cyclic-cover-moduli}
	For fixed $d$, there is an integer $b_d \geq 1$ such
	that each geometric isomorphism class contains $C_t$
	for at most $b_d$ values of $t$.
\end{proposition}

\begin{proof}
	The proof tracks the cyclic deck group under an
	isomorphism and then compares the branch divisors on
	$\P^1$. Set $m := d - 1$ and $g := d - 2$. The
	distinguished deck group of $C_t$ is cyclic of order
	$m$. An isomorphism from $C_t$ to a fixed curve $C$
	carries this group to a cyclic subgroup
	$H \subseteq \Aut(C)$ of order $m$.

	Since $g \geq 2$, the Hurwitz bound gives
	$\#\Aut(C) \leq 84(g - 1)$.
	Hence $\Aut(C)$ has at most
	$2^{84(g - 1)} = 2^{84(d - 3)}$ subgroups,
	a bound depending only on $d$.
	For a fixed $H$, the quotient $C/H$ and its unordered
	branch divisor are determined. One cross-ratio of
	$\zeta$, $\zeta t$, $\xi$, and $\xi t$ is
	\[
		\lambda(t)
			:= \frac{t(\zeta - \xi)^2}
			{(\zeta - \xi t)(\zeta t - \xi)}.
	\]
	The function $\lambda(t)$ is nonconstant and has
	degree at most two. An unordered set of four points
	determines at most six cross-ratios. Thus only
	boundedly many values of $t$ can occur for each
	possible subgroup $H$.
\end{proof}

The bound of the explicit Mordell theorem of Yu--Yuan--Zhou
\cite[Theorem~1.3]{yu-yuan-zhou} involves the
following genus-dependent constant:
for $g \geq 2$,
set $\rho(g) := \min\set{1 + 5/(4\sqrt g),
1 + 3\log g/g}$.

\begin{theorem}
\label{thm:jacobians-large-rank}
	Let $K$ be a number field, let $d \geq 4$,
	and let $b_d$ be the integer from
	Proposition~\ref{prop:cyclic-cover-moduli}.
	There is a finite extension $M/K$,
	depending only on $K$ and $d$
	such that: if $c \in K$ and
	$f_{d, c}$ has a $K$-rational cycle of length $n \geq 3$,
	then there are
	at least $\lceil(n - 2)/b_d\rceil$ geometrically
	nonisomorphic smooth projective curves $C/M$ of genus
	$d - 2$ such that $\#C(M) \geq n - 2$ and
	\[
		\rank J(C)(M)
		\geq \frac{
			\log\paren{n - 2}
			- \log\paren{10^{13}(d - 2)^8}}
			{\log \rho(d - 2)}.
	\]
	Furthermore, the principally polarized Jacobians
	$J(C)$ are pairwise
	geometrically nonisomorphic.
\end{theorem}

\begin{proof}
	We first construct the fixed field extension $M/K$.
	Choose distinct $\zeta, \xi \in \mu_d \sm \set{1}$.
	Recall the the finite set of classes $\calc_{K, d}$
	and the fixed representative set
	$A \subset K(\mu_d)^\times$
	from Definition~\ref{def:classes-modulo-powers}.
	Let $M/K(\mu_d)$ be the extension obtained by
	adjoining $\mu_{d - 1}$ and a $(d - 1)$-st root
	of every $a \in A$.
	This is a fixed finite extension
	depending only on $K, d$, and every class in
	$\calc_{K, d}$ becomes a $(d - 1)$-st power over $M$.

	Next, we lift periodic cycles to the curves $C_{Q, R}$.
	Let $\sco \subset K$ be a cycle of length $n$ and
	fix a nonzero point $Q \in \sco$. Apart from $Q$
	and a possible zero point, there are at least $n - 2$
	choices of a nonzero point $R \in \sco$. For each
	such $R$, the four points
	$\zeta Q$, $\xi Q$, $\zeta R$, and $\xi R$ are
	pairwise distinct by
	Lemma~\ref{lem:four-branch-points},
	so the curve $C_{Q, R}/M$ is well-defined.

	For every $P \in \sco\sm\set{Q, R}$,
	Theorem~\ref{thm:power-classes}, applied with $Q$ and
	with $R$, shows that the quotient
	\[
		\frac{(P - \zeta Q)(P - \zeta R)}
		{(P - \xi Q)(P - \xi R)}
	\]
	is a $(d - 1)$-st power in $M$. Hence $P$ lifts to
	an $M$-rational point of $C_{Q, R}$. By
	Lemma~\ref{lem:four-branch-cyclic-geometry},
	$g(C_{Q, R}) = d - 2$ and
	$\#C_{Q, R}(M) \geq n - 2$.
	Applying Yu--Yuan--Zhou~\cite[Theorem~1.3]{yu-yuan-zhou}
	gives the claimed lower bound for $\rank J(C_{Q, R})(M)$.

	The change of variables $(X, Y) \mapsto (X/Q, Y)$
	identifies $C_{Q, R}$ with
	the specializiation $C_{R/Q}$ of the family $C_t$
	from Definition \ref{def:Ct-curve}.
	Proposition~\ref{prop:cyclic-cover-moduli} shows that
	the choices of $R$ give at least
	$\lceil(n - 2)/b_d\rceil$ geometric isomorphism
	classes. The Torelli theorem then shows
	that their principally polarized Jacobians are
	pairwise geometrically nonisomorphic.
\end{proof}

Thus a long rational cycle produces both a logarithmic
rank lower bound and, since $b_d$ depends only on $d$,
linear variation in moduli in every degree $d \geq 4$.

\begin{remark}[A fixed hyperelliptic variant]
	There is also a fixed-curve version of the same idea.
	Recall that
	$G_d(X,A) := (X^d - A^d)/(X - A)$. Retain the fixed
	points $Q, R$ and the fields $L_n$
	from Section~\ref{sec:fixed-elliptic-fields}. The
	roots of $G_d(X, Q)G_d(X,R)$ are distinct by the same
	argument as in
	Lemma~\ref{lem:four-branch-points}. Hence the smooth
	projective curve
	\[
		Y^2 = G_d(X, Q)G_d(X,R)
	\]
	has genus $d - 2$. After adjoining to $L_n$ the square
	roots of $G_d(P, Q)G_d(P,R)$ for all
	$P \in \Per_{\overline{K}}(f_{d, c},n)$, every such point
	lifts to this fixed curve. The quantitative bound of
	Yu--Yuan--Zhou \cite[Theorem~1.3]{yu-yuan-zhou}
	therefore gives the analogous lower bound for the rank
	of its Jacobian. Since
	Theorem~\ref{thm:fixed-elliptic-curve-rank} already
	produces a fixed elliptic curve in every degree,
	we just remark that this gives a higher-dimensional variant
	of Theorem~\ref{thm:fixed-elliptic-curve-rank}.
\end{remark}

The extension $M/K$ constructed in the proof of
Theorem~\ref{thm:jacobians-large-rank} makes every
$a \in A$ a $(d - 1)$-st power.
When $3 \mid d - 1$,
every $a \in A$ is therefore a cube in $M$.
The elimination in Proposition~\ref{prop:quotient-curves}
then places quotients $P/Q$ of distinct periodic points
on one fixed Fermat cubic over $M$.
The cubic has an $M$-rational point and $j$-invariant
zero, so it is an elliptic curve over $M$.
The resulting statement complements
Theorem~\ref{thm:elliptic-curves-large-rank}:
it also applies when $d \equiv 4 \pmod 6$,
but it does not give variation in moduli.

\begin{corollary}
\label{cor:fixed-fermat-cubic}
	Let $K$ be a number field, and let $d \geq 4$ satisfy
	$3 \mid d - 1$. There exist a finite extension $M/K$, an
	elliptic curve $E/M$ with $j(E) = 0$, and a constant
	$C_{K, d} \geq 2$, all depending only on $K$ and $d$,
	such that, for every $c \in K$, every $K$-rational
	cycle of $z^d + c$ of length $n \geq 3$ satisfies
	\[
		n - 2
		\leq C_{K, d}^{1 + 2\rank E(M)}.
	\]
	Equivalently,
	\[
		\rank E(M)
		\geq \frac{1}{2}
		\paren{\frac{\log\paren{n - 2}}
		{\log C_{K, d}} - 1}.
	\]
\end{corollary}

\begin{proof}
	Fix distinct $\zeta_1, \zeta_2, \zeta_3$ in
	$\mu_d \sm \set{1}$ as in
	Section~\ref{sec:twisted-fermat}.
	Let $M$ be the fixed extension constructed in the proof of
	Theorem~\ref{thm:jacobians-large-rank};
	over $M$, every class in $\calc_{K, d}$ becomes a
	$(d - 1)$-st power. The Fermat cubic
	\[
		E : (\zeta_2 - \zeta_3)U^3
		+ (\zeta_3 - \zeta_1)V^3
		+ (\zeta_1 - \zeta_2)W^3 = 0
	\]
	has the $M$-rational point $[1 : 1 : 1]$
	and hence is an elliptic curve.
	It is geometrically a Fermat cubic,
	so $j(E) = 0$. Set
	\[
		\tau
		:= \frac{\zeta_3U^3 - \zeta_1W^3}
		{U^3 - W^3}
		\in M(E).
	\]
	The cubic $U^3 = W^3$ meets $E$ in nine distinct
	points, and the numerator of $\tau$ vanishes at
	none of them: a common zero would satisfy
	$\zeta_3U^3 = \zeta_1W^3$ and $U^3 = W^3$, hence
	$U = W = 0$ because $\zeta_1 \neq \zeta_3$, and
	then $V = 0$ as well, which is impossible.
	Therefore $\tau$ has degree $9$,
	its pole divisor is reduced,
	and $m_{\min}(\tau) = m_{\max}(\tau) = 1$.

	Let $\sco$ be a cycle of $f_{d, c}$ of length $n$.
	At most one point of $\sco$ is zero, so choose a
	nonzero $Q \in \sco$.
	For $P \in \sco \sm \set{Q}$, the two quotients
	$(P - \zeta_i Q)/(P - \zeta_3 Q)$, for $i = 1, 2$,
	are cubes in $M$ by Theorem~\ref{thm:power-classes}
	and the assumption $3 \mid d - 1$.
	Choose $u, v \in M$ satisfying
	$u^3 = (P - \zeta_1 Q)/(P - \zeta_3 Q)$ and
	$v^3 = (P - \zeta_2 Q)/(P - \zeta_3 Q)$.
	These equalities give
	\[
		(\zeta_2 - \zeta_3)u^3
		+ (\zeta_3 - \zeta_1)v^3 + (\zeta_1 - \zeta_2) = 0,
	\]
	so $[u : v : 1] \in E(M)$.
	Also, $u^3 - 1 = (\zeta_3 - \zeta_1)Q/(P - \zeta_3 Q)$
	is nonzero, and the formula for $\tau$ gives
	$\tau([u : v : 1]) = P/Q$.
	Set $g_{Q, c}(X) := Q^{d - 1}X^d + c/Q \in M[X]$,
	a polynomial of degree $d$.
	Apart from $Q$ and its predecessor in $\sco$,
	every $P \in \sco$ gives $A, B \in E(M)$ with
	$\tau(A) = P/Q$ and $\tau(B) = f_{d, c}(P)/Q$.
	These values satisfy $\tau(B) = g_{Q, c}(\tau(A))$.
	Let $Z \subset E \times E$ be the curve
	$Z_{g_{Q, c}}$ of
	Definition~\ref{def:elliptic-relation-curve},
	formed with $E_1 = E_2 = E$,
	$\phi_1 = \phi_2 = \tau$, and $g = g_{Q, c}$.
	Each such $P$ gives a point
	$(A, B) \in Z \cap \paren{E(M) \times E(M)}$, and
	distinct points $P$ give distinct values
	$\tau(A) = P/Q$ and hence distinct points of $Z$.
	Therefore
	\[
		\#\paren{Z \cap \paren{E(M) \times E(M)}}
		\geq n - 2.
	\]

	Since $d \cdot m_{\min}(\tau) = d > 1
	= m_{\max}(\tau)$,
	Lemma~\ref{lem:no-elliptic-coset} shows that no
	irreducible component of $Z$ contains a translate
	of a positive-dimensional abelian subvariety of
	$E \times E$.
	Write $p_i : E \times E \lra E$ for the two
	projections, let $H_i$ be the class of
	$p_i^*\tau^*\calo_{\P^1}(1)$ in the
	N\'{e}ron--Severi group of $E \times E$, and let
	$\calm_\tau$ be the polarization with class
	$H_1 + H_2$. Then $H_1^2 = H_2^2 = 0$ and
	$H_1H_2 = (\deg \tau)^2 = 81$. As in
	\eqref{eq:elliptic-relation-degree}, the divisor
	with support $Z$ has class $dH_1 + H_2$, so
	\[
		\deg_{\calm_\tau} Z
		\leq (dH_1 + H_2)(H_1 + H_2)
		= 81(d + 1).
	\]
	Hence $Z$ has at most $81(d + 1)$ geometric
	irreducible components, each of degree at most
	$81(d + 1)$ with respect to $\calm_\tau$.

	Each of these components has degree at most
	$81(d + 1) \leq 4\paren{21(d + 1) + 1}$, so
	we can bound every uniform Mordell--Lang constant
	here by $\kappa_{21(d + 1)}$ from
	Definition~\ref{def:uniform-mordell-lang-constant}.
	The group $E(M) \times E(M)$ has rank
	$2\rank E(M)$. Applying uniform Mordell--Lang
	\cite[Theorem~1.1]{gao-ge-kuhne} to
	$E(M) \times E(M)$ and to each geometric
	irreducible component of $Z$, exactly as in the
	proof of
	Proposition~\ref{prop:elliptic-mordell-lang-bound},
	gives
	\[
		n - 2
		\leq 81(d + 1)
		\kappa_{21(d + 1)}^{1 + 2\rank E(M)}.
	\]
	The corollary follows with
	$C_{K, d} := 81(d + 1)\kappa_{21(d + 1)}$,
	which is at least $2$ and depends only on $d$.
\end{proof}


\section{Effectivity and strong uniformity}
\label{sec:quantitative}

This section has two purposes.
First, we make the bounds
in Theorem~\ref{thm:main} as explicit
as the available Diophantine inputs permit.
Second, we identify the two obstacles
that prevent the present method
from proving strong uniform boundedness.
The explicit formulas in this section
are computed numerically in several examples
in \hyperref[app:example-effective-constants]{Appendix 1}.

Throughout this section,
let $K$ be a number field and $d \geq 3$.
Recall the finite subgroup
$\calc_{K, d} \subset
K(\mu_d)^\times / (K(\mu_d)^\times)^{d - 1}$
and its representative set $A$ from
Definition~\ref{def:classes-modulo-powers}.
For each $\alpha \in \calc_{K, d}$, let $a_\alpha \in A$
be its unique representative.

When $d \geq 5$, we sharpen the bound in
Proposition~\ref{prop:dge5} by normalizing the third
coefficient of $C_{a_1, a_2, a_3}$.
For $(\alpha_1, \alpha_2) \in \calc_{K, d}^2$, write
$C_{\alpha_1, \alpha_2} := C_{a_{\alpha_1}, a_{\alpha_2}, 1}$.
Given $a_1, a_2, a_3 \in A$, let $\alpha_i = [a_i/a_3]$
for $i = 1, 2$. Choose $u_i \in K(\mu_d)^\times$ with
$a_i/a_3 = a_{\alpha_i}u_i^{d - 1}$.
Dividing the equation in
Definition~\ref{def:twisted-fermat-family} by $a_3$ shows
that the coordinate change
\[
	[X_1 : X_2 : X_3] \longmapsto [u_1X_1 : u_2X_2 : X_3]
\]
identifies $C_{a_1, a_2, a_3}$ with
$C_{\alpha_1, \alpha_2}$ over $K(\mu_d)$.
It preserves the expression for $P/Q$ in the proof of
Proposition~\ref{prop:quotient-curves}.
For fixed $Q \neq 0$, a point on $C_{\alpha_1, \alpha_2}$
therefore still determines $P$. Applying these coordinate
changes to the injection in
Proposition~\ref{prop:quotient-curves} gives
\begin{equation}
	\#\Per_K(f_{d, c}) \leq 1 +
	\sum_{(\alpha_1, \alpha_2) \in \calc_{K, d}^2}
	\#C_{\alpha_1, \alpha_2}(K(\mu_d)).
	\label{eq:normalized-fermat-bound}
\end{equation}
Thus we use only
$\paren{\#\calc_{K, d}}^2$ many curves,
rather than the $\paren{\#\calc_{K, d}}^3$
many curves in Section~\ref{sec:dge5},
to get sharper bounds.

For $d = 4$, we make the bound $2 + N_K$ obtained
in the proof of Proposition~\ref{prop:d4}
effective using the explicit form of
R\'{e}mond's Mordell--Lang estimate in
David--Philippon~\cite[p.~643]{david-philippon-ii}.

\subsection{Explicit fixed-field bounds}

The proofs use several external results, of which only
some carry computable constants. In the table below,
$F$ denotes a number field, $C_v$ is the local period
bound in Section~\ref{sec:d3}, and $Z_t \subset E \times E$
is the curve over $L$ defined in Section~\ref{sec:d4}.
\begin{center}
\small
\begin{tabular}{@{}l p{0.36\textwidth} l@{}}
input & used for & explicit \\ \hline
	\cite[Theorem~3]{rajagopal-zhang}
		& the good-reduction branch
		& yes \\
	\cite[Theorem~1]{pezda-local}
		& the local period bound $C_v$
		& yes \\
	\cite{faltings-mordell}
		& finiteness of $C_{\alpha_1, \alpha_2}(F)$
			when $d \geq 5$
		& no \\
	\cite[Theorem~1.6]{yu-yuan-zhou}
		& a bound for $\#C_{\alpha_1, \alpha_2}(F)$
		& yes \\
	\cite[Theorem~1.3]{yu-yuan-zhou}
		& the rank-dependent bound
			\eqref{eq:rank-explicit-bound}
		& except for $\rk_{K, d}$ \\
	\begin{tabular}[t]{@{}l@{}}
		\cite[Theorem~2.1]{remond-count}, \\
		\cite[p.~643]{david-philippon-ii}
	\end{tabular}
		& a bound for $\#Z_t(L)$ independent of $t$
		& yes \\
	\cite[Theorem~1.1]{gao-ge-kuhne}
		& the constants $\kappa_e$
		& no \\
	\cite[Theorem~1.8]{doyle-poonen}
		& passage from periodic to preperiodic bounds
		& no
\end{tabular}
\end{center}
Faltings's theorem gives qualitative finiteness for the
curves of genus at least two, while the bounds of
Yu--Yuan--Zhou make their point counts effective.
The bounds in Theorem~\ref{thm:main} are effective.
In the unconditional $d \geq 5$ and $d = 3$ cases,
explicit formulas are given below.
For $d = 4$, the explicit form of R\'{e}mond's estimate
applies because $E \times E$ is fixed; we do not
numerically evaluate the resulting height constants.

By contrast, the rank bounds in
Theorems~\ref{thm:elliptic-curves-large-rank} and
\ref{thm:fixed-elliptic-curve-rank},
the conditional $d = 3$ periodic-point bound in
Corollary~\ref{cor:conditional-d3}, and the rank
bound in Corollary~\ref{cor:fixed-fermat-cubic}
depend on the constants $\kappa_e$ of
Definition~\ref{def:uniform-mordell-lang-constant},
for which no effective value is presently known.

Given an explicit uniform bound for the ranks of
$A_{n, m}(K(\mu_d))$, the periodic-point bound in
Proposition~\ref{prop:unicritical-cyclotomic-rank}
is effectively computable.
The same conclusion holds in the second assertion of
Theorem~\ref{thm:unicritical-dynatomic-rank}
when an explicit uniform bound for the ranks of
$J_{1, d}(n)(K(\mu_d))$ is supplied.
However, the bounds in
Corollaries~\ref{cor:ubc-preperiodic-unconditional}
and~\ref{cor:preperiodic} remain non-effective
because they use Doyle--Poonen
\cite[Theorem~1.8]{doyle-poonen}.
The Jacobian-rank bound in
Theorem~\ref{thm:jacobians-large-rank} is explicit.

Let $S$ consist of the archimedean places of
$K(\mu_d)$ and the finite places above $d$,
and let $S_f$ be its finite part.
Set $r_S := r_1\paren{K(\mu_d)}
+ r_2\paren{K(\mu_d)} + \#S_f - 1$. The $S$-unit
group and the $S$-ideal class group determine
$\calc_{K, d}$ and the chosen representatives. More
precisely,
\[
	\#\calc_{K, d}
	= (d - 1)^{r_S}
	\#\frac{\mu\paren{K(\mu_d)}}
	{\mu\paren{K(\mu_d)}^{d - 1}}
	\cdot
	\#\Cl\paren{\calo_{K(\mu_d),S}}[d - 1].
\]
The two factors after $(d - 1)^{r_S}$ account for the
torsion in the $S$-unit group and the
$(d - 1)$-torsion in the $S$-ideal class group.
Consequently, $\calc_{K, d}^2$ and all the associated
curves are effectively determined.

Suppose first that $d \geq 5$. Recall that
$g_d = (d - 2)(d - 3)/2$, and set
\[
	\rho_d := \rho(g_d)
	= \min\set{
	1 + \frac{5}{4\sqrt{g_d}},
	1 + \frac{3\log g_d}{g_d}}.
\]
For $(\alpha_1, \alpha_2) \in \calc_{K, d}^2$, let
$J_{\alpha_1, \alpha_2}$ be the Jacobian of
$C_{\alpha_1, \alpha_2}$, and set
\[
	\rk_{K, d}
	:= \max_{(\alpha_1, \alpha_2) \in
	\calc_{K, d}^2}
	\rank J_{\alpha_1, \alpha_2}(K(\mu_d)).
\]
Yu--Yuan--Zhou \cite[Theorem~1.3]{yu-yuan-zhou}
prove that
\[
	\#C_{\alpha_1, \alpha_2}\paren{K(\mu_d)}
	\leq 10^{13}g_d^8\rho_d^{
	\rank J_{\alpha_1, \alpha_2}(K(\mu_d))}.
\]
Combining \eqref{eq:normalized-fermat-bound} with this
estimate shows that one may take
\begin{equation}
	B_{\Per}(K, d)
	= 1 + 10^{13}g_d^8
	\paren{\#\calc_{K, d}}^2
	\rho_d^{\rk_{K, d}}.
	\label{eq:rank-explicit-bound}
\end{equation}

There is also a formula that does not contain
Mordell--Weil ranks. For
$(\alpha_1, \alpha_2) \in \calc_{K, d}^2$, let
$N_{\alpha_1, \alpha_2}$ be the product of the
norms of the finite places of $K(\mu_d)$ at which
$C_{\alpha_1, \alpha_2}$ has bad reduction.
Yu--Yuan--Zhou \cite[Theorem~1.6]{yu-yuan-zhou}
give explicit functions $c_1,c_2,c_3$ such that
\[
\begin{aligned}
	\#C_{\alpha_1, \alpha_2}\paren{K(\mu_d)}
		&\leq c_1\paren{g_d,[K(\mu_d):\Q]}
		\paren{N_{\alpha_1, \alpha_2}}^{
		c_2\paren{g_d,[K(\mu_d):\Q]}}
		\verts{\Delta_{K(\mu_d)}}^{
		c_3\paren{g_d,[K(\mu_d):\Q]}}.
\end{aligned}
\]
Let $M_{\alpha_1, \alpha_2}$
denote the right-hand side.
Applying \eqref{eq:normalized-fermat-bound} gives
the following explicit bound for the fixed
number field $K$.
\begin{equation}
	B_{\Per}(K, d)
	= 1 + \sum_{(\alpha_1, \alpha_2) \in
	\calc_{K, d}^2}M_{\alpha_1, \alpha_2}.
	\label{eq:fully-explicit-bound}
\end{equation}
For each auxiliary curve $C_{\alpha_1, \alpha_2}$,
any effectively computable finite set
containing its places of bad reduction places
may be used.
For the model in
Definition~\ref{def:twisted-fermat-family},
it is enough to include the places dividing $d - 1$
and the places at which
one of the three coefficients is not a unit
after a common scaling.

For $d = 4$, the Fermat curves have genus one, so
Yu--Yuan--Zhou~\cite[Theorem~1.6]{yu-yuan-zhou}
does not apply. We instead make the bound $2 + N_K$
from Proposition~\ref{prop:d4} effective.
The field $L$, the elliptic curve $E/L$, and the map
$\phi : E \lra \P^1$ are effectively determined by $K$.
The divisor $\phi(R) = t\phi(S)$ on $E \times E$ has
degree $16(9 + 9) = 288$ with respect to the sixteenth
power of the product principal polarization, the
convention used by
David--Philippon~\cite[p.~643]{david-philippon-ii}.
Thus $Z_t$ has at most $288$ geometric components,
each of degree at most $288$, independently of $t$.
For $t \in L^\times \sm \set{1}$, every component has
geometric genus at least two by the proof of
Proposition~\ref{prop:d4}. The explicit form of
R\'{e}mond's estimate in
David--Philippon~\cite[p.~643]{david-philippon-ii}
bounds its rational points in terms of this degree,
the fixed polarized abelian surface, and $2\rank E(L)$.
The rank estimate quoted by
Yu--Yuan--Zhou~\cite[after Theorem~1.6]{yu-yuan-zhou}
gives an effective upper bound for $\rank E(L)$.
Applying these estimates to the geometric components
makes $N_K$ in \eqref{eq:quartic-intersection-bound}
effectively computable.
Hence one may take $B_{\Per}(K, 4) = 2 + N_K$.

For $d = 3$ and
$\Q(\sqrt{-3}) \not\subset K$, the proof of
Proposition~\ref{prop:d3} is explicit once a nonsplit
tame place $v$ is chosen. Such a place can be found by
enumerating finite places and testing their splitting
behavior in $K(\mu_3)$. If $C_v$ denotes Pezda's
effective local period bound for $K_v$, then one may
take
\[
	B_{\Per}(K, 3) = \frac{3^{C_v + 1} - 3}{2}.
\]

If $\Q(\sqrt{-3}) \subset K$ and the ranks of the
curves in \eqref{eq:cubic-double-cover-family} are bounded
by $r_K$, Corollary~\ref{cor:conditional-d3}
gives the conditional bound
$\max\set{3^{[K : \Q]},
16\paren{\#\calc_{K, 3}}^2\kappa_3^{1 + 2r_K}}$.
The constant $\kappa_3$ comes from
the uniform Mordell--Lang theorem of
Gao--Ge--K\"{u}hne \cite{gao-ge-kuhne}
and is not presently explicit.
Consequently, this bound is quantitative in $r_K$
but not numerical.

Doyle--Poonen
\cite[Theorem~1.8]{doyle-poonen}
non-effectively prove that uniform boundedness
for periodic points of unicritical polynomials
implies uniform boundedness for preperiodic points.
Consequently, the explicit periodic bounds above do
not presently give effective bounds
for the number of preperiodic points.

\subsection{Obstructions to strong uniformity}

The bounds $B_{\Per}(K, d)$ in Theorems~\ref{thm:main},
\ref{thm:elliptic-curves-large-rank},
and \ref{thm:unicritical-dynatomic-rank}
and the bounds $B_{\PrePer}(K, d)$
in Corollaries~\ref{cor:ubc-preperiodic-unconditional}
and \ref{cor:preperiodic}
depend only on the number field $K$ and the degree $d$
of the unicritical polynomial $f_{d, c}$.
However, there are two main obstructions
that prevent the bounds from only depending on
the field extension degree $[K : \Q]$ and $d$.

The first obstruction to strong uniform boundedness is
the size of the finite group $\calc_{K, d}$.
Lemma~\ref{lem:finite-power-classes} constructs this
group from the $S$-unit group and the $(d - 1)$-torsion
of the $S$-ideal class group. The $S$-unit rank is
controlled by $d$ and $[K : \Q]$. The present argument
does not bound the class-group contribution in terms
of $d$ and $[K : \Q]$ alone. For $d \geq 5$, the bound
uses all the normalized Fermat curves indexed by
$\calc_{K, d}^2$, so this contribution enters through
$\paren{\#\calc_{K, d}}^2$.
For $d = 4$, it enters through the number of cube roots
adjoined in the definition of $L$ in Section~\ref{sec:d4}.

Second, the quantitative bounds of
Yu--Yuan--Zhou~\cite[Theorems~1.3 and~1.6]{yu-yuan-zhou}
depend on Jacobian ranks or on field discriminants and
bad reduction; see
\hyperref[sec:appendix-explicit-without-ranks]{Appendix 1.1}.
The quartic bound depends on the field $L$ and on
$\rank E(L)$. These quantities are not controlled by
$[K : \Q]$ alone in the present argument.

Hence we have the following two questions, for which
affirmative answers would prove strong uniform
boundedness for the unicritical family when $d \geq 5$.

\begin{question}
\label{q:strong-uniformity}
	Can one bound the number of classes
	in Theorem~\ref{thm:power-classes}
	independently of the class-group torsion,
	or otherwise avoid summing over all of
	$\calc_{K, d}^2$?
	Can such a bound be combined with a
	point count on the twisted Fermat curves
	that is independent
	of their Mordell--Weil ranks?
\end{question}

If $C_0$ and $r_0$ were uniform bounds
for $\#\calc_{K, d}$ and $\rk_{K, d}$ over all fields $K$
with $[K:\Q] \leq D$, then
Equation~\eqref{eq:rank-explicit-bound} would give, for
$d \geq 5$,
\[
	\#\Per_K(f_{d, c})
	\leq 1 + 10^{13}g_d^8C_0^2
	\rho_d^{r_0}.
\]
Thus uniform control of these two quantities would
upgrade our results to strong uniform boundedness.
If one also assumes $\verts{\Delta_K} \leq \Delta_0$, then
Hermite's theorem leaves only finitely many fields $K$
of bounded degree. Therefore,
Section~\ref{sec:quantitative} gives an effectively
computable bound depending on $(d, D, \Delta_0)$.

\clearpage
\appendix

\renewcommand{\thesection}{\arabic{section}}

\renewcommand{\theHsection}{appendix.\arabic{section}}

\section{Examples of numerical bounds}
\label{app:example-effective-constants}

As discussed in
Section~\ref{sec:quantitative},
the bounds in Theorem~\ref{thm:main} are effective.
This appendix works out several
examples when $K = \Q$ and when $K/\Q$ is quadratic,
showing that the bounds are not
only theoretically computable,
but can also be made explicit.

The calculation has three parts. We first write the
rational point bound that is independent of Mordell--Weil ranks
in terms of the fixed finite family of auxiliary
twisted Fermat curves and their bad-reduction data. We next
evaluate these inputs for $K = \Q$ in degrees $5$
through $8$. Finally, we record the corresponding
estimates for several quadratic fields. The quartic
argument is effective, but we do not evaluate its
height constants here.
The SageMath files
used to verify the calculations
in this appendix are archived in
\cite{zhang-ubc-computations}.

For $3 \leq d \leq 8$, we calculate the following
bounds when $K = \Q$. In the cubic row, $C_v$ is the
effective local period bound used in
Proposition~\ref{prop:d3}.
For odd $d \geq 3$, Narkiewicz
\cite{narkiewicz-odd} observes that $x^d + c$ is an
increasing function on $\R$. Consequently, over a field
with a real embedding, every periodic point is fixed,
and there are at most three such points.
The $d = 4$ entry refers to the effectivity argument in
Section~\ref{sec:quantitative}.
We do not numerically evaluate the Mordell--Lang constant $N_K$.
The entries for $d = 5$
through $8$ are rounded upward
to a power of $10$.
\begin{center}
\small
\begin{tabular}{c|c|c}
$d$ & bound from the general argument & elementary bound
	\\ \hline
$3$ & $(3^{C_v + 1} - 3)/2$ & $3$ \\
$4$ & effective but not evaluated &
	-- \\
$5$ & $10^{\,10^{25}}$ & $3$ \\
$6$ & $10^{\,2 \cdot 10^{90}}$ & -- \\
$7$ & $10^{\,2 \cdot 10^{246}}$ & $3$ \\
$8$ & $10^{\,2 \cdot 10^{547}}$ & --
\end{tabular}
\end{center}

Recall from Section~\ref{sec:quantitative} that the
normalized twisted Fermat curves are indexed by
$(\alpha_1, \alpha_2) \in \calc_{K, d}^2$ and defined by
$C_{\alpha_1, \alpha_2} :=
C_{a_{\alpha_1}, a_{\alpha_2}, 1}$.
Also recall that $g_d = (d - 2)(d - 3)/2$, that
$\rho_d = \rho(g_d)$, and that $N_{\alpha_1, \alpha_2}$
is the product of the norms of the finite places where
$C_{\alpha_1, \alpha_2}$ has bad reduction.
For convenience, we recall the following notation.
\begin{center}
	\small
	\begin{tabular}{@{}lll@{}}
		symbol & meaning & defined in \\ \hline
		$g_d$ & genus of the twisted Fermat curves
			& \S\ref{sec:twisted-fermat} \\
		$\calc_{K, d}$ &
			fixed finite subgroup of
			$K(\mu_d)^\times / (K(\mu_d)^\times)^{d - 1}$
			& Def.~\ref{def:classes-modulo-powers} \\
		$A_2$ &
		\begin{tabular}[t]{@{}l@{}}
			representatives of the image of $\calc_{K, d}$\\
			in $K(\mu_d)^\times / (K(\mu_d)^\times)^2$
		\end{tabular}
			& Def.~\ref{def:odd-square-classes} \\
		$\kappa_d$ & uniform Mordell--Lang constant
			& Def.~\ref{def:uniform-mordell-lang-constant} \\
		$\rho_d$ & constant in the count of \cite{yu-yuan-zhou}
			& \S\ref{sec:quantitative} \\
		$\rk_{K, d}$ & largest rank of the Jacobians
			& \S\ref{sec:quantitative}
	\end{tabular}
\end{center}

\subsection{A bound without Mordell--Weil ranks}
\label{sec:appendix-explicit-without-ranks}

Suppose that $d \geq 5$, and set
$D := [K(\mu_d):\Q]$. Let $N$ be an upper bound
for $N_{\alpha_1, \alpha_2}$ over all
$(\alpha_1, \alpha_2) \in \calc_{K, d}^2$.
Yu--Yuan--Zhou
\cite[Theorem~1.6]{yu-yuan-zhou} give
\[
\begin{aligned}
	c_1(g_d,D)
		&= 10^{13}g_d^8
		\rho_d^{2g_d^3 D^3 2^{8g_d^2}}, \\
	c_2(g_d,D)
		&= 4g_d^3 D^2 2^{8g_d^2}
		\log_4\paren{\rho_d}, \\
	c_3(g_d,D)
		&= g_d D 2^{8g_d^2}\log_4\paren{\rho_d}.
\end{aligned}
\]
Equation~\eqref{eq:fully-explicit-bound} therefore gives
\begin{equation}
	\begin{aligned}
		B_{\Per}(K, d)
			&\leq 1 + \paren{\#\calc_{K, d}}^2
			c_1(g_d, D)N^{c_2(g_d, D)}
			\verts{\Delta_{K(\mu_d)}}^{c_3(g_d, D)}.
	\end{aligned}
	\label{eq:appendix-general-point-bound}
\end{equation}
A valid choice for $N$ is the
product of the norms of the places
in $S_{K, d}$.
Thus one only needs to compute $\#\calc_{K, d}$ and a
finite set $S_{K, d}$ of finite places of
$K(\mu_d)$ containing the places of bad reduction of
every auxiliary twisted Fermat curve
$C_{\alpha_1, \alpha_2}$, for
$(\alpha_1, \alpha_2) \in \calc_{K, d}^2$.
The set $S_{K, d}$ is independent of $c$,
only depending on $K$ and $d$.

In the cyclotomic examples below, the relevant
$S$-ideal class groups are trivial.
The representatives $a_\alpha$
may therefore be chosen to be $S$-units.
It is therefore enough to take $S_{K, d}$ to be the
set of places of $K(\mu_d)$
above the rational prime factors of
$d(d - 1)$, although we point out that
this set may not be minimal for any
individual curve.

\subsection{The rational field in degrees
	\texorpdfstring{$5$ through $8$}{5 through 8}}

For $K = \Q$, the fields $\Q(\mu_d)$ have class
number one for $5 \leq d \leq 8$. The formula in
Section~\ref{sec:quantitative} gives
\[
\begin{array}{c|cccc}
	d & 5 & 6 & 7 & 8 \\ \hline
	\#\calc_{\Q,d} & 32 & 25 & 432 & 49.
\end{array}
\]
The corresponding numbers of pairs of classes in $\calc_{\Q, d}$ are
$32^2 = 1024$, $25^2 = 625$, $432^2 = 186624$, and
$49^2 = 2401$, respectively.

The remaining inputs are listed below. In the table,
$D = [\Q(\mu_d):\Q]$.
\begin{center}
\small
\begin{tabular}{c|c|c|c|c|c}
$d$ & $D$ & $\#\calc_{\Q,d}$ &
	$\paren{\#\calc_{\Q,d}}^2$ & $N$ &
	$\verts{\Delta_{\Q(\mu_d)}}$ \\ \hline
$5$ & $4$ & $32$ & $1024$ &
	$16 \cdot 5 = 80$ & $125$ \\
$6$ & $2$ & $25$ & $625$ &
	$4 \cdot 3 \cdot 25 = 300$ & $3$ \\
$7$ & $6$ & $432$ & $186624$ &
	$8^2 \cdot 3^6 \cdot 7 = 326592$ & $16807$ \\
$8$ & $4$ & $49$ & $2401$ &
	$2 \cdot 49^2 = 4802$ & $256$
\end{tabular}
\end{center}
For example, $2$ is inert in $\Q(\mu_5)$, while $5$
is totally ramified. In $\Q(\sqrt{-3})$, the primes
$2$ and $5$ are inert and $3$ is ramified. In
$\Q(\mu_7)$, the residue degrees of $2$ and $3$ are
$3$ and $6$. In $\Q(\mu_8)$, the prime $2$ is
totally ramified and $7$ has two primes of residue
degree $2$. These facts give the displayed values of
$N$.

Substitution in
\eqref{eq:appendix-general-point-bound} gives
\[
\begin{array}{c|c}
	d & \log_{10}\paren{B_{\Per}(\Q,d)} \\ \hline
	5 & < 9.984 \cdot 10^{24} \\
	6 & < 1.575 \cdot 10^{90} \\
	7 & < 1.690 \cdot 10^{246} \\
	8 & < 1.522 \cdot 10^{547}.
\end{array}
\]
In particular, the general argument gives
\begin{align*}
	B_{\Per}(\Q,5) &< 10^{\,10^{25}}, \\
	B_{\Per}(\Q,6) &< 10^{\,2 \cdot 10^{90}}, \\
	B_{\Per}(\Q,7) &< 10^{\,2 \cdot 10^{246}}, \\
	B_{\Per}(\Q,8) &< 10^{\,2 \cdot 10^{547}}.
\end{align*}
The factor $2^{8g_d^2}$ in
Yu--Yuan--Zhou
\cite[Theorem~1.6]{yu-yuan-zhou}
accounts for almost all of this growth.
These bounds should be very far from being sharp,
but they are remarkable for being both
uniform and explicit.

\subsection{Examples for quadratic \texorpdfstring{$K$}{K}}

Suppose that $K$ is a subfield of $\Q(\mu_d)$. Hence
$K(\mu_d) = \Q(\mu_d)$, and consequently
$\calc_{K, d}$ and the family indexed by
$\calc_{K, d}^2$ are unchanged. This gives the following
table.
\begin{center}
\small
\begin{tabular}{c|c|c|c}
$K$ & $d$ & $\paren{\#\calc_{K, d}}^2$ &
bound from the general argument \\ \hline
$\Q(\sqrt 5)$ & $5$ & $1024$ & $10^{\,10^{25}}$ \\
$\Q(\sqrt{-3})$ & $6$ & $625$ &
	$10^{\,2 \cdot 10^{90}}$ \\
$\Q(\sqrt{-7})$ & $7$ & $186624$ &
	$10^{\,2 \cdot 10^{246}}$ \\
$\Q(i)$ & $8$ & $2401$ &
	$10^{\,2 \cdot 10^{547}}$ \\
$\Q(\sqrt 2)$ & $8$ & $2401$ &
	$10^{\,2 \cdot 10^{547}}$
\end{tabular}
\end{center}
The rounded exponents do not change, since the finite
number of pairs of $\calc_{K, d}$ is negligible on this scale.

These examples confirm that the bounds
in the main theorems are computable,
and also indicate that the numerical growth
primarily comes from the general rational point bounds
of Yu--Yuan--Zhou \cite[Theorem~1.6]{yu-yuan-zhou}
rather than from the
twisted Fermat curve indices in $\calc_{K, d}$.


\bibliographystyle{alpha}
\bibliography{bibliography-UBC}

\end{document}